\documentclass[10pt,reqno]{amsart}

\usepackage[utf8]{inputenc}
\usepackage[T1]{fontenc}
\usepackage[english]{babel}
\usepackage[margin=1.2in]{geometry}
\usepackage{amssymb,amsmath,amsthm,mathrsfs}
\usepackage[dvipsnames]{xcolor}
\usepackage[numbers]{natbib}
\usepackage{enumitem}
\usepackage{mathtools}
\usepackage{bm}
\usepackage{orcidlink}
\usepackage{hyperref}

\hypersetup{
  colorlinks=true,
  linkcolor=NavyBlue,
  citecolor=NavyBlue,
  urlcolor=NavyBlue
}

\theoremstyle{plain}
\newtheorem{theorem}{Theorem}[section]
\newtheorem{proposition}[theorem]{Proposition}
\newtheorem{lemma}[theorem]{Lemma}
\newtheorem{corollary}[theorem]{Corollary}

\theoremstyle{definition}
\newtheorem{definition}[theorem]{Definition}
\newtheorem{remark}[theorem]{Remark}
\newtheorem{example}[theorem]{Example}

\theoremstyle{plain}
\newtheorem{conjecture}[theorem]{Conjecture}

\newcommand{\Q}{\mathbb{Q}}
\newcommand{\Z}{\mathbb{Z}}
\newcommand{\rad}{\operatorname{rad}}
\newcommand{\sq}{\operatorname{sq}}
\newcommand{\ord}{\operatorname{ord}}
\newcommand{\lin}{\mathrm{lin}}
\newcommand{\tot}{\mathrm{tot}}

\begin{document}

\title[Defect identities for the $abc$ conjecture]{Radical defects, Wieferich primes, and the $abc$ conjecture}
\author[R. Laniewski]{R. Laniewski}
\date{\today}



\begin{abstract}
For coprime $a+b=c$, the parity class in which $a$ and $b$ are odd and $c$ is even admits an exact treatment. Triples are measured by the radical excess
$E_{\varepsilon}=\log c-(1+\varepsilon)\log\operatorname{rad}(abc)$,
and called transgressive at fixed $\varepsilon>0$ when it is non-negative. The $abc$ conjecture asserts that such triples are finite in number.
The excess is written exactly in terms of the defect, the mass of repeated primes the radical discards, and rearranges into a linear threshold in which the smaller summand $s=\min\{a,b\}$ appears explicitly. Either $s$ stays bounded along a subsequence, or the defect must overshoot the threshold and rejoin it at an amplified exponent.

For the Mersenne family $\mathcal{V}_m = (1, 2^m-1, 2^m)$ transgression at $\varepsilon=0$ holds precisely when $2^m-1$ fails to be squarefree, on a set of exponents of density $47/210$. The defect obeys the exact law $\Delta_m = \Omega_m + \log G_m$, where $\Omega_m$ collects the Wieferich primes dividing $2^m-1$ and $G_m$ is the largest divisor of $m$ whose prime factors divide $2^m-1$. This divisor is trivial on prime-power exponents and largest on $m_k = \operatorname{lcm}(1,\dots,k)$, where the margin $q(\mathcal{V}_{m_k}) - 1 \geq (1-o(1))\log m_k / (m_k\log 2)$ is the most any Wieferich-free mechanism can give. An $abc$ counterexample family along the Mersenne line would require infinitely many Wieferich primes with exponential order--defect growth.

A final section separates two elliptic curves attached to the class: the Frey curve, whose minimal discriminant expresses the total defect with a bounded correction at $2$, and a congruent-number Jacobian, whose Szpiro quotient stays below $3$.
\end{abstract}

\maketitle

\tableofcontents

\section{Introduction}\label{sec:introduction}
Let \(a+b=c\) with \(a,b,c\in\Z_{>0}\) pairwise coprime. Such a triple is subject at once to an additive constraint and a multiplicative one, and the two are not readily compared. The \(abc\) conjecture of Masser and Oesterl\'e~\cite{oesterle1988abc} states how weakly they interact: for every \(\varepsilon>0\), only finitely many coprime triples satisfy
\[
  c\geq\rad(abc)^{1+\varepsilon}.
\]
Exactly one class among the coprime solutions has \(c\) even. Writing \(c=2K\) and \(a,b=K\mp M\) turns the additive relation into the identity \(ab=K^{2}-M^{2}\), and the odd parts of \(K\), \(K-M\) and \(K+M\) are pairwise coprime. The radical then separates into three independent pieces and the excess above the \(abc\) threshold can be computed rather than estimated.

A coprime triple \((a,b,c)\) satisfying \(a+b=c\) is called \emph{transgressive at exponent \(\varepsilon\)} if \(c \geq \rad(abc)^{1+\varepsilon}\), or equivalently if the quality \(q(a,b,c) = \log c / \log\rad(abc)\) satisfies \(q(a,b,c) \geq 1+\varepsilon\).
The \(abc\) conjecture asserts that for every fixed \(\varepsilon>0\), the set of transgressive triples is finite.

This paper asks under which arithmetic conditions an explicit one-parameter family can contain infinitely many transgressive triples.
For the family of Mersenne triples
\[
  \mathcal{V}_m = (1, 2^m-1, 2^m)
\]
the condition can be written as an exact criterion involving the order and defect structure of Wieferich primes.

\subsection{The Mersenne criterion}\label{subsec:intro-mersenne}

Write \(\theta_{\varepsilon}=\varepsilon/(1+\varepsilon)\).
Theorem~\ref{thm:mersenne-hits} answers a weak form of the question.
Its triples \(\mathcal{V}_{6n}\) all have quality greater than one.
The Mersenne line thus carries infinitely many triples above the critical quality, though with an excess that decays away.
Theorem~\ref{thm:mersenne-squarefree-classification} gives a complete description.
The quality of \(\mathcal{V}_m\) exceeds one exactly when \(2^m-1\) fails to be squarefree, and the indices at which this happens have lower density at least \(47/210\) by Corollary~\ref{cor:mersenne-hit-density}.
For the size of the excess, Theorem~\ref{thm:lcm-margin} gives, on the exponents \(m_k=\operatorname{lcm}(1,\dots,k)\),
\[
  q(\mathcal{V}_{m_k})-1
  \geq
  (1-o(1))\,\frac{\log m_k}{m_k\log2}.
\]
The growth estimate used in the proof is Nair's elementary lower bound \(\operatorname{lcm}(1,\dots,k) \geq 2^{k-1}\)~\cite{Nair1982}.
Theorem~\ref{thm:extremal-lcm} shows the margin to be extremal.
The part of the Mersenne defect free of Wieferich primes never exceeds \(\log m\).

The question of whether the line carries infinitely many triples transgressive at one fixed exponent lies deeper.
Section~\ref{sec:mersenne} shows this to hold precisely when
\[
  \frac{2^m-1}{\rad(2^m-1)}
  \geq
  2(1-2^{-m})\,2^{\theta_{\varepsilon}m}
\]
for infinitely many integers \(m\).
The left side is the repeated-prime mass of \(2^m-1\), and the criterion asks that this mass grow at a fixed exponential rate in \(m\).

By the lifting-the-exponent identity, the contribution of a fixed prime to the defect of \(2^m-1\) is logarithmic in \(m\) unless the prime is a \emph{Wieferich prime}, that is, a prime \(p\) with \(2^{p-1}\equiv1\pmod{p^2}\).
Theorem~\ref{thm:exact-defect-law} in Section~\ref{sec:wieferich} gives
\[
  \Delta_m=\Omega_m+\log G_m.
\]
Here \(\Omega_m\) is the intrinsic contribution of the Wieferich primes dividing \(2^m-1\), and \(G_m\) is the largest divisor of \(m\) all of whose prime factors divide \(2^m-1\).
On prime-power exponents \(G_m=1\), so quality above one can occur on that subline only through a Wieferich prime of prime-power order, a prime at present undiscovered (Corollary~\ref{cor:prime-power-quality}).
The order-closed exponents behave differently.
Among them are the numbers \(\operatorname{lcm}(1,\dots,k)\), where \(G_m\) contains the whole odd part of \(m\) and yields the quantitative margin recalled above.
Every fixed even base \(g\) gives a line of parity-class triples \((1,g^m-1,g^m)\) on which the same identity holds, by Theorem~\ref{thm:even-base-exact-law}.

For a Wieferich prime \(p\), let \(d_p=\ord_p(2)\) and \(A_p=v_p(2^{d_p}-1)\).
For a finite set \(S\) of Wieferich primes put
\[
  D(S)=\operatorname{lcm}_{p\in S}d_p
  \qquad\text{and}\qquad
  B(S)=\prod_{p\in S}p^{A_p-1}.
\]
Theorems~\ref{thm:wieferich-density-criterion} and~\ref{thm:order-defect-density} then give the following equivalence.
An infinite Mersenne family transgressive at some fixed positive exponent exists precisely when there are a constant \(\kappa>0\) and finite sets \(S_j\) of Wieferich primes with
\[
  D(S_j)\longrightarrow\infty
  \qquad\text{and}\qquad
  B(S_j)\geq2^{\kappa D(S_j)}.
\]
Whether such growth occurs is not known.

Writing \(D_m=D(\mathcal{W}(m))\) for the order core of the exponent, Theorem~\ref{thm:bounded-order-cofactor} bounds the quotient \(m/D_m\) by the integers below \(\lceil1+1/\varepsilon\rceil\) at every sufficiently large transgressive exponent.
By Theorem~\ref{thm:order-core-descent}, any infinite transgressive family at a given exponent may be replaced by one for which that quotient equals one.
On prime exponents \(\ell\) the condition takes a simpler form, since every prime dividing \(2^{\ell}-1\) then has order exactly \(\ell\) (Theorem~\ref{thm:prime-subline}).

Theorem~\ref{thm:wieferich-finite} shows that the infinitude of the Wieferich primes is necessary for an \(abc\) counterexample family on the Mersenne line.
If \(\mathcal{W}\) is finite, the Mersenne defect is at most \(\log m+O(1)\) and the line carries only finitely many transgressive triples.
An explicit ceiling for the surviving indices is supplied by Theorem~\ref{thm:effective-wieferich-bound}.
Infinitude alone would not suffice.
The criterion asks the orders \(d_p\) to stay small against the accumulated weights \(A_p\), and the joint distribution of these two quantities has not been investigated.
The heuristic count of Wieferich primes below \(x\), of order \(\log\log x\), falls short of the \(\exp(\kappa\ell)\) growth in \(B(\mathcal{W}(\ell))\) that Theorem~\ref{thm:wieferich-density-criterion} would require.

\subsection{The class and the defect identities}\label{subsec:intro-frame}

The proofs proceed inside the parity class where \(a\) and \(b\) are odd and \(c\) is even.
Writing
\[
  K=\frac{a+b}{2}
  \qquad\text{and}\qquad
  M=\frac{b-a}{2}
\]
one has \(a=K-M\), \(b=K+M\), and \(c=2K\).
The pair \((a,b)\) corresponds bijectively to a pair of coprime integers \((K,M)\) of opposite parity with \(K>|M|\geq0\).
The advantage of the class is that the identity \(ab=K^2-M^2\) holds in \(\Z\), that the odd prime supports of \(K\), \(K-M\), and \(K+M\) are pairwise disjoint, and that the height is exactly \(\log(2K)\), so that the radical factors into three pairwise coprime parts, and the estimates below rest upon that factorization.

These three properties are all that the identities of Sections~\ref{sec:defects} and~\ref{sec:boundary-free} require. Proposition~\ref{prop:exact-energy} expresses the excess through the defect, Proposition~\ref{prop:boundary-identity} evaluates what remains through the smaller summand \(s=\min\{a,b\}\), and Theorem~\ref{thm:exact-linear-threshold} reads the result as a threshold upon the linear defect. The two regimes announced above are the two ways in which a sequence may meet that threshold, and Theorem~\ref{thm:boundary-interior-dichotomy} separates them.

A transgressive triple is constrained locally as well. Its total defect concentrates on a single prime within one summand (Theorem~\ref{thm:linear-concentration}), and at least two distinct primes must appear with multiplicity at least two in \(abc\) (Theorem~\ref{thm:two-repeated-primes}). Along any transgressive sequence, the squarefull parts, the relative quotients, and the total defects diverge at explicit rates (Theorem~\ref{thm:counterexample-portrait}). The kernel congruence modulo \(8\) is the complete local obstruction at \(2\) for the residual quadratic equation (Proposition~\ref{prop:two-adic-complete}), and a positive proportion of all kernel pairs satisfies it (Proposition~\ref{prop:kernel-classes}).

Subsection~\ref{subsec:recurrences} presents two families that approach these thresholds to within an absolute constant.
A Pell recurrence and a dyadic norm recurrence each produce an infinite family of parity-class triples whose total defect lies within an absolute constant of the principal threshold.
The constraints of Section~\ref{sec:boundary-free} are thus sharp up to bounded error.
On the Pell family the sign of the residual margin can be read off.
Its quality exceeds one exactly when the product \(x_ny_n\) of the two Pell coordinates fails to be squarefree (Theorem~\ref{thm:pell-infinite-quality}), and a divisibility that propagates along the recurrence gives a progression of indices, of density \(1/7\), on which this comes to pass.

\subsection{The curves attached to the class}\label{subsec:intro-descent}

The class carries its own elliptic curves, and they differ from the Frey curve in the treatment of the prime \(2\) and in the rigidity of the family obtained. The construction of Frey attaches to a coprime triple the curve \(y^2=x(x-a)(x+b)\), after a choice of ordering of the summands governed by a congruence modulo \(4\), and the passage from the \(abc\) conjecture to the conjecture of Szpiro~\cite{Szpiro1990} is then made through inequalities in which the power of \(2\) dividing the conductor is carried as a bounded but unspecified factor. Definition~\ref{def:descent-frey} fixes the orientation of the pair \((\alpha,\beta)\) by the two-adic valuation of \(c\), a choice the parity class makes available because \(c\) is the even member. Proposition~\ref{prop:descent-minimal} determines the minimal model at \(2\) for each of the three ranges \(1\leq v_2(c)\leq3\), \(v_2(c)=4\), and \(v_2(c)\geq5\), together with the reduction type in each, and Corollary~\ref{cor:descent-szpiro} turns the Szpiro quotient into the identity
\[
  \log\lvert\Delta_{\min}(E_{\mathcal{T}})\rvert
  =
  2\delta_{\tot}(\mathcal{T})
  +2\log\rad(abc)
  +\kappa_2(\mathcal{T})\log2
\]
where \(\kappa_2\) takes only the two values \(4\) and \(-8\). The total defect appears in the minimal discriminant as an equality, with a two-adic correction of at most \(8\log2\), and the arithmetic of Sections~\ref{sec:defects} and~\ref{sec:wieferich} may therefore be read upon the curve itself. On the Mersenne line every term becomes explicit. Corollary~\ref{cor:descent-mersenne} expresses the Szpiro quotient of \(E_{\mathcal{V}_m}\) through \(\Omega_m\) and \(\log G_m\), so that the Wieferich decomposition of Theorem~\ref{thm:exact-defect-law} is recovered from the curve without loss.

The genus-one intersection of Definition~\ref{def:descent-curve} answers a separate concern. A descent conducted upon a Frey curve opens with the question of local solubility, whereas the two quadrics assembled from the squarefree kernels \(d_a\) and \(d_b\) carry the distinguished rational point \(P_{\mathcal{T}}=(u_a:u_b:1:1)\), supplied by the square extraction and demanding no local--global argument (Corollary~\ref{cor:descent-witnessed-local}). Transgression then constrains two quantities in opposite directions. The height \(H_\times(P_{\mathcal{T}})=u_au_b\) is bounded below by \(\bigl(s(\mathcal{T})R_{\mathcal K}(\mathcal{T})\bigr)^{1/2}(2K)^{\theta_\varepsilon/2}\), while the product \(d_ad_b\) of the kernels is bounded above by \(\bigl(2K-s(\mathcal{T})\bigr)\big/R_{\mathcal K}(\mathcal{T})(2K)^{\theta_\varepsilon}\), the same factor \((2K)^{\theta_\varepsilon}\) governing both (Proposition~\ref{prop:descent-height} and Theorem~\ref{thm:descent-kernel-compression}). The Jacobian retains none of this, since it depends only upon the square class of \(\Gamma_{\mathcal{T}}=d_ad_bKM\), and its Szpiro quotient stays below \(3\) for every triple of the class (Corollary~\ref{cor:descent-jacobian-szpiro}). Section~\ref{sec:descent} therefore treats the two constructions separately.

A second difference concerns rigidity. A Frey curve varies with its triple, its \(j\)-invariant taking a fresh value each time and its reduction at the odd primes being of type \(I_{2v_p(abc)}\), so that the family ranges over the moduli of elliptic curves as the triple varies. Theorem~\ref{thm:descent-jacobian} places \(J_{\mathcal{T}}\) instead within a single isomorphism class up to quadratic twist, that of the congruent-number curve \(Y^2=X^3-\gamma_{\mathcal{T}}^2X\), with constant \(j\)-invariant \(1728\), full rational \(2\)-torsion, and complex multiplication by \(\Z[i]\) over \(\Q(i)\). Proposition~\ref{prop:descent-jacobian-conductor} gives the conductor as \(2^5\gamma_{\mathcal{T}}^2\) for \(\gamma_{\mathcal{T}}\) odd and \(2^4\gamma_{\mathcal{T}}^2\) for \(\gamma_{\mathcal{T}}\) even, with no case left undetermined. Every parity-class triple therefore yields a member of one quadratic-twist family, the family of the congruent-number problem, whose \(L\)-functions, Heegner points, and Selmer groups carry an established literature. A generic Frey curve admits no such reference class. The Jacobian retains only the single parameter \(\gamma_{\mathcal{T}}\), which Theorem~\ref{thm:descent-kernel-compression} bounds from above whenever the triple is transgressive.

\subsection{Position within the literature}\label{subsec:intro-literature}

The strongest unconditional pointwise estimates remain logarithmic in \(c\), whereas the conjecture asserts a bound polynomial in \(\rad(abc)\).
Stewart and Yu proved that only finitely many triples satisfy \(\rad(abc)<(\log c)^{3-\eta}\) for a fixed \(\eta>0\)~\cite{StewartYu2001}.
Pasten obtained sharper subexponential estimates under additional restrictions, among them the hypothesis that one summand be bounded above by a fixed power of \(c\) smaller than \(c\) itself~\cite{Pasten2024}.
In the opposite direction Bright proved that infinitely many coprime triples satisfy
\[
  c>
  \rad(abc)
  \exp\left(
    6.563\sqrt{\frac{\log c}{\log\log c}}
  \right)
\]
which refines the earlier constants of Stewart and Tijdeman and of van Frankenhuysen~\cite{StewartTijdeman1986,vanFrankenhuysen2000}.
Examples of this kind fix a lower-bound scale for whatever strength a refinement of the conjecture might hope to possess~\cite{Bright2024}.

Silverman proved that the \(abc\) conjecture implies the existence of infinitely many non-Wieferich primes~\cite{Silverman1988}.
In the reverse direction it has long been known that a single Wieferich prime forces low-radical triples along a Mersenne sub-line, as in the family of Granville and Tucker~\cite{GranvilleTucker2002}.
On the line \((1,2^m-1,2^m)\), the present paper establishes an exact identity for the radical defect (Theorem~\ref{thm:exact-defect-law}), and an equivalence that ties an infinite family transgressive at a fixed exponent to the order--defect growth of the Wieferich primes (Theorems~\ref{thm:wieferich-density-criterion} and~\ref{thm:order-defect-density}).

Progress has also been made in counting the triples whose radical is too small.
For fixed \(0<\lambda<1\), let \(N_{\lambda}(X)\) count the coprime triples in \([1,X]^3\) satisfying \(a+b=c\) and \(\rad(abc)<c^\lambda\).
An argument resting upon work of de Bruijn gives
\[
  N_{\lambda}(X)
  \ll_{\lambda,\eta}
  X^{2\lambda/3+\eta}
\]
for every \(\eta>0\)~\cite{deBruijn1962}.
The first power saving over this bound is due to Bernert, Browning, Lichtman, and Ter\"av\"ainen, who obtained, for every fixed \(0<\lambda\leq2\) and every \(\eta>0\),
\[
  N_{\lambda}(X)
  \ll_{\lambda,\eta}
  X^{(23\lambda+3)/40+\eta}
\]
and, sharper as \(\lambda\to1\), the estimate
\[
  N_{\lambda}(X)
  \ll_{\lambda,\eta}
  X^{3/5+\eta}
\]
for every fixed \(0<\lambda<1\) and every \(\eta>0\)~\cite[Theorems~1.2 and~1.3]{BernertEtAl2026}.
Lichtman has since shown the conjecture to hold outside a sparse exceptional set~\cite{Lichtman2025}.
The sparsity falls short of the finiteness predicted by the conjecture.
Section~\ref{sec:benchmarks} sets these estimates beside the pointwise constraints obtained here.

\section{The parity class}\label{sec:frame}

\begin{definition}[Parity-class triple]\label{def:parity-triple}
A \emph{parity-class triple} is a triple \(\mathcal{T}=(a,b,c)\) of pairwise coprime positive integers satisfying
\[
  a+b=c
\]
where \(a\) and \(b\) are odd and \(c\) is even.
\end{definition}

Every pairwise coprime solution of \(a+b=c\) has exactly one even member.
Definition~\ref{def:parity-triple} keeps the class in which that member is \(c\), and the class in which the even member is \(a\) or \(b\) is not considered here.

\begin{definition}[Class coordinates]\label{def:frame-coordinates}
For a parity-class triple \(\mathcal{T}\), define
\[
  K=K(\mathcal{T})
  \coloneqq
  \frac{a+b}{2}
  \qquad\text{and}\qquad
  M=M(\mathcal{T})
  \coloneqq
  \frac{b-a}{2}.
\]
Then
\[
  a=K-M
  \qquad
  b=K+M
  \qquad
  c=2K.
\]
\end{definition}

\begin{lemma}[Coprimality of the class coordinates]\label{lem:KM-coprime}
For every parity-class triple one has
\[
  \gcd(K,M)=1.
\]
\end{lemma}

\begin{proof}
Every common divisor of \(K\) and \(M\) divides \(K-M=a\) and \(K+M=b\).
Since \(\gcd(a,b)=1\), the common divisor is \(1\).
\end{proof}

\begin{lemma}[Disjoint odd supports]\label{lem:disjoint-supports}
No odd prime divides two of the three integers
\[
  K,
  \qquad
  K-M,
  \qquad
  K+M.
\]
\end{lemma}

\begin{proof}
Suppose that an odd prime \(p\) divides two of these integers.
Their sum or difference then shows that \(p\) divides the third, and in particular \(p\) divides both \(K\) and \(M\).
This contradicts Lemma~\ref{lem:KM-coprime}.
\end{proof}

\begin{definition}[Odd radical of the center]\label{def:center-radical}
Define
\[
  R_{\mathcal{K}}(\mathcal{T})
  \coloneqq
  \prod_{\substack{p\mid K\\p\ \mathrm{odd}}}p
\]
and
\[
  t^{\mathcal{K}}(\mathcal{T})
  \coloneqq
  \log R_{\mathcal{K}}(\mathcal{T}).
\]
\end{definition}

\begin{proposition}[Exact radical separation]\label{prop:radical-separation}
For every parity-class triple,
\[
  \rad(abc)
  =
  2R_{\mathcal{K}}(\mathcal{T})\rad(a)\rad(b).
\]
It follows that
\[
  \log\rad(abc)
  =
  \log2+t^{\mathcal{K}}(\mathcal{T})
  +\log\rad(a)+\log\rad(b).
\]
\end{proposition}

\begin{proof}
The prime \(2\) divides \(c\) and does not divide \(a\) or \(b\).
Lemma~\ref{lem:disjoint-supports} shows that the odd primes dividing \(K\), \(a\), and \(b\) form three disjoint sets.
Multiplication over these sets gives the identity.
\end{proof}

\section{Linear and total defects}\label{sec:defects}

\begin{definition}[Radical excess and quality]\label{def:energy-quality}
For \(\varepsilon>0\), define the \emph{radical excess}
\[
  E_{\varepsilon}(\mathcal{T})
  \coloneqq
  \log c-(1+\varepsilon)\log\rad(abc).
\]
Define the \emph{quality}
\[
  q(\mathcal{T})
  \coloneqq
  \frac{\log c}{\log\rad(abc)}.
\]
Thus
\[
  E_{\varepsilon}(\mathcal{T})\geq0
  \quad\Longleftrightarrow\quad
  q(\mathcal{T})\geq1+\varepsilon.
\]
A triple satisfying these equivalent inequalities is called \emph{transgressive at exponent \(\varepsilon\)}.
\end{definition}

Since the conjecture forbids only the infinitude of such triples at a fixed exponent, a transgressive triple does not by itself refute it.

\begin{conjecture}[\(abc\) in the parity class]\label{conj:parity-abc}
For every \(\varepsilon>0\), only finitely many parity-class triples satisfy
\[
  E_{\varepsilon}(\mathcal{T})\geq0.
\]
\end{conjecture}

Failure of Conjecture~\ref{conj:parity-abc} would imply failure of the full \(abc\) conjecture.
The converse need not follow, since Conjecture~\ref{conj:parity-abc} concerns only the parity class in which \(c\) is even.

\begin{definition}[Linear defect]\label{def:linear-defect}
The \emph{linear defect} is
\[
  \delta_{\lin}(\mathcal{T})
  \coloneqq
  \log\left(
    \frac{ab}{\rad(a)\rad(b)}
  \right).
\]
Equivalently,
\[
  \delta_{\lin}(\mathcal{T})
  =
  \sum_{p\mid a}\bigl(v_p(a)-1\bigr)\log p
  +
  \sum_{p\mid b}\bigl(v_p(b)-1\bigr)\log p.
\]
\end{definition}

The linear defect is non-negative, and it vanishes precisely when \(a\) and \(b\) are squarefree.

\begin{proposition}[Exact radical-excess identity]\label{prop:exact-energy}
Put
\[
  r(\mathcal{T})
  \coloneqq
  \frac{|M|}{K}.
\]
Then \(0\leq r(\mathcal{T})<1\), and
\[
\begin{aligned}
  E_{\varepsilon}(\mathcal{T})
  ={}&
  -(1+2\varepsilon)\log K
  -(1+\varepsilon)t^{\mathcal{K}}(\mathcal{T})
  +(1+\varepsilon)\delta_{\lin}(\mathcal{T})
  \\[0.3em]
  &{}-(1+\varepsilon)\log\bigl(1-r(\mathcal{T})^2\bigr)
  -\varepsilon\log2.
\end{aligned}
\]
\end{proposition}

\begin{proof}
Since
\[
  ab=(K-M)(K+M)=K^2\bigl(1-r(\mathcal{T})^2\bigr)
\]
Definition~\ref{def:linear-defect} gives
\[
  \log\rad(a)+\log\rad(b)
  =
  2\log K+\log\bigl(1-r(\mathcal{T})^2\bigr)-\delta_{\lin}(\mathcal{T}).
\]
Substitution into Proposition~\ref{prop:radical-separation} and then into Definition~\ref{def:energy-quality} proves the identity.
\end{proof}

The term involving \(r(\mathcal{T})\) measures the asymmetry between \(a\) and \(b\), and it is of order \(\log(c/s(\mathcal{T}))\) once one summand is small against the other.
In Section~\ref{sec:boundary-free} this term is evaluated exactly through the smaller summand, and the identity then becomes a threshold.

\begin{definition}[Center defect and total defect]\label{def:total-defect}
Define the odd center defect by
\[
  \delta^{\mathcal{K}}(\mathcal{T})
  \coloneqq
  \sum_{\substack{p\mid K\\p\ \mathrm{odd}}}
  \bigl(v_p(K)-1\bigr)\log p.
\]
Define the two-adic defect by
\[
  \delta_2(\mathcal{T})
  \coloneqq
  v_2(K)\log2.
\]
Since \(c=2K\) gives \(v_2(c)=v_2(K)+1\), an equivalent form is
\[
  \delta_2(\mathcal{T})
  =
  \bigl(v_2(c)-1\bigr)\log2.
\]
This exhibits \(\delta_2\) as the two-adic multiplicity discarded by the radical in \(c\), the single factor of \(2\) retained by \(\rad(c)\) accounting for the subtracted unit.
The \emph{total defect} is
\[
  \delta_{\tot}(\mathcal{T})
  \coloneqq
  \delta_2(\mathcal{T})
  +
  \delta^{\mathcal{K}}(\mathcal{T})
  +
  \delta_{\lin}(\mathcal{T}).
\]
\end{definition}

\begin{lemma}[Total defect identity]\label{lem:total-defect-identity}
For every parity-class triple,
\[
  \delta_{\tot}(\mathcal{T})
  =
  \log\left(
    \frac{abc}{\rad(abc)}
  \right).
\]
\end{lemma}

\begin{proof}
The three summands of Definition~\ref{def:total-defect} are the logarithmic multiplicities discarded by the radical in \(c\), in the odd part of \(K\), and in the two odd summands.
Lemma~\ref{lem:disjoint-supports} permits their addition without overlap.
\end{proof}

\begin{proposition}[Total radical-excess identity]\label{prop:total-energy}
For every \(\varepsilon>0\),
\[
  E_{\varepsilon}(\mathcal{T})
  =
  (1+\varepsilon)
  \bigl(
    \delta_{\tot}(\mathcal{T})-\log(ab)
  \bigr)
  -
  \varepsilon\log(2K).
\]
In particular,
\[
  \log c-\log\rad(abc)
  =
  \delta_{\tot}(\mathcal{T})-\log(ab).
\]
\end{proposition}

\begin{proof}
Lemma~\ref{lem:total-defect-identity} gives
\[
  \log\rad(abc)
  =
  \log(ab)+\log(2K)-\delta_{\tot}(\mathcal{T}).
\]
Substitution into the radical excess proves the first identity.
The second follows from the same algebraic identity at \(\varepsilon=0\).
\end{proof}

\begin{theorem}[Exact total-defect criterion]\label{thm:total-criterion}
Put
\[
  \theta_{\varepsilon}
  \coloneqq
  \frac{\varepsilon}{1+\varepsilon}.
\]
Then
\[
  E_{\varepsilon}(\mathcal{T})<0
  \quad\Longleftrightarrow\quad
  \delta_{\tot}(\mathcal{T})
  <
  \log(ab)+\theta_{\varepsilon}\log(2K).
\]
Hence if \(E_{\varepsilon}(\mathcal{T})\geq0\), then
\[
  \delta_{\tot}(\mathcal{T})
  \geq
  \log(ab)+\theta_{\varepsilon}\log(2K).
\]
\end{theorem}

\begin{corollary}[Quality through the total defect]\label{cor:quality-total}
For every parity-class triple,
\[
  q(\mathcal{T})
  =
  1+
  \frac{
    \delta_{\tot}(\mathcal{T})-\log(ab)
  }{
    \log\rad(abc)
  }.
\]
\end{corollary}

\section{Threshold certificates and concentration}\label{sec:boundary-free}

\begin{lemma}[Minimal product]\label{lem:minimal-product}
Suppose that \(K>1\).
Then
\[
  ab\geq2K-1.
\]
Equivalently,
\[
  1-r(\mathcal{T})^2
  \geq
  \frac{2K-1}{K^2}.
\]
\end{lemma}

\begin{proof}
The positive odd integers \(a\) and \(b\) have sum \(2K\).
Their product is smallest when the pair is \(1\) and \(2K-1\).
The second inequality follows from \(ab=K^2(1-r(\mathcal{T})^2)\).
\end{proof}

\begin{proposition}[Exact boundary identity]\label{prop:boundary-identity}
Let
\[
  s(\mathcal{T})
  \coloneqq
  \min\{a,b\}.
\]
Then, for every parity-class triple,
\[
\begin{aligned}
  E_{\varepsilon}(\mathcal{T})
  ={}&
  (1+\varepsilon)
  \bigl(
    \delta_{\lin}(\mathcal{T})
    -
    t^{\mathcal{K}}(\mathcal{T})
  \bigr)
  +
  \log K
  \\
  &-
  (1+\varepsilon)
  \log\bigl(
    s(\mathcal{T})\,(2K-s(\mathcal{T}))
  \bigr)
  -
  \varepsilon\log2.
\end{aligned}
\]
\end{proposition}

\begin{proof}
The larger summand equals \(2K-s(\mathcal{T})\), so that
\[
  ab
  =
  s(\mathcal{T})\bigl(2K-s(\mathcal{T})\bigr)
  =
  K^2\bigl(1-r(\mathcal{T})^2\bigr).
\]
Substitute this expression into the identity of Proposition~\ref{prop:exact-energy}.
The coefficient of \(\log K\) becomes \(-(1+2\varepsilon)+2(1+\varepsilon)=1\).
\end{proof}

\begin{theorem}[Exact linear threshold]\label{thm:exact-linear-threshold}
Suppose that \(K>1\) and put \(s=s(\mathcal{T})\).
Then
\[
  E_{\varepsilon}(\mathcal{T})\geq0
  \quad\Longleftrightarrow\quad
  \delta_{\lin}(\mathcal{T})
  \geq
  t^{\mathcal{K}}(\mathcal{T})
  +
  \theta_{\varepsilon}\log(2K)
  +
  \log\Bigl(s\Bigl(2-\frac{s}{K}\Bigr)\Bigr).
\]
The boundary term satisfies
\[
  \log\Bigl(2-\frac{1}{K}\Bigr)
  \leq
  \log\Bigl(s\Bigl(2-\frac{s}{K}\Bigr)\Bigr)
  \leq
  \log(2s)
\]
and it is an increasing function of \(s\) on \([1,K]\), so that its least value occurs on the lines \(s=1\).
\end{theorem}

\begin{proof}
Division of the identity of Proposition~\ref{prop:boundary-identity} by \(1+\varepsilon\) shows that \(E_{\varepsilon}(\mathcal{T})\geq0\) holds precisely when
\[
  \delta_{\lin}(\mathcal{T})
  -
  t^{\mathcal{K}}(\mathcal{T})
  \geq
  \log\bigl(s(2K-s)\bigr)
  -
  \frac{\log K}{1+\varepsilon}
  +
  \theta_{\varepsilon}\log2.
\]
Since \(2K-s=K(2-s/K)\),
\[
  \log\bigl(s(2K-s)\bigr)
  =
  \log K+\log s+\log\Bigl(2-\frac{s}{K}\Bigr)
\]
and the right side equals \(\theta_{\varepsilon}\log(2K)+\log(s(2-s/K))\).

For the range of the boundary term, put \(g(x)=x(2-x/K)\) on \([1,K]\).
Then \(g'(x)=2-2x/K\geq0\), so that \(g\) increases, with \(g(1)=2-1/K\) at one end and \(g(s)\leq2s\) throughout.
\end{proof}

When \(s\) is comparable to \(K\), the boundary term \(\log(s(2-s/K))\) in Theorem~\ref{thm:exact-linear-threshold} approaches \(\log K\). The threshold in Theorem~\ref{thm:exact-linear-threshold} then exceeds the boundary-free lower bound of Theorem~\ref{thm:boundary-free-certificate} by approximately \(\log K\), leading to the interior divergence described in Theorem~\ref{thm:boundary-interior-dichotomy} and the exponent amplification of Proposition~\ref{prop:exponent-amplification}.

\begin{theorem}[Boundary-free certificate]\label{thm:boundary-free-certificate}
Suppose that \(K>1\).
If
\[
  \delta_{\lin}(\mathcal{T})
  <
  t^{\mathcal{K}}(\mathcal{T})
  +
  \theta_{\varepsilon}\log(2K)
\]
then
\[
  E_{\varepsilon}(\mathcal{T})<0.
\]
Conversely, if \(E_{\varepsilon}(\mathcal{T})\geq0\), then
\[
  \delta_{\lin}(\mathcal{T})
  \geq
  t^{\mathcal{K}}(\mathcal{T})
  +
  \theta_{\varepsilon}\log(2K).
\]
\end{theorem}

\begin{proof}
For \(K>1\) the boundary term of Theorem~\ref{thm:exact-linear-threshold} is at least \(\log(2-1/K)>0\).
If \(\delta_{\lin}(\mathcal{T})<t^{\mathcal{K}}(\mathcal{T})+\theta_{\varepsilon}\log(2K)\), then the threshold of that theorem fails, and \(E_{\varepsilon}(\mathcal{T})<0\).
The second assertion is the contrapositive.
\end{proof}

\begin{corollary}[Exact criterion on a fixed line]\label{cor:line-criterion}
Fix an odd integer \(s_0\geq1\), and consider the parity-class triples with \(\min\{a,b\}=s_0\), that is, the triples \((s_0,2K-s_0,2K)\) with \(K>s_0\) and \(\gcd(s_0,2K)=1\).
On this line,
\[
  E_{\varepsilon}(\mathcal{T})\geq0
  \quad\Longleftrightarrow\quad
  \delta_{\lin}(\mathcal{T})
  \geq
  t^{\mathcal{K}}(\mathcal{T})
  +
  \theta_{\varepsilon}\log(2K)
  +
  \log\Bigl(s_0\Bigl(2-\frac{s_0}{K}\Bigr)\Bigr).
\]
Moreover, every transgressive triple on the line satisfies
\[
  \log\left(\frac{2K-s_0}{\rad(2K-s_0)}\right)
  \geq
  t^{\mathcal{K}}(\mathcal{T})
  +
  \theta_{\varepsilon}\log(2K)
  +
  \log\rad(s_0).
\]
\end{corollary}

\begin{proof}
For \(K>s_0\) the smaller summand is \(s_0\), and the first equivalence is Theorem~\ref{thm:exact-linear-threshold}.
For the second assertion, the linear defect on the line decomposes as
\[
  \delta_{\lin}(\mathcal{T})
  =
  \log\left(\frac{s_0}{\rad(s_0)}\right)
  +
  \log\left(\frac{2K-s_0}{\rad(2K-s_0)}\right).
\]
The threshold gives
\[
  \log\left(\frac{2K-s_0}{\rad(2K-s_0)}\right)
  \geq
  t^{\mathcal{K}}(\mathcal{T})
  +
  \theta_{\varepsilon}\log(2K)
  +
  \log s_0
  -
  \log\left(\frac{s_0}{\rad(s_0)}\right)
  +
  \log\Bigl(2-\frac{s_0}{K}\Bigr)
\]
and the claim follows from \(\log s_0-\log(s_0/\rad(s_0))=\log\rad(s_0)\) together with \(\log(2-s_0/K)>0\).
\end{proof}

The boundary-free certificate of Theorem~\ref{thm:boundary-free-certificate} differs from the exact threshold in Theorem~\ref{thm:exact-linear-threshold} solely by the non-negative term \(\log(s(2-s/K))\). On the boundary lines \(s=1\), this term is bounded between \(\log(2-1/K)\) and \(\log 2\), whereas for interior sequences with \(s \to \infty\) it diverges as \(\log s\).
Section~\ref{sec:mersenne} examines the case \(s=1\), where Corollary~\ref{cor:line-criterion} yields an explicit single-variable criterion on the Mersenne line.

\begin{definition}[Relative linear quotient]\label{def:relative-quotient}
Define
\[
  D_{\lin}(\mathcal{T})
  \coloneqq
  \frac{ab}{\rad(a)\rad(b)}
\]
and
\[
  Q_{\mathrm{rel}}(\mathcal{T})
  \coloneqq
  \frac{D_{\lin}(\mathcal{T})}{R_{\mathcal{K}}(\mathcal{T})}.
\]
Then
\[
  \log Q_{\mathrm{rel}}(\mathcal{T})
  =
  \delta_{\lin}(\mathcal{T})-t^{\mathcal{K}}(\mathcal{T}).
\]
\end{definition}

\begin{corollary}[Quotient certificate]\label{cor:quotient-certificate}
Suppose that \(K>1\).
If
\[
  Q_{\mathrm{rel}}(\mathcal{T})
  <
  (2K)^{\theta_{\varepsilon}}
\]
then \(E_{\varepsilon}(\mathcal{T})<0\).
If \(E_{\varepsilon}(\mathcal{T})\geq0\), then
\[
  Q_{\mathrm{rel}}(\mathcal{T})
  \geq
  s(\mathcal{T})\,(2K)^{\theta_{\varepsilon}}
\]
and so
\[
  s(\mathcal{T})
  \leq
  \frac{Q_{\mathrm{rel}}(\mathcal{T})}{(2K)^{\theta_{\varepsilon}}}.
\]
\end{corollary}

\begin{proof}
The first assertion follows upon exponentiating Theorem~\ref{thm:boundary-free-certificate}.
For the second, Theorem~\ref{thm:exact-linear-threshold} gives
\[
  \delta_{\lin}(\mathcal{T})-t^{\mathcal{K}}(\mathcal{T})
  \geq
  \theta_{\varepsilon}\log(2K)
  +
  \log s(\mathcal{T})
  +
  \log\Bigl(2-\frac{s(\mathcal{T})}{K}\Bigr)
  \geq
  \theta_{\varepsilon}\log(2K)
  +
  \log s(\mathcal{T})
\]
since \(2-s(\mathcal{T})/K\geq1\).
Exponentiation proves the bound, and the final estimate is its rearrangement.
\end{proof}

\begin{proposition}[Exponent amplification in the interior]\label{prop:exponent-amplification}
Suppose that \(K>1\), that \(0\leq\alpha\leq1\), and that
\[
  s(\mathcal{T})\geq K^{\alpha}.
\]
If \(E_{\varepsilon}(\mathcal{T})\geq0\), then
\[
  \delta_{\lin}(\mathcal{T})
  \geq
  t^{\mathcal{K}}(\mathcal{T})
  +
  (\theta_{\varepsilon}+\alpha)\log(2K)
  -
  \log2.
\]
Suppose in addition that \(\theta_{\varepsilon}+\alpha<1\), and define \(\varepsilon'\) by
\[
  \theta_{\varepsilon'}=\theta_{\varepsilon}+\alpha
  \qquad\text{that is}\qquad
  \varepsilon'=\frac{\theta_{\varepsilon}+\alpha}{1-\theta_{\varepsilon}-\alpha}.
\]
Then the triple satisfies, within the additive constant \(\log2\), the necessary condition of Theorem~\ref{thm:boundary-free-certificate} at the amplified exponent \(\varepsilon'\).
\end{proposition}

\begin{proof}
Theorem~\ref{thm:exact-linear-threshold} together with \(2-s(\mathcal{T})/K\geq1\) gives
\[
  \delta_{\lin}(\mathcal{T})-t^{\mathcal{K}}(\mathcal{T})
  \geq
  \theta_{\varepsilon}\log(2K)+\alpha\log K
  =
  (\theta_{\varepsilon}+\alpha)\log(2K)-\alpha\log2
\]
and \(\alpha\log2\leq\log2\) proves the first claim.
The map \(\varepsilon\mapsto\theta_{\varepsilon}\) carries \((0,\infty)\) bijectively onto \((0,1)\), so the stated \(\varepsilon'\) exists, and the first claim then reads
\[
  \delta_{\lin}(\mathcal{T})
  \geq
  t^{\mathcal{K}}(\mathcal{T})
  +
  \theta_{\varepsilon'}\log(2K)
  -
  \log2
\]
which is the necessary condition of Theorem~\ref{thm:boundary-free-certificate} at exponent \(\varepsilon'\) within \(\log2\).
\end{proof}

\begin{corollary}[Squarefree summands]\label{cor:squarefree-summands}
Suppose that \(K>1\).
If \(a\) and \(b\) are squarefree, then
\[
  E_{\varepsilon}(\mathcal{T})<0
\]
for every \(\varepsilon>0\).
\end{corollary}

\begin{proof}
The squarefree hypothesis gives \(\delta_{\lin}(\mathcal{T})=0\), and \(t^{\mathcal{K}}(\mathcal{T})\geq0\).
Theorem~\ref{thm:boundary-free-certificate} applies.
\end{proof}

\subsection{Square extraction}\label{subsec:square-extraction}

\begin{definition}[Largest square divisor]\label{def:square-divisor}
For a positive integer
\[
  N=\prod_p p^{e_p}
\]
define
\[
  \sq(N)
  \coloneqq
  \prod_p p^{\lfloor e_p/2\rfloor}.
\]
Then \(\sq(N)^2\) is the largest square dividing \(N\).
\end{definition}

Put
\[
  u_a=\sq(a)
  \qquad\text{and}\qquad
  u_b=\sq(b).
\]

\begin{lemma}[Linear defect and square divisors]\label{lem:defect-square}
For every parity-class triple,
\[
  D_{\lin}(\mathcal{T})
  \leq
  (u_a u_b)^2.
\]
Equivalently,
\[
  \delta_{\lin}(\mathcal{T})
  \leq
  2\log(u_a u_b).
\]
\end{lemma}

\begin{proof}
For every integer \(e\geq1\) one has
\[
  e-1\leq2\lfloor e/2\rfloor
\]
since the two sides agree when \(e\) is odd and differ by one when \(e\) is even.
Apply this inequality to each prime exponent in \(a\) and \(b\), every such exponent being at least one.
\end{proof}

\begin{theorem}[Large square forced by transgression]\label{thm:large-square}
Suppose that \(K>1\) and \(E_{\varepsilon}(\mathcal{T})\geq0\).
Then
\[
  \max\{u_a^2,u_b^2\}
  \geq
  u_a u_b
  \geq
  \bigl(s(\mathcal{T})\,R_{\mathcal{K}}(\mathcal{T})\bigr)^{1/2}
  (2K)^{\theta_{\varepsilon}/2}
\]
where \(u_a=\sq(a)\) and \(u_b=\sq(b)\) are the square roots of the largest square divisors of the two summands.
\end{theorem}

\begin{proof}
Corollary~\ref{cor:quotient-certificate} gives
\[
  D_{\lin}(\mathcal{T})
  \geq
  s(\mathcal{T})\,R_{\mathcal{K}}(\mathcal{T})\,
  (2K)^{\theta_{\varepsilon}}.
\]
Lemma~\ref{lem:defect-square} gives \((u_a u_b)^2\geq D_{\lin}(\mathcal{T})\), and taking square roots bounds \(u_au_b\) from below by the stated quantity.
The remaining step is \(\max\{u_a,u_b\}^2\geq u_au_b\).
\end{proof}

Write
\[
  a=d_a u_a^2
  \qquad\text{and}\qquad
  b=d_b u_b^2
\]
where \(d_a\) and \(d_b\) are squarefree.

\begin{corollary}[Residual quadratic equation]\label{cor:residual-equation}
The extracted variables satisfy
\[
  d_a u_a^2+d_b u_b^2=2K.
\]
If \(E_{\varepsilon}(\mathcal{T})\geq0\), then the product \(u_a u_b\) satisfies the lower bound in Theorem~\ref{thm:large-square}.
\end{corollary}

\begin{lemma}[Two-adic congruence of the squarefree kernels]\label{lem:two-adic-kernel}
Write
\[
  a=d_a u_a^2
  \qquad\text{and}\qquad
  b=d_b u_b^2
\]
with \(d_a\) and \(d_b\) squarefree.
Then \(d_a\) and \(d_b\) are odd and coprime, and
\[
  d_a+d_b\equiv2K\pmod8.
\]
The residues of \(d_a\) and \(d_b\) modulo \(8\) then determine the two-adic valuation of \(c\), according to the three cases below.
\begin{enumerate}[label=\textup{(\arabic*)}]
\item \(v_2(c)=1\) precisely when \(d_a+d_b\equiv2\) or \(6\pmod8\).
\item \(v_2(c)=2\) precisely when \(d_a+d_b\equiv4\pmod8\).
\item \(v_2(c)\geq3\) precisely when \(d_a\equiv-d_b\pmod8\).
\end{enumerate}
\end{lemma}

\begin{proof}
The summands \(a\) and \(b\) are odd.
Since \(a=d_au_a^2\) and \(b=d_bu_b^2\), every one of \(d_a\), \(u_a\), \(d_b\), and \(u_b\) divides an odd number and is therefore odd itself.
Every divisor common to \(d_a\) and \(d_b\) divides both \(a\) and \(b\), whence \(\gcd(d_a,d_b)=1\).
The square of an odd integer is congruent to \(1\) modulo \(8\), so that
\[
  a\equiv d_a
  \qquad\text{and}\qquad
  b\equiv d_b
  \pmod8.
\]
Addition gives \(2K=a+b\equiv d_a+d_b\pmod8\).

The sum \(d_a+d_b\) is even.
The congruence \(2K\equiv2\) or \(6\pmod8\) holds precisely when \(K\) is odd, that is, when \(v_2(c)=1\).
The congruence \(2K\equiv4\pmod8\) holds precisely when \(v_2(K)=1\), that is, when \(v_2(c)=2\).
The congruence \(2K\equiv0\pmod8\) holds precisely when \(v_2(K)\geq2\), that is, when \(v_2(c)\geq3\).
These three cases exhaust the even residues modulo \(8\).
\end{proof}

\begin{proposition}[Two-adic completeness of the kernel congruence]\label{prop:two-adic-complete}
Let \(d_a\) and \(d_b\) be odd positive integers, and let \(2K\) be an even positive integer.
Then the equation
\[
  d_aX^2+d_bY^2=2K
\]
has a solution with \(X\) and \(Y\) in the unit group \(\Z_2^{\times}\) of the ring \(\Z_2\) of two-adic integers precisely when
\[
  d_a+d_b\equiv2K\pmod8.
\]
\end{proposition}

\begin{proof}
The square of a two-adic unit lies in \(1+8\Z_2\), and conversely every element of \(1+8\Z_2\) is the square of a two-adic unit.
The direct half is the congruence for odd squares.
For the converse half, let \(\alpha\in1+8\Z_2\) and put \(f(X)=X^2-\alpha\).
Then \(v_2(f(1))=v_2(1-\alpha)\geq3\) and \(v_2(f'(1))=1\), so that \(v_2(f(1))>2v_2(f'(1))\), and Hensel's lemma furnishes a root of \(f\) in \(\Z_2\), necessarily a unit.

Suppose first that units \(X\) and \(Y\) solve the equation.
Reduction modulo \(8\) gives \(2K\equiv d_aX^2+d_bY^2\equiv d_a+d_b\pmod8\).

Suppose now that the congruence holds.
Take \(Y=1\) and put
\[
  \alpha
  =
  \frac{2K-d_b}{d_a}.
\]
The numerator is odd and the denominator is a two-adic unit, so that \(\alpha\in\Z_2^{\times}\).
The congruence gives \(2K-d_b\equiv d_a\pmod8\), and multiplication by \(d_a^{-1}\), together with \(d_a^{-1}\equiv d_a\pmod8\) and \(d_a^2\equiv1\pmod8\), gives \(\alpha\equiv1\pmod8\).
By the first paragraph, \(\alpha\) is the square of a unit \(X\), and the pair \((X,1)\) solves the equation.
\end{proof}

\begin{remark}[Local filters and size conditions]
\label{rem:residual-local-obstructions}
For fixed odd squarefree integers \(d_a\) and \(d_b\), the equation
\[
  d_au_a^2+d_bu_b^2=2K
\]
admits direct local tests at every prime dividing \(2Kd_ad_b\).
At an odd prime, reduction gives quadratic-residue conditions which may exclude a candidate pair of kernels.
At the prime \(2\), Proposition~\ref{prop:two-adic-complete} gives the exact unit criterion
\[
  d_a+d_b\equiv2K\pmod8.
\]
These tests can only exclude candidates.
They neither produce coprime integral values of \(u_a\) and \(u_b\) nor give the lower bound
\[
  u_au_b
  \geq
  \bigl(s(\mathcal{T})R_{\mathcal K}(\mathcal{T})\bigr)^{1/2}
  (2K)^{\theta_\varepsilon/2}
\]
required by Theorem~\ref{thm:large-square}.
Local solubility and the integral size demanded of the solution remain separate questions.
\end{remark}

\begin{proposition}[Equidistribution of squarefree kernels modulo \(8\)]\label{prop:kernel-classes}
For each odd residue \(r\) modulo \(8\),
\[
  \#\bigl\{d\leq x:d\ \text{squarefree},\ d\equiv r\ (\mathrm{mod}\ 8)\bigr\}
  =
  \frac{x}{\pi^2}
  +
  O(\sqrt{x}).
\]
\end{proposition}

\begin{proof}
By M\"obius inversion over square divisors,
\[
  \#\bigl\{d\leq x:d\ \text{squarefree},\ d\equiv r\ (\mathrm{mod}\ 8)\bigr\}
  =
  \sum_{e\leq\sqrt{x}}
  \mu(e)\,
  \#\bigl\{f\leq x/e^2:e^2f\equiv r\ (\mathrm{mod}\ 8)\bigr\}.
\]
For even \(e\) the congruence \(e^2f\equiv r\pmod8\) has no solution, since \(e^2f\) is even while \(r\) is odd.
For odd \(e\) one has \(e^2\equiv1\pmod8\), the congruence reads \(f\equiv r\pmod8\), and the inner count equals \(x/(8e^2)+O(1)\).
The sum therefore equals
\[
  \frac{x}{8}
  \sum_{e\ \mathrm{odd}}
  \frac{\mu(e)}{e^2}
  +
  O(\sqrt{x})
\]
and the Euler products
\[
  \sum_{e\ \mathrm{odd}}\frac{\mu(e)}{e^2}
  =
  \prod_{p>2}\left(1-\frac{1}{p^2}\right)
  =
  \frac{6/\pi^2}{1-1/4}
  =
  \frac{8}{\pi^2}
\]
complete the evaluation.
\end{proof}

Proposition~\ref{prop:kernel-classes} shows that odd squarefree integers are equidistributed among the four reduced residue classes modulo \(8\).
For a fixed even residue \(2K\) modulo \(8\), exactly \(4\) of the \(16\) residue pairs \((d_a,d_b)\) modulo \(8\) satisfy \(d_a+d_b\equiv2K\pmod8\).
One fourth of all squarefree pairs therefore meets the two-adic condition of Proposition~\ref{prop:two-adic-complete}, leaving the size bound of Theorem~\ref{thm:large-square} and the coprimality of the extracted variables still to be met.

\subsection{Prime-power concentration}\label{subsec:concentration}

For each prime \(p\mid abc\), let \(N_p\) denote the unique member of \(\{a,b,c\}\) divisible by \(p\).
This member exists and is unique because the three supports are pairwise disjoint by Lemma~\ref{lem:disjoint-supports}.
Put
\[
  e_p
  \coloneqq
  v_p(N_p)-1
  \qquad\text{and}\qquad
  x_p
  \coloneqq
  e_p\log p.
\]

\begin{definition}[Linear support data]\label{def:linear-carriers}
Define
\[
  \omega_{\lin}(\mathcal{T})
  \coloneqq
  \omega(a)+\omega(b)
  =
  \omega(ab).
\]
The second equality holds because \(\gcd(a,b)=1\) places the two prime supports in disjoint sets.
When \(K>1\), let \(P_{\lin}(\mathcal{T})\) be the largest prime dividing \(ab\).
\end{definition}

\begin{lemma}[Weighted maximum over a prime support]\label{lem:weighted-max}
Let \(\mathcal{P}\) be a finite non-empty set of primes, and let \((e_p)_{p\in\mathcal{P}}\) be non-negative integers.
Then
\[
  \max_{p\in\mathcal{P}}e_p
  \geq
  \frac{
    \sum_{p\in\mathcal{P}}e_p\log p
  }{
    \sum_{p\in\mathcal{P}}\log p
  }.
\]
Since the left member is an integer, one has moreover
\[
  \max_{p\in\mathcal{P}}e_p
  \geq
  \left\lceil
    \frac{
      \sum_{p\in\mathcal{P}}e_p\log p
    }{
      \sum_{p\in\mathcal{P}}\log p
    }
  \right\rceil.
\]
\end{lemma}

\begin{proof}
Put \(E=\max_{p\in\mathcal{P}}e_p\).
For each \(p\in\mathcal{P}\) one has \(\log p>0\), whence \(e_p\log p\leq E\log p\).
Summation over \(\mathcal{P}\) gives
\[
  \sum_{p\in\mathcal{P}}e_p\log p
  \leq
  E\sum_{p\in\mathcal{P}}\log p
\]
and division by the positive quantity \(\sum_{p\in\mathcal{P}}\log p\) proves the first inequality.
The second follows from the integrality of \(E\).
\end{proof}

Since \(\sum_{p\mid ab}\log p=\log\rad(ab)\), Lemma~\ref{lem:weighted-max} bounds the maximum valuation \(v_p(N)-1\) from below by the average defect per logarithmic prime mass.
Division of the defect by \(\log\rad(ab)\), or else by \(\log\rad(abc)\), bounds the largest single-prime extraction from below, and the exponents being integers, that bound may be raised to its ceiling.

\begin{theorem}[Weighted concentration in one summand]\label{thm:linear-concentration}
Suppose that \(K>1\) and \(E_{\varepsilon}(\mathcal{T})\geq0\).
Then some prime \(p\) dividing a summand \(N\in\{a,b\}\) satisfies
\[
  v_p(N)-1
  \geq
  \frac{
    t^{\mathcal{K}}(\mathcal{T})
    +
    \theta_{\varepsilon}\log(2K)
  }{
    \log\rad(ab)
  }.
\]
In particular,
\[
  v_p(N)-1
  \geq
  \frac{
    t^{\mathcal{K}}(\mathcal{T})
    +
    \theta_{\varepsilon}\log(2K)
  }{
    \omega_{\lin}(\mathcal{T})
    \log P_{\lin}(\mathcal{T})
  }.
\]
\end{theorem}

\begin{proof}
For each prime \(p\mid ab\) one has \(N_p\in\{a,b\}\), and
\[
  \delta_{\lin}(\mathcal{T})
  =
  \sum_{p\mid ab}e_p\log p.
\]
Since
\[
  \sum_{p\mid ab}\log p
  =
  \log\rad(ab)
\]
and since \(K>1\) gives \(ab\geq2K-1>1\), the set of primes dividing \(ab\) is non-empty.
Lemma~\ref{lem:weighted-max}, applied to that set, shows that the largest value among the \(e_p\) is at least
\[
  \frac{
    \delta_{\lin}(\mathcal{T})
  }{
    \log\rad(ab)
  }.
\]
Theorem~\ref{thm:boundary-free-certificate} gives
\[
  \delta_{\lin}(\mathcal{T})
  \geq
  t^{\mathcal{K}}(\mathcal{T})
  +
  \theta_{\varepsilon}\log(2K)
\]
which proves the first inequality.
The second follows from
\[
  \log\rad(ab)
  \leq
  \omega_{\lin}(\mathcal{T})
  \log P_{\lin}(\mathcal{T}).
\]
\end{proof}

\begin{corollary}[Few small primes]\label{cor:few-small-primes}
Fix \(\delta_0>0\) and an integer \(k\geq1\).
Suppose that \(K>1\), that every prime dividing \(ab\) is at most \(K^{\delta_0}\), and that \(\omega_{\lin}(\mathcal{T})\leq k\).
If \(E_{\varepsilon}(\mathcal{T})\geq0\), then some prime \(p\mid ab\) satisfies
\[
  v_p(ab)-1
  \geq
  \frac{\theta_{\varepsilon}}{\delta_0k}.
\]
\end{corollary}

\begin{proof}
Apply Theorem~\ref{thm:linear-concentration}, discard the non-negative term \(t^{\mathcal{K}}(\mathcal{T})\), use \(\log(2K)\geq\log K\), and use \(\log P_{\lin}(\mathcal{T})\leq\delta_0\log K\).
\end{proof}

\subsection{Concentration of the total defect}\label{subsec:total-concentration}

With \(x_p=e_p\log p\) as above, now taken over every prime dividing \(abc\),
\[
  \delta_{\tot}(\mathcal{T})
  =
  \sum_{p\mid abc}x_p.
\]

\begin{definition}[Total support data]\label{def:total-carriers}
Define
\[
  \omega_{\tot}(\mathcal{T})
  \coloneqq
  \omega(abc)
\]
and let \(P_{\tot}(\mathcal{T})\) be the largest prime dividing \(abc\).
\end{definition}

\begin{theorem}[Weighted total concentration]\label{thm:total-concentration}
Suppose that \(K>1\) and \(E_{\varepsilon}(\mathcal{T})\geq0\).
Then some prime \(p\mid abc\) satisfies
\[
  v_p(N_p)-1
  \geq
  \frac{
    \log(2K-1)+\theta_{\varepsilon}\log(2K)
  }{
    \log\rad(abc)
  }.
\]
In particular,
\[
  v_p(N_p)-1
  \geq
  \frac{
    \log(2K-1)+\theta_{\varepsilon}\log(2K)
  }{
    \omega_{\tot}(\mathcal{T})
    \log P_{\tot}(\mathcal{T})
  }.
\]
\end{theorem}

\begin{proof}
For each prime \(p\mid abc\), put
\[
  e_p
  \coloneqq
  v_p(N_p)-1.
\]
Then
\[
  \delta_{\tot}(\mathcal{T})
  =
  \sum_{p\mid abc}e_p\log p.
\]
Since
\[
  \sum_{p\mid abc}\log p
  =
  \log\rad(abc)
\]
Lemma~\ref{lem:weighted-max}, applied to the set of primes dividing \(abc\), shows that some prime \(p\) satisfies
\[
  e_p
  \geq
  \frac{
    \delta_{\tot}(\mathcal{T})
  }{
    \log\rad(abc)
  }.
\]
Theorem~\ref{thm:total-criterion} and Lemma~\ref{lem:minimal-product} give
\[
  \delta_{\tot}(\mathcal{T})
  \geq
  \log(2K-1)+\theta_{\varepsilon}\log(2K).
\]
This proves the first inequality.
The second follows from
\[
  \log\rad(abc)
  \leq
  \omega_{\tot}(\mathcal{T})
  \log P_{\tot}(\mathcal{T}).
\]
\end{proof}

\begin{corollary}[Scale-independent total floor]\label{cor:total-floor}
Fix \(\delta_0>0\) and an integer \(k\geq1\).
Suppose that every prime dividing \(abc\) is at most \(K^{\delta_0}\) and that \(\omega_{\tot}(\mathcal{T})\leq k\).
If \(E_{\varepsilon}(\mathcal{T})\geq0\), then some prime \(p\mid abc\) satisfies
\[
  v_p(N_p)-1
  \geq
  \frac{1+\theta_{\varepsilon}}{\delta_0k}.
\]
\end{corollary}

\begin{proof}
The numerator of Theorem~\ref{thm:total-concentration} is at least \((1+\theta_{\varepsilon})\log K\).
The denominator is at most \(\delta_0k\log K\).
\end{proof}

As \(\varepsilon\) tends to zero the lower bound of Corollary~\ref{cor:total-floor} tends to \(1/(\delta_0k)\).
So long as the carrier count and the prime scale are held fixed, a positive floor for the total concentration persists.

Every lower bound in Theorems~\ref{thm:linear-concentration} and~\ref{thm:total-concentration} may be replaced by its ceiling, since \(v_p(N)-1\) and \(v_p(N_p)-1\) are integers, which is the second form of Lemma~\ref{lem:weighted-max}. The exact logarithmic support masses appear in the first inequalities of those two theorems, while the coarser consequences replace those masses by the worst-case quantities \(\omega_{\lin}\log P_{\lin}\) and \(\omega_{\tot}\log P_{\tot}\).
Both forms weaken as \(\omega_{\tot}(\mathcal{T})\) grows, the guaranteed multiplicity at a single prime falling off like \(\delta_{\tot}(\mathcal{T})/\log\rad(abc)\), and the exceptional-set estimates of Section~\ref{sec:benchmarks} draw upon such configurations.
Where \(\omega_{\tot}(\mathcal{T})\) stays bounded the opposite holds, and a large total defect forces a large multiplicity at one prime or at a few.
Any stronger pointwise conclusion would have to know more of the way multiplicity distributes itself across the prime support.

\begin{proposition}[Every defect share is at most \(\log K\)]\label{prop:share-ceiling}
Suppose that \(K>1\).
Then
\[
  x_p\leq\log K
\]
for every prime \(p\mid abc\).
\end{proposition}

\begin{proof}
For \(p=2\) one has \(N_2=c=2K\), so that \(x_2=v_2(K)\log2\), which is the two-adic defect \(\delta_2(\mathcal{T})\) of Definition~\ref{def:total-defect}, and \(2^{v_2(K)}\mid K\) gives \(x_2\leq\log K\).

For an odd prime \(p\mid c\),
\[
  x_p
  \leq
  \log c-\log p
  \leq
  \log(2K)-\log3
  <
  \log K.
\]

Suppose finally that \(p\) divides \(a\) or \(b\).
The corresponding summand \(N_p\) is at most \(2K-1\) and \(p\geq3\), whence
\[
  x_p
  \leq
  \log(2K-1)-\log3
  <
  \log K
\]
the last inequality resting on \(2K-1<3K\).
Thus \(x_p\leq\log K\) in every case.
\end{proof}

\begin{theorem}[Two repeated primes]\label{thm:two-repeated-primes}
Suppose that \(K>1\) and \(E_{\varepsilon}(\mathcal{T})\geq0\).
Then at least two distinct primes \(p\mid abc\) satisfy
\[
  v_p(N_p)\geq2.
\]
More precisely, for every prime \(p\mid abc\) one has
\[
  \delta_{\tot}(\mathcal{T})-x_p
  \geq
  \log\left(2-\frac{1}{K}\right)
  +
  \theta_{\varepsilon}\log(2K).
\]
\end{theorem}

\begin{proof}
Theorem~\ref{thm:total-criterion} and Lemma~\ref{lem:minimal-product} give
\[
  \delta_{\tot}(\mathcal{T})
  \geq
  \log(2K-1)+\theta_{\varepsilon}\log(2K).
\]
Fix any prime \(p\mid abc\).
Proposition~\ref{prop:share-ceiling} gives \(x_p\leq\log K\), and hence
\[
\begin{aligned}
  \delta_{\tot}(\mathcal{T})-x_p
  &\geq
  \log(2K-1)-\log K
  +
  \theta_{\varepsilon}\log(2K)
  \\
  &=
  \log\left(2-\frac{1}{K}\right)
  +
  \theta_{\varepsilon}\log(2K).
\end{aligned}
\]
Since \(K>1\), the right-hand side is strictly positive.

Suppose, to the contrary, that a single prime \(p_0\) satisfies \(v_{p_0}(N_{p_0})\geq2\).
The shares \(x_p=(v_p(N_p)-1)\log p\) attached to the remaining primes then vanish, so that \(\delta_{\tot}(\mathcal{T})=x_{p_0}\) and \(\delta_{\tot}(\mathcal{T})-x_{p_0}=0\), against the strict positivity just established.
At least two distinct primes therefore carry positive defect shares.
\end{proof}

\subsection{Two recurrence families near the defect threshold}\label{subsec:recurrences}

The thresholds established above admit no improvement beyond a bounded additive amount.
The two recurrences below keep the total defect within an absolute constant of the principal threshold \(\log(ab)\) of Theorem~\ref{thm:total-criterion}.

\subsubsection{A Pell family}\label{subsec:pell-family}

\begin{lemma}[Pell recurrence]\label{lem:pell-recurrence}
Define
\[
  (x_1,y_1)=(1,1)
\]
and
\[
  x_{n+1}=3x_n+4y_n
  \qquad
  y_{n+1}=2x_n+3y_n.
\]
Then
\[
  x_n^2-2y_n^2=-1
\]
for every \(n\geq1\).
The integers \(x_n\) and \(y_n\) are positive, odd, and coprime.
\end{lemma}

\begin{proof}
A direct calculation gives
\[
  (3x+4y)^2-2(2x+3y)^2=x^2-2y^2.
\]
The quadratic identity follows by induction.
Positivity and oddness are preserved by the recurrence.
Every common divisor of \(x_n\) and \(y_n\) divides \(x_n^2-2y_n^2=-1\).
\end{proof}

\begin{proposition}[Pell triples]\label{prop:pell-triples}
For \(n\geq2\), define
\[
  \mathcal{P}_n
  \coloneqq
  \bigl(1,x_n^2,2y_n^2\bigr).
\]
Then \(\mathcal{P}_n\) is a parity-class triple and
\[
  \rad(abc)
  \leq
  2x_ny_n
  <
  \sqrt2\,c.
\]
One therefore has
\[
  \delta_{\tot}(\mathcal{P}_n)-\log(ab)
  >
  -\frac12\log2.
\]
Moreover,
\[
  \delta_{\lin}(\mathcal{P}_n)
  \geq
  \log x_n.
\]
\end{proposition}

\begin{proof}
The Pell identity gives \(1+x_n^2=2y_n^2\).
Coprimality follows from Lemma~\ref{lem:pell-recurrence}.
Since \(\rad(x_n^2)\leq x_n\) and \(\rad(y_n^2)\leq y_n\), one has \(\rad(abc)\leq2x_ny_n\).
The inequality \(x_n<\sqrt2\,y_n\) gives the stated comparison with \(c=2y_n^2\).
Proposition~\ref{prop:total-energy} at exponent zero gives the total-defect bound.
For the final assertion, each prime \(p\mid x_n\) contributes \((2v_p(x_n)-1)\log p\geq v_p(x_n)\log p\) to the linear defect.
\end{proof}

\begin{example}\label{ex:pell-hit}
The fourth Pell pair is \((x_4,y_4)=(239,169)\).
It gives
\[
  1+239^2=2\cdot169^2
\]
in which the second coordinate is the square \(y_4=13^2\).
The radical of the product is \(2\cdot239\cdot13=6214\), while \(c=57122\), so this triple has quality greater than one.
\end{example}

The non-squarefree factor in \(y_4=13^2\) propagates along the recurrence via algebraic divisibility, yielding an infinite sequence of triples with quality strictly greater than one.

\begin{lemma}[Divisibility propagation in the Pell family]\label{lem:pell-divisibility-propagation}
Let \((x_n,y_n)\) be the sequence of Lemma~\ref{lem:pell-recurrence}.
If \(2n-1=t(2r-1)\) for positive integers \(n, r, t\), then \(t\) is odd, \(x_r\mid x_n\), and \(y_r\mid y_n\).
\end{lemma}

\begin{proof}
Put
\[
  \alpha=1+\sqrt2
  \qquad\text{and}\qquad
  \beta=1-\sqrt2.
\]
The recurrence of Lemma~\ref{lem:pell-recurrence} is multiplication by \(\alpha^2=3+2\sqrt2\), and hence
\[
  x_n+y_n\sqrt2=\alpha^{2n-1}
  \qquad\text{and}\qquad
  x_n-y_n\sqrt2=\beta^{2n-1}.
\]
Set \(h=2r-1\), \(A=\alpha^h\), and \(B=\beta^h\).
Since both \(2n-1\) and \(h\) are odd, the quotient \(t\) is odd.
The two expressions
\[
  \frac{A^t-B^t}{A-B}
  \qquad\text{and}\qquad
  \frac{A^t+B^t}{A+B}
\]
are algebraic integers.
For odd \(t\) they are symmetric in \(A\) and \(B\), so they are fixed by conjugation and are therefore rational integers.
The identities
\[
  2\sqrt2\,y_n=A^t-B^t
  \qquad\text{and}\qquad
  2x_n=A^t+B^t
\]
together with
\[
  2\sqrt2\,y_r=A-B
  \qquad\text{and}\qquad
  2x_r=A+B
\]
give \(y_r\mid y_n\) and \(x_r\mid x_n\).
\end{proof}

\begin{theorem}[An infinite Pell family above quality one]\label{thm:pell-infinite-quality}
For every \(n\geq2\),
\[
  q(\mathcal{P}_n)>1
  \quad\Longleftrightarrow\quad
  x_ny_n\ \text{is not squarefree}.
\]
In particular,
\[
  q(\mathcal{P}_{7j+4})>1
  \qquad(j\geq0).
\]
More quantitatively,
\[
  q(\mathcal{P}_{7j+4})-1
  >
  \frac{\log(13/\sqrt2)}
       {\log\bigl(2y_{7j+4}^2\bigr)}.
\]
The progression \(n\equiv4\pmod7\) therefore yields
\[
  \liminf_{X\to\infty}
  \frac{
    \#\{n\leq X:q(\mathcal{P}_n)>1\}
  }{X}
  \geq
  \frac17.
\]
\end{theorem}

\begin{proof}
Lemma~\ref{lem:pell-recurrence} gives \(\gcd(x_n,y_n)=1\) with both coordinates odd, and therefore
\[
  \rad\bigl(1\cdot x_n^2\cdot2y_n^2\bigr)
  =
  2\rad(x_n)\rad(y_n).
\]
Define
\[
  H_n
  \coloneqq
  \frac{x_ny_n}{\rad(x_n)\rad(y_n)}.
\]
Then
\[
  \frac{c}{\rad(abc)}
  =
  \frac{2y_n^2}{2\rad(x_n)\rad(y_n)}
  =
  H_n\frac{y_n}{x_n}.
\]
For \(n\geq2\) the Pell identity \(x_n^2=2y_n^2-1\) gives
\[
  1<\frac{x_n}{y_n}<\sqrt2.
\]
If \(x_ny_n\) is squarefree, then \(H_n=1\), so that \(c/\rad(abc)=y_n/x_n<1\).
If \(x_ny_n\) is not squarefree, then an odd prime occurs to exponent at least two in one of \(x_n\) and \(y_n\), and hence \(H_n\geq3\).
It follows that
\[
  \frac{c}{\rad(abc)}
  >
  \frac{3}{\sqrt2}
  >
  1.
\]
This proves the equivalence, and shows in addition that quality one itself never occurs for \(n\geq2\).

Example~\ref{ex:pell-hit} gives \(y_4=169=13^2\).
For \(n=7j+4\) one has
\[
  2n-1=7(2j+1)
\]
so that Lemma~\ref{lem:pell-divisibility-propagation}, applied with \(r=4\), gives \(y_4\mid y_n\).
Thus \(13^2\mid y_n\) and \(H_n\geq13\).
It follows that
\[
  \frac{c}{\rad(abc)}
  =
  H_n\frac{y_n}{x_n}
  >
  \frac{13}{\sqrt2}.
\]
Since \(\rad(abc)<c\) on this subsequence, and since \(c=2y_n^2\),
\[
\begin{aligned}
  q(\mathcal{P}_n)-1
  &=
  \frac{\log\bigl(c/\rad(abc)\bigr)}{\log\rad(abc)}
  \\
  &>
  \frac{\log(13/\sqrt2)}{\log\bigl(2y_n^2\bigr)}
\end{aligned}
\]
which is the asserted estimate.
The progression \(n\equiv4\pmod7\) has density \(1/7\).
\end{proof}

Theorem~\ref{thm:pell-infinite-quality} provides an explicit one-parameter family with infinitely many triples of quality \(q>1\), of the same kind as the Mersenne family in Theorem~\ref{thm:mersenne-hits}. In both cases the excess \(q-1\) decays to zero as the height grows, so neither family contains infinitely many triples transgressive at any fixed exponent \(\varepsilon>0\).

\subsubsection{A dyadic norm family}\label{subsec:dyadic-family}

\begin{lemma}[A dyadic recurrence]\label{lem:dyadic-ascent}
Suppose that \(m\geq3\) and that \(x,y\) are odd coprime positive integers satisfying
\[
  x^2+7y^2=2^m.
\]
Define
\[
  (x',y')
  =
  \begin{cases}
  \left(
    \dfrac{x+7y}{2},
    \dfrac{|x-y|}{2}
  \right)
  &
  x\not\equiv y\pmod4
  \\[1.2em]
  \left(
    \dfrac{|x-7y|}{2},
    \dfrac{x+y}{2}
  \right)
  &
  x\equiv y\pmod4.
  \end{cases}
\]
Then \(x'\) and \(y'\) are odd coprime positive integers satisfying
\[
  x'^2+7y'^2=2^{m+1}.
\]
\end{lemma}

\begin{proof}
For either compatible choice of signs, one has
\[
\begin{aligned}
  \left(\frac{x\pm7y}{2}\right)^2
  +
  7\left(\frac{x\mp y}{2}\right)^2
  &=
  \frac{(x\pm7y)^2+7(x\mp y)^2}{4}
  \\
  &=
  2(x^2+7y^2).
\end{aligned}
\]
Thus the transformed pair satisfies
\[
  x'^2+7y'^2=2^{m+1}.
\]

It remains to verify the parity of the quotients.
Suppose first that \(x\not\equiv y\pmod4\).
Since \(x\) and \(y\) are odd, one is congruent to \(1\) modulo \(4\) and the other is congruent to \(3\) modulo \(4\).
Hence
\[
  x-y\equiv2\pmod4.
\]
Moreover,
\[
  x+7y=(x-y)+8y\equiv2\pmod4.
\]
Therefore both \((x+7y)/2\) and \(|x-y|/2\) are odd integers.

Suppose now that \(x\equiv y\pmod4\).
Since both integers are odd, one has
\[
  x+y\equiv2\pmod4.
\]
One also has
\[
  x-7y=(x+y)-8y\equiv2\pmod4.
\]
Therefore both \(|x-7y|/2\) and \((x+y)/2\) are odd integers.

Every common divisor of \(x'\) and \(y'\) is odd and divides
\[
  x'^2+7y'^2=2^{m+1}.
\]
It must therefore equal \(1\).
The transformed coordinates are positive.
Indeed, suppose that \(x=y\).
Then the identity \(x^2+7y^2=2^m\) gives \(8x^2=2^m\), which with \(x\) odd occurs only for the seed \((x,y,m)=(1,1,3)\).
In this case the second branch gives the positive pair \((3,1)\).
Suppose that \(x=7y\).
Then \(7\) divides \(x^2+7y^2=2^m\), which is impossible.
This completes the proof.
\end{proof}

Starting with \((x_3,y_3)=(1,1)\), Lemma~\ref{lem:dyadic-ascent} gives a pair \((x_m,y_m)\) for every \(m\geq3\).

\begin{proposition}[Dyadic triples]\label{prop:dyadic-triples}
For \(m\geq4\), let \(\mathcal{D}_m\) have odd summands \(x_m^2\) and \(7y_m^2\) and let
\[
  c=2^m.
\]
Then \(\mathcal{D}_m\) is a parity-class triple with \(K=2^{m-1}\), and
\[
  \rad(abc)
  \leq
  14x_my_m
  \leq
  \sqrt7\,c.
\]
Hence
\[
  \delta_{\tot}(\mathcal{D}_m)-\log(ab)
  \geq
  -\frac12\log7.
\]
The two-adic defect is
\[
  \delta_2(\mathcal{D}_m)=\log K.
\]
\end{proposition}

\begin{proof}
The summands are odd and coprime, and their sum is \(2^m\).
The arithmetic-geometric mean inequality gives
\[
  x_m^2+7y_m^2
  \geq
  2\sqrt7\,x_my_m.
\]
Thus \(14x_my_m\leq\sqrt7\,2^m\).
The total-defect inequality follows from Proposition~\ref{prop:total-energy} at exponent zero.
Since \(c=2^m\) and \(K=2^{m-1}\), the last identity follows from Definition~\ref{def:total-defect}.
\end{proof}

\begin{example}\label{ex:dyadic-hit}
At \(m=6\), the recurrence gives \((x_6,y_6)=(1,3)\).
The associated triple is
\[
  1+63=64.
\]
Its radical is \(42\), and its quality is greater than one.
\end{example}

Propositions~\ref{prop:pell-triples} and~\ref{prop:dyadic-triples} keep the total defect within an absolute constant of the principal threshold \(\log(ab)\).
For the Pell family the sign of what remains is settled by Theorem~\ref{thm:pell-infinite-quality}.
For the dyadic family it is undetermined, since the recurrence of Lemma~\ref{lem:dyadic-ascent} controls neither the squarefree parts of \(x_m\) and \(y_m\) nor the balance between the two summands.

\section{The Mersenne line}\label{sec:mersenne}

On the line
\[
  1+(2^m-1)=2^m
\]
the whole defect of \(c\) lies at the prime \(2\), and what remains to be understood is the squarefull part of the Mersenne number \(2^m-1\), so that the obstruction becomes a question in the single variable \(m\).

\subsection{The exact Mersenne criterion}\label{subsec:mersenne-criterion}

\begin{definition}[Mersenne triple and defect]\label{def:mersenne-triple}
For \(m\geq2\), define
\[
  \mathcal{V}_m
  \coloneqq
  \bigl(1,2^m-1,2^m\bigr).
\]
Define
\[
  \Delta_m
  \coloneqq
  \log\left(
    \frac{2^m-1}{\rad(2^m-1)}
  \right)
\]
and
\[
  L_m
  \coloneqq
  \exp(\Delta_m)
  =
  \frac{2^m-1}{\rad(2^m-1)}.
\]
\end{definition}

For \(\mathcal{V}_m\) one has
\[
  K=2^{m-1}
  \qquad
  t^{\mathcal{K}}(\mathcal{V}_m)=0
  \qquad
  \delta_{\lin}(\mathcal{V}_m)=\Delta_m
\]
and
\[
  \delta_{\tot}(\mathcal{V}_m)
  =
  (m-1)\log2+\Delta_m.
\]

\begin{proposition}[Exact Mersenne transgression criterion]\label{prop:mersenne-exact}
Let \(\varepsilon>0\) and \(m\geq2\).
Then
\[
  E_{\varepsilon}(\mathcal{V}_m)\geq0
\]
holds precisely when
\[
  \Delta_m
  \geq
  \theta_{\varepsilon}m\log2
  +
  \log\bigl(2(1-2^{-m})\bigr).
\]
Equivalently,
\[
  E_{\varepsilon}(\mathcal{V}_m)\geq0
\]
holds precisely when
\[
  L_m
  \geq
  2(1-2^{-m})\,2^{\theta_{\varepsilon}m}.
\]
The quality satisfies
\[
  q(\mathcal{V}_m)>1
\]
precisely when
\[
  \Delta_m
  >
  \log\bigl(2(1-2^{-m})\bigr).
\]
\end{proposition}

\begin{proof}
Theorem~\ref{thm:total-criterion} gives
\[
  (m-1)\log2+\Delta_m
  \geq
  \log(2^m-1)+\theta_{\varepsilon}m\log2.
\]
Since
\[
  \log(2^m-1)
  =
  m\log2+\log(1-2^{-m})
\]
rearrangement proves the first equivalence.
Exponentiation proves the second.
The quality-one criterion follows from the sign of \(\log c-\log\rad(abc)\).
\end{proof}

Proposition~\ref{prop:mersenne-exact} is the exact linear threshold specialized.
For \(\mathcal{V}_m\) one has \(s=1\), \(t^{\mathcal{K}}=0\), and \(2K=2^m\), so the boundary term of Theorem~\ref{thm:exact-linear-threshold} equals
\[
  \log\Bigl(2-\frac{1}{K}\Bigr)
  =
  \log\bigl(2(1-2^{-m})\bigr).
\]
The Mersenne line thus realizes the least possible boundary term, and the distance between the exact threshold and the certificate of Theorem~\ref{thm:boundary-free-certificate} stays bounded along it.
That same distance diverges like \(\log s\) along an interior sequence.

\subsection{An infinite family above quality one}\label{subsec:mersenne-hits}

\begin{theorem}[Infinite Mersenne triples above quality one]\label{thm:mersenne-hits}
For every integer \(n\geq1\), the triple
\[
  \mathcal{V}_{6n}
  =
  \bigl(1,2^{6n}-1,2^{6n}\bigr)
\]
has quality greater than one.
\end{theorem}

\begin{proof}
Since \(2^6\equiv1\pmod9\), one has
\[
  9\mid2^{6n}-1.
\]
The prime \(3\) therefore contributes at least a factor \(3\) to \(L_{6n}\), so
\[
  L_{6n}\geq3.
\]
On the other hand,
\[
  2(1-2^{-6n})<2.
\]
Proposition~\ref{prop:mersenne-exact} now gives \(q(\mathcal{V}_{6n})>1\).
\end{proof}

The same argument applies to any odd prime lifted to a higher power.
Whenever \(p^h \mid 2^d-1\) with \(h\geq2\), the entire arithmetic progression of multiples of \(d\) yields triples with quality strictly greater than one.
Theorem~\ref{thm:mersenne-squarefree-classification} characterizes the union of all such arithmetic progressions as the full set of indices where \(q(\mathcal{V}_m) > 1\).

Theorem~\ref{thm:mersenne-hits} gives infinitely many triples with \(q>1\).
At no fixed \(\varepsilon>0\) does it give \(q\geq1+\varepsilon\) infinitely often.
Proposition~\ref{prop:mersenne-exact} requires \(\Delta_m\) to grow linearly in \(m\) for a fixed exponent.
Theorem~\ref{thm:lcm-margin} strengthens the present statement to a quantitative margin of order \(\log m/m\) on the exponents \(\operatorname{lcm}(1,\dots,k)\), and Theorem~\ref{thm:extremal-lcm} shows that margin to be extremal among the mechanisms free of Wieferich primes.

\subsection{Fixed primes give only logarithmic defect}\label{subsec:fixed-primes}

\begin{lemma}[Lifting a fixed prime, by lifting the exponent]\label{lem:mersenne-lifting}
Let \(p\) be an odd prime.
Define
\[
  d_p=\ord_p(2)
  \qquad\text{and}\qquad
  A_p=v_p(2^{d_p}-1).
\]
Then
\[
  p\mid2^m-1
  \quad\Longleftrightarrow\quad
  d_p\mid m.
\]
If \(d_p\mid m\), then
\[
  v_p(2^m-1)
  =
  A_p+v_p\left(\frac{m}{d_p}\right).
\]
\end{lemma}

\begin{proof}
Every prime dividing \(2^m-1\) is odd, so that the restriction to odd \(p\) loses nothing of the prime support.
The first equivalence is the definition of the multiplicative order.
Suppose that \(d_p\mid m\), and write
\[
  m=d_pr.
\]
Since \(p\geq3\) is odd and \(p\mid2^{d_p}-1\), no exceptional case arises, those being confined to \(p=2\), and the lifting-the-exponent identity, whose classical form for the divisors of \(a^n-b^n\) goes back to Lucas and to Birkhoff and Vandiver~\cite{Lucas1878,BirkhoffVandiver1904}, gives
\[
\begin{aligned}
  v_p(2^m-1)
  &=
  v_p\bigl((2^{d_p})^r-1\bigr)
  \\
  &=
  v_p(2^{d_p}-1)+v_p(r)
  \\
  &=
  A_p+v_p\left(\frac{m}{d_p}\right).
\end{aligned}
\]
\end{proof}

\begin{definition}[Defect on a finite prime set]\label{def:finite-prime-defect}
For a finite set \(S\) of odd primes, define
\[
  \Delta_{m,S}
  \coloneqq
  \sum_{\substack{p\in S\\p\mid2^m-1}}
  \bigl(v_p(2^m-1)-1\bigr)\log p.
\]
\end{definition}

\begin{proposition}[Finite-prime bound]\label{prop:finite-prime-bound}
Let \(S\) be a finite set of odd primes, and put
\[
  C_S
  \coloneqq
  \sum_{p\in S}(A_p-1)\log p.
\]
Then
\[
  \Delta_{m,S}
  \leq
  \#S\log m+C_S
\]
for every \(m\geq2\).
\end{proposition}

\begin{proof}
Put
\[
  S(m)
  \coloneqq
  \bigl\{
    p\in S
    :
    p\mid2^m-1
  \bigr\}
\]
so that Definition~\ref{def:finite-prime-defect} reads
\[
  \Delta_{m,S}
  =
  \sum_{p\in S(m)}
  \bigl(v_p(2^m-1)-1\bigr)\log p.
\]
Let \(p\in S(m)\).
Lemma~\ref{lem:mersenne-lifting} gives \(d_p\mid m\) and
\[
  \bigl(v_p(2^m-1)-1\bigr)\log p
  =
  (A_p-1)\log p
  +
  v_p\left(\frac{m}{d_p}\right)\log p.
\]
The integer \(m/d_p\) divides \(m\), whence \(p^{v_p(m/d_p)}\) divides \(m\) and
\[
  v_p\left(\frac{m}{d_p}\right)\log p
  =
  \log p^{v_p(m/d_p)}
  \leq
  \log m.
\]
Summation over \(S(m)\) therefore yields
\[
\begin{aligned}
  \Delta_{m,S}
  &\leq
  \sum_{p\in S(m)}
  \bigl(
    (A_p-1)\log p+\log m
  \bigr)
  \\
  &=
  \sum_{p\in S(m)}(A_p-1)\log p
  +
  \#S(m)\log m.
\end{aligned}
\]
Every term \((A_p-1)\log p\) is non-negative, since \(A_p\geq1\) for every odd prime \(p\).
Extension of the first sum from \(S(m)\) to \(S\) can therefore only increase it, and it gives \(C_S\).
The inclusion \(S(m)\subseteq S\) gives \(\#S(m)\leq\#S\), and \(\log m\geq0\) for \(m\geq2\).
Hence
\[
  \Delta_{m,S}
  \leq
  C_S+\#S\log m
\]
as asserted.
\end{proof}

\begin{corollary}[Remainder after a finite prime set]\label{cor:finite-prime-escape}
Fix \(\varepsilon>0\).
Suppose that \((m_j)\) is an increasing sequence satisfying
\[
  E_{\varepsilon}(\mathcal{V}_{m_j})\geq0.
\]
Then, for every finite set \(S\) of odd primes,
\[
\begin{aligned}
  \Delta_{m_j}-\Delta_{m_j,S}
  \geq{}&
  \theta_{\varepsilon}m_j\log2
  +
  \log\bigl(2(1-2^{-m_j})\bigr)
  \\
  &-
  \#S\log m_j-C_S.
\end{aligned}
\]
In particular,
\[
  \liminf_{j\to\infty}
  \frac{
    \Delta_{m_j}-\Delta_{m_j,S}
  }{
    m_j\log2
  }
  \geq
  \theta_{\varepsilon}.
\]
\end{corollary}

\begin{proof}
Subtract Proposition~\ref{prop:finite-prime-bound} from the lower bound in Proposition~\ref{prop:mersenne-exact}.
Divide by \(m_j\log2\) and pass to the lower limit.
\end{proof}

\begin{example}[The prime \(3\)]\label{ex:prime-three}
For
\[
  m=6\cdot3^k
\]
Lemma~\ref{lem:mersenne-lifting} gives
\[
  v_3(2^m-1)=k+2.
\]
The contribution of \(3\) to \(\Delta_m\) is
\[
  (k+1)\log3
  =
  \log\left(\frac{m}{2}\right).
\]
This is logarithmic in \(m\), and falls far short of the quantity linear in \(m\) that transgression at a fixed exponent would require.
\end{example}

\subsection{Squarefull and prime-power conditions}\label{subsec:conditional-route}

\begin{definition}[Mersenne threshold]\label{def:mersenne-threshold}
Define
\[
  \mathcal{A}_{\varepsilon}(m)
  \coloneqq
  2(1-2^{-m})\,2^{\theta_{\varepsilon}m}.
\]
Then Proposition~\ref{prop:mersenne-exact} reads
\[
  E_{\varepsilon}(\mathcal{V}_m)\geq0
  \quad\Longleftrightarrow\quad
  L_m\geq\mathcal{A}_{\varepsilon}(m).
\]
\end{definition}

\begin{proposition}[A single-prime condition]\label{prop:single-prime-route}
Fix \(\varepsilon>0\) and \(m\geq2\).
Suppose that an odd prime \(p\) and an integer \(e\geq2\) satisfy
\[
  p^e\mid2^m-1
\]
and
\[
  p^{e-1}
  \geq
  \mathcal{A}_{\varepsilon}(m).
\]
Then
\[
  E_{\varepsilon}(\mathcal{V}_m)\geq0.
\]
Moreover,
\[
  e>1+\varepsilon.
\]
\end{proposition}

\begin{proof}
The divisibility gives
\[
  L_m\geq p^{e-1}.
\]
Definition~\ref{def:mersenne-threshold} gives the transgression.
Since
\[
  p^e<2^m
\]
and
\[
  p^{e-1}>2^{\theta_{\varepsilon}m}
\]
one obtains
\[
  \frac{e-1}{e}>\theta_{\varepsilon}.
\]
This is equivalent to \(e>1+\varepsilon\).
\end{proof}

\begin{corollary}[A fixed prime cannot meet the threshold]\label{cor:fixed-prime-no-route}
Let \(p\) be a fixed odd prime.
Whenever \(p\mid2^m-1\), one has
\[
  p^{v_p(2^m-1)-1}
  \leq
  p^{A_p-1}m.
\]
An infinite family satisfying Proposition~\ref{prop:single-prime-route} must therefore use primes escaping every fixed finite set.
\end{corollary}

\begin{proof}
Lemma~\ref{lem:mersenne-lifting} gives
\[
  p^{v_p(2^m-1)-1}
  =
  p^{A_p-1}
  p^{v_p(m/d_p)}
  \leq
  p^{A_p-1}m.
\]
For fixed \(p\), this bound is linear in \(m\), while \(\mathcal{A}_{\varepsilon}(m)\) is exponential in \(m\).
\end{proof}

\begin{proposition}[Squarefull bounds]\label{prop:squarefull-corridor}
Put
\[
  u_m=\sq(2^m-1).
\]
Then
\[
  u_m\leq L_m\leq u_m^2.
\]
Thus if \(E_{\varepsilon}(\mathcal{V}_m)\geq0\), then
\[
  u_m
  \geq
  \mathcal{A}_{\varepsilon}(m)^{1/2}.
\]
If
\[
  u_m
  \geq
  \mathcal{A}_{\varepsilon}(m)
\]
then
\[
  E_{\varepsilon}(\mathcal{V}_m)\geq0.
\]
\end{proposition}

\begin{proof}
Write
\[
  2^m-1=\prod_p p^{e_p}.
\]
For every \(e_p\geq1\),
\[
  \left\lfloor\frac{e_p}{2}\right\rfloor
  \leq
  e_p-1
  \leq
  2\left\lfloor\frac{e_p}{2}\right\rfloor.
\]
Multiplication over the prime factors gives \(u_m\leq L_m\leq u_m^2\).
The implications follow from Definition~\ref{def:mersenne-threshold}.
\end{proof}

\begin{theorem}[Sufficient conditions for a Mersenne counterexample family]\label{thm:conditional-mersenne}
Fix \(\varepsilon>0\).
Suppose that there is an infinite set \(\mathcal{M}\) of positive integers satisfying
\[
  \sq(2^m-1)
  \geq
  \mathcal{A}_{\varepsilon}(m)
\]
for every \(m\in\mathcal{M}\).
Then
\[
  \bigl\{
    \mathcal{V}_m
    :
    m\in\mathcal{M}
  \bigr\}
\]
is an infinite family violating Conjecture~\ref{conj:parity-abc} at exponent \(\varepsilon\).

The same conclusion follows if, for every \(m\in\mathcal{M}\), there are an odd prime \(p_m\) and an integer \(e_m\geq2\) satisfying
\[
  p_m^{e_m}\mid2^m-1
\]
and
\[
  p_m^{e_m-1}
  \geq
  \mathcal{A}_{\varepsilon}(m).
\]
\end{theorem}

\begin{proof}
The first assertion follows from Proposition~\ref{prop:squarefull-corridor} and the second from Proposition~\ref{prop:single-prime-route}, the triples being pairwise distinct because distinct indices give distinct values of \(c=2^m\).
\end{proof}

Theorem~\ref{thm:mersenne-hits} gives an infinite family above quality one, and Theorem~\ref{thm:conditional-mersenne} supplies mechanisms sufficient for a fixed positive exponent.
An exact asymptotic criterion in terms of Wieferich primes is proved in Section~\ref{sec:wieferich}.

\section{Wieferich primes and the asymptotic Mersenne criterion}\label{sec:wieferich}

The Mersenne line depends on the growth of \(L_m\). Beyond a divisor of the exponent, that growth comes from the primes Wieferich described in 1909~\cite{Wieferich1909}.
Theorem~\ref{thm:exact-defect-law} states that division exactly.

\subsection{The exact defect law}\label{subsec:wieferich-decomposition}

\begin{definition}[Wieferich prime]\label{def:wieferich}
An odd prime \(p\) is a \emph{Wieferich prime} when
\[
  2^{p-1}\equiv1\pmod{p^2}.
\]
The set of Wieferich primes is written \(\mathcal{W}\).
\end{definition}

Only two members of \(\mathcal{W}\) are known, namely \(1093\) and \(3511\).
Extensive computation has produced no third example, and it is not known whether \(\mathcal{W}\) is finite or infinite.
The following lemma explains the relevance of this question to the Mersenne line.

\begin{lemma}[Wieferich primes and the lifting constant]\label{lem:wieferich-lifting}
Let \(p\) be an odd prime, and let
\[
  d_p=\ord_p(2)
  \qquad\text{and}\qquad
  A_p=v_p(2^{d_p}-1)
\]
be as in Lemma~\ref{lem:mersenne-lifting}.
Then
\[
  A_p\geq2
  \quad\Longleftrightarrow\quad
  p\in\mathcal{W}.
\]
\end{lemma}

\begin{proof}
Suppose first that \(A_p\geq2\).
Then \(p^2\) divides \(2^{d_p}-1\), so \(\ord_{p^2}(2)\) divides \(d_p\).
Reduction modulo \(p\) gives the reverse divisibility.
Hence the two orders coincide.
Since \(d_p\) divides \(p-1\), one obtains
\[
  p^2\mid2^{p-1}-1.
\]
Thus \(p\) lies in \(\mathcal{W}\).

Suppose now that \(p\in\mathcal{W}\).
Lemma~\ref{lem:mersenne-lifting}, applied with \(m=p-1\), gives
\[
  v_p(2^{p-1}-1)
  =
  A_p+v_p\left(\frac{p-1}{d_p}\right).
\]
The prime \(p\) does not divide \((p-1)/d_p\), so the second term vanishes.
Therefore
\[
  A_p=v_p(2^{p-1}-1)\geq2.
\]
\end{proof}

The constant \(A_p\) therefore exceeds one precisely for the Wieferich primes.
It measures the power of \(p\) present in \(2^m-1\) at the first exponent that admits the prime at all.
Every other prime enters \(2^m-1\) to the first power at that exponent, and only the factor \(v_p(m/d_p)\) of Lemma~\ref{lem:mersenne-lifting} can lift it further.

\begin{theorem}[Exact defect law]\label{thm:exact-defect-law}
For \(m\geq2\) put
\[
  \mathcal{W}(m)
  \coloneqq
  \bigl\{
    p\in\mathcal{W}
    :
    p\mid2^m-1
  \bigr\}
\]
and
\[
  \Omega_m
  \coloneqq
  \sum_{p\in\mathcal{W}(m)}
  (A_p-1)\log p.
\]
The convention adopted is that the empty sum equals \(0\), so that
\[
  \Omega_m=0
  \qquad\text{whenever}\qquad
  \mathcal{W}(m)=\varnothing.
\]
Define
\[
  G_m
  \coloneqq
  \prod_{p\mid2^m-1}p^{v_p(m)}.
\]
Then \(G_m\) is the largest divisor of \(m\) all of whose prime factors divide \(2^m-1\), it divides the odd part of \(m\), and
\[
  \Delta_m
  =
  \Omega_m
  +
  \log G_m.
\]
\end{theorem}

\begin{proof}
Let \(p\) be a prime dividing \(2^m-1\).
Then \(p\) is odd, and Lemma~\ref{lem:mersenne-lifting} gives \(d_p\mid m\) together with
\[
  v_p(2^m-1)-1
  =
  (A_p-1)+v_p\left(\frac{m}{d_p}\right).
\]
The order \(d_p\) divides \(p-1\) (hence \(p\nmid d_p\)), and therefore
\[
  v_p\left(\frac{m}{d_p}\right)
  =
  v_p(m)-v_p(d_p)
  =
  v_p(m).
\]
Summation over the primes dividing \(2^m-1\) yields
\[
  \Delta_m
  =
  \sum_{p\mid2^m-1}(A_p-1)\log p
  +
  \sum_{p\mid2^m-1}v_p(m)\log p.
\]
By Lemma~\ref{lem:wieferich-lifting}, the terms of the first sum vanish outside \(\mathcal{W}(m)\), so that the first sum equals \(\Omega_m\), and the second sum equals \(\log G_m\) by the definition of \(G_m\).

It remains to identify \(G_m\).
Every prime factor of \(G_m\) divides \(2^m-1\) and is therefore odd, and \(G_m\) divides \(\prod_pp^{v_p(m)}=m\), whence \(G_m\) divides the odd part of \(m\).
Conversely, let \(g\) be a divisor of \(m\) all of whose prime factors divide \(2^m-1\).
Then \(v_p(g)\leq v_p(m)\) for every such prime, whence \(g\) divides \(G_m\).
Thus \(G_m\) is the largest divisor of the stated kind.
\end{proof}

\begin{corollary}[Two-sided Wieferich decomposition]\label{cor:defect-bounds}
For every \(m\geq2\),
\[
  \Omega_m
  \leq
  \Delta_m
  \leq
  \Omega_m+\log\left(\frac{m}{2^{v_2(m)}}\right)
  \leq
  \Omega_m+\log m.
\]
Equivalently,
\[
  \prod_{p\in\mathcal{W}(m)}p^{A_p-1}
  \leq
  L_m
  \leq
  m
  \prod_{p\in\mathcal{W}(m)}p^{A_p-1}.
\]
\end{corollary}

\begin{proof}
The divisor \(G_m\) of Theorem~\ref{thm:exact-defect-law} satisfies \(1\leq G_m\leq m\,2^{-v_2(m)}\leq m\), since it divides the odd part of \(m\).
The additive bounds follow from the exact law, and exponentiation gives the multiplicative form.
\end{proof}

The estimate of Corollary~\ref{cor:defect-bounds} improves upon the finite-prime bound of Proposition~\ref{prop:finite-prime-bound}, the single term \(\log m\) accounting at once for every non-Wieferich prime, however numerous these may be, and Theorem~\ref{thm:exact-defect-law} goes further in naming the remainder as the divisor \(G_m\) of \(m\).
The defect therefore separates into an intrinsic Wieferich part \(\Omega_m\) and a part bounded by the arithmetic of the exponent alone.
Since \(G_m\) divides \(m\), the second part is \(\log G_m\leq\log m=o(m)\), and linear growth in \(\Delta_m\), should it occur, can come only from \(\Omega_m\).

In the cyclotomic factorization \(2^m-1=\prod_{d\mid m}\Phi_d(2)\) the prime divisors are grouped by their order.
A prime rises above the first power in \(2^m-1\) either through the intrinsic constant \(A_p\geq2\), which places \(p\) within \(\mathcal{W}\), or through the divisibility \(p\mid m\), whose whole contribution the divisor \(G_m\) measures and whose smallest instance is \(3^2\mid2^6-1\), with \(3\) outside \(\mathcal{W}\).
From this dichotomy Theorem~\ref{thm:mersenne-squarefree-classification} draws the complete description of the crossings of quality one.

In its classical form the Wieferich--Mersenne connection is a construction. A single Wieferich prime \(p\) forces a square factor into \(2^{p(p-1)}-1\), yielding the low-radical triple \((1,2^{p(p-1)}-1,2^{p(p-1)})\) and its prime-power iterates, as constructed by Granville and Tucker~\cite{GranvilleTucker2002}. Transgressive triples arise in such a family, and the exact law gives the full defect at each exponent together with its variation from one exponent to the next.

\begin{remark}[General bases]\label{rem:general-base}
Nothing in the proof of Theorem~\ref{thm:exact-defect-law} is peculiar to the base \(2\).
The lifting identity of Lemma~\ref{lem:mersenne-lifting} holds for the divisors of \(a^m-b^m\) with coprime integers \(a>b\geq1\), the multiplicative order of \(a/b\) modulo \(p\) taking the place of the order of \(2\), and the primitive-divisor theorem of Bang extends to this generality in the form given it by Zsigmondy and by Birkhoff and Vandiver~\cite{Zsigmondy1892,BirkhoffVandiver1904,Bang1886}.
The exact law therefore transfers verbatim.
The defect of \(a^m-b^m\) equals an intrinsic sum over the base-\(a/b\) Wieferich primes together with the logarithm of the largest divisor of \(m\) whose prime factors divide \(a^m-b^m\).
The sole adjustments are the exclusion of the primes dividing \(ab\) and the replacement of the exceptional pair \((2,6)\) in Proposition~\ref{prop:primitive-ceiling} by the exceptional cases of the general theorem.

The parity class accommodates a whole family of such lines.
For every fixed even \(g\geq2\) the two summands \(1\) and \(g^m-1\) are odd while \(g^m\) is even, so that \((1,g^m-1,g^m)\) is a parity-class triple whose smaller summand is \(1\).
The two-adic part of the defect along such a line is
\[
  \delta_2\bigl(1,g^m-1,g^m\bigr)
  =
  \bigl(mv_2(g)-1\bigr)\log2
\]
by Definition~\ref{def:total-defect}.
This grows linearly in \(m\) for every even \(g\), the base entering only through the constant \(v_2(g)\), and at \(g=2\) it returns the value \((m-1)\log2\) of the Mersenne line.
Subsection~\ref{subsec:even-base-lines} carries the exact law and its normalized criterion along every one of these lines.
The exposition stays centered upon \(g=2\), where the odd radical of the center vanishes and the arithmetic of \(2^m-1\) is the most fully documented.
\end{remark}

\subsection{The normalized defect criterion}\label{subsec:wieferich-density}

\begin{definition}[Normalized Wieferich defect density]\label{def:wieferich-density}
With \(\Omega_m\) as in Theorem~\ref{thm:exact-defect-law}, define
\[
  \rho_{\mathcal{W}}
  \coloneqq
  \limsup_{m\to\infty}
  \frac{\Omega_m}{m\log2}.
\]
Then
\[
  0\leq\rho_{\mathcal{W}}\leq1.
\]
\end{definition}

Indeed, Corollary~\ref{cor:defect-bounds} gives
\[
  0
  \leq
  \Delta_m-\Omega_m
  \leq
  \log m
\]
while
\[
  \Delta_m
  =
  \log\left(
    \frac{2^m-1}{\rad(2^m-1)}
  \right)
  <
  m\log2.
\]

\begin{theorem}[Sharp density criterion on the Mersenne line]\label{thm:wieferich-density-criterion}
Let \(\varepsilon>0\) and put
\[
  \theta_{\varepsilon}
  =
  \frac{\varepsilon}{1+\varepsilon}.
\]
Then the three assertions below hold.
\begin{enumerate}[label=\textup{(\arabic*)}]
\item If \(\theta_{\varepsilon}<\rho_{\mathcal{W}}\), then \(E_{\varepsilon}(\mathcal{V}_m)\geq0\) for infinitely many integers \(m\).

\item If \(\theta_{\varepsilon}>\rho_{\mathcal{W}}\), then \(E_{\varepsilon}(\mathcal{V}_m)\geq0\) for only finitely many integers \(m\).

\item There exist \(\varepsilon>0\) and infinitely many \(m\) satisfying \(E_{\varepsilon}(\mathcal{V}_m)\geq0\) if and only if \(\rho_{\mathcal{W}}>0\).
\end{enumerate}
At the boundary \(\theta_{\varepsilon}=\rho_{\mathcal{W}}\), the normalized density alone yields no conclusion.
\end{theorem}

\begin{proof}
Suppose first that \(\theta_{\varepsilon}<\rho_{\mathcal{W}}\).
Choose \(\eta>0\) such that
\[
  \theta_{\varepsilon}+\eta<\rho_{\mathcal{W}}.
\]
There are infinitely many \(m\) for which
\[
  \Omega_m
  \geq
  (\theta_{\varepsilon}+\eta)m\log2.
\]
For all sufficiently large such \(m\), one has
\[
  \eta m\log2
  >
  \log\bigl(2(1-2^{-m})\bigr).
\]
Since \(\Delta_m\geq\Omega_m\), it follows that
\[
  \Delta_m
  \geq
  \theta_{\varepsilon}m\log2
  +
  \log\bigl(2(1-2^{-m})\bigr).
\]
Proposition~\ref{prop:mersenne-exact} now gives \(E_{\varepsilon}(\mathcal{V}_m)\geq0\).

Suppose next that \(\theta_{\varepsilon}>\rho_{\mathcal{W}}\).
Choose \(\eta>0\) such that
\[
  \rho_{\mathcal{W}}+\eta<\theta_{\varepsilon}.
\]
For all sufficiently large \(m\), one has
\[
  \Omega_m
  \leq
  (\rho_{\mathcal{W}}+\eta)m\log2.
\]
By Corollary~\ref{cor:defect-bounds},
\[
  \Delta_m
  \leq
  (\rho_{\mathcal{W}}+\eta)m\log2+\log m.
\]
The linear gap between \(\theta_{\varepsilon}m\log2\) and \((\rho_{\mathcal{W}}+\eta)m\log2\) eventually dominates \(\log m\).
Hence, for all sufficiently large \(m\),
\[
  \Delta_m
  <
  \theta_{\varepsilon}m\log2
  +
  \log\bigl(2(1-2^{-m})\bigr).
\]
Proposition~\ref{prop:mersenne-exact} gives \(E_{\varepsilon}(\mathcal{V}_m)<0\).

Suppose now that \(\rho_{\mathcal{W}}>0\).
Choose \(0<\theta<\rho_{\mathcal{W}}\) and put
\[
  \varepsilon=\frac{\theta}{1-\theta}.
\]
The first assertion gives infinitely many transgressions.
Conversely, suppose that infinitely many transgressions occur at some fixed \(\varepsilon\).
Then Proposition~\ref{prop:mersenne-exact} and Corollary~\ref{cor:defect-bounds} give
\[
  \Omega_m
  \geq
  \theta_{\varepsilon}m\log2
  +
  \log\bigl(2(1-2^{-m})\bigr)
  -
  \log m
\]
along an infinite sequence.
Therefore
\[
  \rho_{\mathcal{W}}\geq\theta_{\varepsilon}>0.
\]
\end{proof}

\begin{corollary}[Critical quality of the Mersenne line]\label{cor:mersenne-critical-quality}
One has
\[
  \limsup_{m\to\infty}q(\mathcal{V}_m)
  =
  \begin{cases}
    \dfrac{1}{1-\rho_{\mathcal{W}}}
      & 0\leq\rho_{\mathcal{W}}<1
    \\[1.2em]
    +\infty
      & \rho_{\mathcal{W}}=1.
  \end{cases}
\]
One may therefore define
\[
  \varepsilon_{\mathcal{V}}^*
  \coloneqq
  \sup\left\{
    \varepsilon>0
    :
    E_{\varepsilon}(\mathcal{V}_m)\geq0
    \text{ for infinitely many }m
  \right\}
\]
where the supremum of the empty set is taken to be \(0\).
Then
\[
  \varepsilon_{\mathcal{V}}^*
  =
  \begin{cases}
    \dfrac{\rho_{\mathcal{W}}}{1-\rho_{\mathcal{W}}}
      & 0\leq\rho_{\mathcal{W}}<1
    \\[1.2em]
    +\infty
      & \rho_{\mathcal{W}}=1.
  \end{cases}
\]
\end{corollary}

\begin{proof}
For the Mersenne triple,
\[
  \log\rad(abc)
  =
  (m+1)\log2
  +
  \log(1-2^{-m})
  -
  \Delta_m.
\]
It follows that
\[
  q(\mathcal{V}_m)
  =
  \left(
    1+\frac{1}{m}
    +
    \frac{\log(1-2^{-m})}{m\log2}
    -
    \frac{\Delta_m}{m\log2}
  \right)^{-1}.
\]
Corollary~\ref{cor:defect-bounds} gives
\[
  \limsup_{m\to\infty}
  \frac{\Delta_m}{m\log2}
  =
  \rho_{\mathcal{W}}.
\]
The asserted quality formula follows.
The formula for \(\varepsilon_{\mathcal{V}}^*\) follows from Theorem~\ref{thm:wieferich-density-criterion}.
\end{proof}

\subsection{Order--defect data}\label{subsec:order-defect}

The density \(\rho_{\mathcal{W}}\) can be expressed entirely through finite collections of Wieferich primes and the orders of \(2\) modulo those primes.

\begin{definition}[Order--defect data]\label{def:order-defect-data}
For a finite set \(S\subseteq\mathcal{W}\), define
\[
  D(S)
  \coloneqq
  \operatorname{lcm}_{p\in S}d_p
\]
and
\[
  B(S)
  \coloneqq
  \prod_{p\in S}p^{A_p-1}.
\]
The conventions adopted are
\[
  D(\varnothing)=1
  \qquad\text{and}\qquad
  B(\varnothing)=1.
\]
A sequence \((S_j)\) of finite subsets of \(\mathcal{W}\) is called \emph{admissible} when \(D(S_j)\to\infty\).
Define
\[
  \sigma_{\mathcal{W}}
  \coloneqq
  \sup_{(S_j)\ \mathrm{admissible}}\
  \limsup_{j\to\infty}
  \frac{\log B(S_j)}{D(S_j)\log2}.
\]
Should no admissible sequence exist, then the value is taken to be \(0\).
\end{definition}

\begin{remark}[The datum \(D_m\) divides the exponent]\label{rem:order-divides-m}
For a Mersenne index \(m\), every prime \(p\) dividing \(2^m-1\) has order \(d_p\mid m\), whence the datum \(D_m=D(\mathcal{W}(m))\) of Proposition~\ref{prop:order-obstruction} is itself a divisor of \(m\).
The division becomes proper as soon as \(m\) carries a prime factor shared by none of the orders \(d_p\).
The passage from \(D_m\) to the exponent \(m\) in that proposition rests on this divisibility alone.
\end{remark}

\begin{theorem}[Order--defect criterion]\label{thm:order-defect-density}
One has
\[
  \rho_{\mathcal{W}}
  =
  \sigma_{\mathcal{W}}.
\]
Thus an infinite Mersenne family transgressive at some fixed positive exponent exists if and only if there are a constant \(\kappa>0\) and finite sets \(S_j\subseteq\mathcal{W}\) such that
\[
  D(S_j)\longrightarrow\infty
\]
and
\[
  \log B(S_j)
  \geq
  \kappa D(S_j)\log2
\]
for infinitely many \(j\).
Equivalently,
\[
  B(S_j)\geq2^{\kappa D(S_j)}.
\]
\end{theorem}

\begin{proof}
Let \(S\subseteq\mathcal{W}\) be finite and put \(D=D(S)\).
For every \(p\in S\), one has \(d_p\mid D\), and Lemma~\ref{lem:mersenne-lifting} gives
\[
  v_p(2^D-1)
  =
  A_p+v_p\left(\frac{D}{d_p}\right)
  \geq
  A_p.
\]
It follows that
\[
  \Omega_D
  \geq
  \sum_{p\in S}(A_p-1)\log p
  =
  \log B(S).
\]
Pass to sequences for which \(D(S)\to\infty\).
This gives
\[
  \rho_{\mathcal{W}}\geq\sigma_{\mathcal{W}}.
\]

Suppose that \(\rho_{\mathcal{W}}>0\), and choose a sequence \(m_j\) such that
\[
  \frac{\Omega_{m_j}}{m_j\log2}
  \longrightarrow
  \rho_{\mathcal{W}}.
\]
Put
\[
  S_j=\mathcal{W}(m_j)
  \qquad\text{and}\qquad
  D_j=D(S_j).
\]
Then
\[
  D_j\mid m_j
\]
and
\[
  \log B(S_j)=\Omega_{m_j}.
\]
Moreover, \(D_j\to\infty\).
Suppose otherwise.
Then, along a subsequence, all primes in \(S_j\) would have orders dividing one of finitely many bounded integers.
Such primes divide one of finitely many integers \(2^d-1\).
The sets of possible primes and the quantities \(B(S_j)\) would therefore be bounded.
This would contradict the positive limiting density.

Since \(D_j\leq m_j\), one has
\[
  \frac{\log B(S_j)}{D_j\log2}
  =
  \frac{\Omega_{m_j}}{D_j\log2}
  \geq
  \frac{\Omega_{m_j}}{m_j\log2}.
\]
Therefore
\[
  \sigma_{\mathcal{W}}\geq\rho_{\mathcal{W}}.
\]
Suppose that \(\rho_{\mathcal{W}}=0\).
Then the previously proved inequality
\[
  \rho_{\mathcal{W}}\geq\sigma_{\mathcal{W}}\geq0
\]
gives equality in this case as well.

The final assertion follows from Theorem~\ref{thm:wieferich-density-criterion}.
It may also be seen directly.
Suppose that \(\theta_{\varepsilon}<\kappa\).
Then, for all sufficiently large \(j\),
\[
  \log B(S_j)
  \geq
  \theta_{\varepsilon}D(S_j)\log2
  +
  \log\bigl(2(1-2^{-D(S_j)})\bigr).
\]
Since
\[
  \Delta_{D(S_j)}
  \geq
  \log B(S_j)
\]
Proposition~\ref{prop:mersenne-exact} applies.
\end{proof}

\begin{proposition}[Order obstruction at a transgressive exponent]\label{prop:order-obstruction}
Let \(\varepsilon>0\), and suppose that
\[
  E_{\varepsilon}(\mathcal{V}_m)\geq0.
\]
Put
\[
  S_m=\mathcal{W}(m)
  \qquad
  D_m=D(S_m)
  \qquad
  B_m=B(S_m).
\]
Then
\[
  B_m
  \geq
  \frac{2(1-2^{-m})}{m}\,
  2^{\theta_{\varepsilon}m}.
\]
Moreover, writing \(\log_2\) for the logarithm to the base \(2\), the symbol \(\log\) being reserved throughout for the natural logarithm,
\[
  D_m
  >
  \theta_{\varepsilon}m
  -
  \log_2m
  +
  \log_2\bigl(2(1-2^{-m})\bigr).
\]
In particular,
\[
  D_m
  \geq
  \bigl(\theta_{\varepsilon}-o(1)\bigr)m
\]
along every infinite sequence transgressive at exponent \(\varepsilon\).
\end{proposition}

\begin{proof}
Proposition~\ref{prop:mersenne-exact} and Corollary~\ref{cor:defect-bounds} give
\[
\begin{aligned}
  \log B_m
  =
  \Omega_m
  &\geq
  \Delta_m-\log m
  \\
  &\geq
  \theta_{\varepsilon}m\log2
  +
  \log\bigl(2(1-2^{-m})\bigr)
  -
  \log m.
\end{aligned}
\]
Exponentiation proves the first assertion.

Every \(p\in S_m\) satisfies \(d_p\mid D_m\).
Hence
\[
  B_m
  \leq
  L_{D_m}
  <
  2^{D_m}.
\]
Combining this inequality with the lower bound for \(B_m\) and taking base-\(2\) logarithms proves the second assertion.
\end{proof}

\begin{corollary}[Two quantitative alternatives]\label{cor:order-defect-dichotomy}
Two statements hold.
\begin{enumerate}[label=\textup{(\arabic*)}]
\item For every fixed \(\varepsilon>0\), only finitely many Mersenne triples are transgressive at exponent \(\varepsilon\) if and only if
\[
  \log B(S_j)=o\bigl(D(S_j)\bigr)
\]
along every admissible sequence of finite sets \(S_j\subseteq\mathcal{W}\).

\item Suppose that there are \(\kappa>0\) and finite sets \(S_j\subseteq\mathcal{W}\) such that
\[
  D(S_j)\to\infty
  \qquad\text{and}\qquad
  B(S_j)\geq2^{\kappa D(S_j)}.
\]
Then, for every \(\varepsilon>0\) satisfying \(\theta_{\varepsilon}<\kappa\), the triples \(\mathcal{V}_{D(S_j)}\) form an infinite transgressive family once finitely many repeated indices and finitely many initial terms have been discarded.
\end{enumerate}
\end{corollary}

\begin{proof}
The first assertion follows from the equality \(\rho_{\mathcal{W}}=\sigma_{\mathcal{W}}\) together with Theorem~\ref{thm:wieferich-density-criterion}.
The second follows from the final assertion of Theorem~\ref{thm:order-defect-density}.
The condition \(D(S_j)\to\infty\) ensures that infinitely many resulting exponents are distinct.
\end{proof}

The condition that reaches the \(abc\) threshold is the positive exponential growth of \(B(S)\) relative to \(D(S)\), and this requires in particular that the Wieferich primes be infinite.
Subsection~\ref{subsec:order-cofactor-constraint} converts the asymptotic lower bound of Proposition~\ref{prop:order-obstruction} for the order datum into an integrality constraint on the quotient \(m/D_m\), and gives a reduction that normalizes any infinite transgressive family.

\subsection{Exact defect laws on every fixed even base}\label{subsec:even-base-lines}

Remark~\ref{rem:general-base} admits a strengthening within the parity class itself.
Each fixed even base gives a boundary line of the class, along which the Mersenne argument runs with the same normalized conclusion.

\begin{definition}[Even-base exponential data]\label{def:even-base-data}
Fix an even integer \(g\geq2\).
For \(m\geq2\), define
\[
  \mathcal{V}^{(g)}_m
  \coloneqq
  \bigl(1,g^m-1,g^m\bigr)
\]
together with
\[
  R_g\coloneqq\rad(g)
  \qquad\text{and}\qquad
  \Delta^{(g)}_m
  \coloneqq
  \log\left(
    \frac{g^m-1}{\rad(g^m-1)}
  \right).
\]
For every prime \(p\nmid g\), put
\[
  d_{g,p}\coloneqq\ord_p(g)
  \qquad\text{and}\qquad
  A_{g,p}\coloneqq v_p(g^{d_{g,p}}-1)
\]
and define the set of base-\(g\) Wieferich primes by
\[
  \mathcal{W}_g
  \coloneqq
  \bigl\{
    p\nmid g:
    g^{p-1}\equiv1\pmod{p^2}
  \bigr\}.
\]
Finally, put
\[
  \mathcal{W}_g(m)
  \coloneqq
  \bigl\{
    p\in\mathcal{W}_g:
    d_{g,p}\mid m
  \bigr\}
\]
and
\[
  \Omega^{(g)}_m
  \coloneqq
  \sum_{p\in\mathcal{W}_g(m)}
  (A_{g,p}-1)\log p
  \qquad\text{and}\qquad
  G^{(g)}_m
  \coloneqq
  \prod_{p\mid g^m-1}p^{v_p(m)}.
\]
\end{definition}

For \(g=2\) these data reduce to those of Definition~\ref{def:mersenne-triple} and Theorem~\ref{thm:exact-defect-law}.

\begin{theorem}[Exact even-base threshold and defect law]\label{thm:even-base-exact-law}
For every even \(g\geq2\) and every \(m\geq2\), the triple \(\mathcal{V}^{(g)}_m\) lies in the parity class.
Moreover,
\[
\begin{aligned}
  E_{\varepsilon}(\mathcal{V}^{(g)}_m)
  ={}&
  (1+\varepsilon)\Delta^{(g)}_m
  -
  \varepsilon m\log g
  \\
  &-
  (1+\varepsilon)
  \log\bigl(R_g(1-g^{-m})\bigr).
\end{aligned}
\]
Thus
\[
  E_{\varepsilon}(\mathcal{V}^{(g)}_m)\geq0
\]
holds precisely when
\[
  \Delta^{(g)}_m
  \geq
  \theta_{\varepsilon}m\log g
  +
  \log\bigl(R_g(1-g^{-m})\bigr).
\]
The defect admits the exact decomposition
\[
  \Delta^{(g)}_m
  =
  \Omega^{(g)}_m+\log G^{(g)}_m
\]
in which \(G^{(g)}_m\) is the largest divisor of \(m\) all of whose prime factors divide \(g^m-1\).
In particular,
\[
  0
  \leq
  \Delta^{(g)}_m-\Omega^{(g)}_m
  \leq
  \log\left(\frac{m}{2^{v_2(m)}}\right)
  \leq
  \log m.
\]
\end{theorem}

\begin{proof}
Since \(g\) is even, the two summands \(1\) and \(g^m-1\) are odd while \(g^m\) is even, and the three entries are pairwise coprime.
Also
\[
  \rad\bigl(g^m(g^m-1)\bigr)
  =
  R_g\rad(g^m-1).
\]
Since
\[
  \log\rad(g^m-1)
  =
  m\log g+\log(1-g^{-m})-\Delta^{(g)}_m
\]
substitution into Definition~\ref{def:energy-quality} gives the first identity and its threshold form.

Let \(p\mid g^m-1\).
Then \(p\) is odd, one has \(d_{g,p}\mid m\), and the lifting identity gives
\[
  v_p(g^m-1)
  =
  A_{g,p}
  +
  v_p\left(\frac{m}{d_{g,p}}\right).
\]
Hence
\[
  v_p\left(\frac{m}{d_{g,p}}\right)=v_p(m)
\]
(since \(d_{g,p}\mid p-1\), so \(p\nmid d_{g,p}\)).
The same calculation at the exponent \(p-1\) shows that
\[
  A_{g,p}
  =
  v_p(g^{p-1}-1)
\]
so that \(A_{g,p}\geq2\) precisely when \(p\in\mathcal{W}_g\).
Summing over the primes dividing \(g^m-1\) gives
\[
\begin{aligned}
  \Delta^{(g)}_m
  &=
  \sum_{p\mid g^m-1}
  (A_{g,p}-1)\log p
  +
  \sum_{p\mid g^m-1}v_p(m)\log p
  \\
  &=
  \Omega^{(g)}_m+\log G^{(g)}_m.
\end{aligned}
\]
The description and the bounds for \(G^{(g)}_m\) follow exactly as in Theorem~\ref{thm:exact-defect-law}, the primes dividing \(g^m-1\) being coprime to \(g\) and therefore odd.
\end{proof}

\begin{remark}[The law is free of the exceptional cases]\label{rem:even-base-no-exceptions}
The exact decomposition \(\Delta^{(g)}_m=\Omega^{(g)}_m+\log G^{(g)}_m\) is independent of the theory of primitive divisors.
Its proof uses the lifting identity, the divisibility \(d_{g,p}\mid p-1\), and Definition~\ref{def:even-base-data}, all of which hold at every prime \(p\nmid g\).
The exceptional cases of Bang, Birkhoff, and Vandiver are used only in the lower bounds for \(\rad(2^m-1)\), namely in Proposition~\ref{prop:primitive-ceiling} and Lemma~\ref{lem:mersenne-radical-superlinear}.
For \(g^d-1\) the theorem fails only when \(d=1\) with \(g-1=1\), when \(d=2\) with \(g+1\) a power of two, and at the single pair \((g,d)=(2,6)\)~\cite{Bang1886,BirkhoffVandiver1904}.
An even base leaves \(g+1\) odd, so the second case would require \(g+1=1\) and never arises, while the first requires \(g=2\).
The excluded set \(\{1,6\}\) of Proposition~\ref{prop:primitive-ceiling} therefore belongs to the base \(2\) alone, and for every even \(g\geq4\) the product may be taken over all divisors of \(m\) without exception.
\end{remark}

\begin{definition}[Even-base order--defect density]\label{def:even-base-density}
Define
\[
  \rho_{\mathcal{W},g}
  \coloneqq
  \limsup_{m\to\infty}
  \frac{\Omega^{(g)}_m}{m\log g}.
\]
For a finite set \(S\subseteq\mathcal{W}_g\), put
\[
  D_g(S)
  \coloneqq
  \operatorname{lcm}_{p\in S}d_{g,p}
  \qquad\text{and}\qquad
  B_g(S)
  \coloneqq
  \prod_{p\in S}p^{A_{g,p}-1}
\]
with \(D_g(\varnothing)=B_g(\varnothing)=1\), and define
\[
  \sigma_{\mathcal{W},g}
  \coloneqq
  \sup_{\substack{(S_j)\\D_g(S_j)\to\infty}}
  \limsup_{j\to\infty}
  \frac{\log B_g(S_j)}
       {D_g(S_j)\log g}.
\]
If no such sequence exists, then the supremum is taken to be zero.
\end{definition}

\begin{theorem}[Even-base density criterion]\label{thm:even-base-density-criterion}
For every fixed even \(g\geq2\),
\[
  \rho_{\mathcal{W},g}
  =
  \sigma_{\mathcal{W},g}.
\]
For every \(\varepsilon>0\) the four assertions below hold.
\begin{enumerate}[label=\textup{(\arabic*)}]
\item If \(\theta_{\varepsilon}<\rho_{\mathcal{W},g}\), then
\[
  E_{\varepsilon}(\mathcal{V}^{(g)}_m)\geq0
\]
for infinitely many \(m\).

\item If \(\theta_{\varepsilon}>\rho_{\mathcal{W},g}\), then the same inequality holds for only finitely many \(m\).

\item Some fixed positive exponent is attained infinitely often on the base-\(g\) line if and only if \(\rho_{\mathcal{W},g}>0\).

\item Equivalently, this occurs if and only if there are a constant \(\kappa>0\) and finite sets \(S_j\subseteq\mathcal{W}_g\) such that
\[
  D_g(S_j)\longrightarrow\infty
  \qquad\text{and}\qquad
  B_g(S_j)\geq g^{\kappa D_g(S_j)}
\]
for infinitely many \(j\).
\end{enumerate}
At \(\theta_{\varepsilon}=\rho_{\mathcal{W},g}\) the density alone yields no conclusion.
\end{theorem}

\begin{proof}
Theorem~\ref{thm:even-base-exact-law} gives
\[
  0
  \leq
  \Delta^{(g)}_m-\Omega^{(g)}_m
  \leq
  \log m
\]
while the exact transgression threshold differs from \(\theta_{\varepsilon}m\log g\) by a bounded term.
The arguments of Theorem~\ref{thm:wieferich-density-criterion} therefore prove statements \textup{(1)}--\textup{(3)}.

Let \(S\subseteq\mathcal{W}_g\) be finite and put \(D=D_g(S)\).
Every \(p\in S\) divides \(g^D-1\) with valuation at least \(A_{g,p}\), so that
\[
  \Omega^{(g)}_D
  \geq
  \log B_g(S).
\]
Passage to sequences with \(D_g(S_j)\to\infty\) gives
\[
  \rho_{\mathcal{W},g}
  \geq
  \sigma_{\mathcal{W},g}.
\]

Conversely, suppose that \(\rho_{\mathcal{W},g}>0\) and choose \(m_j\to\infty\) with
\[
  \frac{\Omega^{(g)}_{m_j}}{m_j\log g}
  \longrightarrow
  \rho_{\mathcal{W},g}.
\]
Put
\[
  S_j=\mathcal{W}_g(m_j)
  \qquad\text{and}\qquad
  D_j=D_g(S_j).
\]
Then \(D_j\mid m_j\) and
\[
  \log B_g(S_j)=\Omega^{(g)}_{m_j}.
\]
Moreover, \(D_j\to\infty\).
Suppose otherwise.
Then the primes of all the sets \(S_j\) would divide one of finitely many fixed integers \(g^d-1\), and the quantities \(\log B_g(S_j)\) would remain bounded, against the positivity of the limiting density.
Therefore
\[
  \frac{\log B_g(S_j)}{D_j\log g}
  \geq
  \frac{\Omega^{(g)}_{m_j}}{m_j\log g}
\]
which gives the reverse inequality.
If \(\rho_{\mathcal{W},g}=0\), then the first inequality forces \(\sigma_{\mathcal{W},g}=0\) on its own.
Statement \textup{(4)} follows from the positivity of this common density.
\end{proof}

At \(g=2\) one has \(R_2=2\).
Theorem~\ref{thm:even-base-exact-law} then returns Proposition~\ref{prop:mersenne-exact} together with Theorem~\ref{thm:exact-defect-law}, and Theorem~\ref{thm:even-base-density-criterion} returns Theorems~\ref{thm:wieferich-density-criterion} and~\ref{thm:order-defect-density}.
Passage to a general even base adds to the threshold only the supplement \(\log(R_g/2)\), a quantity independent of \(m\), and the criterion for a fixed positive exponent is therefore insensitive to the base.
The arithmetic of \(\mathcal{W}_g\) changes with \(g\). As little is known about it as about \(\mathcal{W}\).

\subsection{A complete classification of the crossings of quality one}\label{subsec:mersenne-quality-classification}

For quality greater than one the exact law yields a description far simpler than anything transgression at a fixed exponent would require.

\begin{theorem}[Squarefreeness criterion on the Mersenne line]\label{thm:mersenne-squarefree-classification}
For every \(m\geq2\) the five conditions below are equivalent.
\begin{enumerate}[label=\textup{(\arabic*)}]
\item \(q(\mathcal{V}_m)>1\).
\item \(2^m-1\) is not squarefree.
\item \(L_m>1\).
\item \(\mathcal{W}(m)\neq\varnothing\) or \(G_m>1\).
\item There is an odd prime \(p\) such that
\[
  d_p\mid m
  \qquad\text{and}\qquad
  \bigl(p\mid m\ \text{or}\ p\in\mathcal{W}\bigr).
\]
\end{enumerate}
In particular, \(q(\mathcal{V}_m)\neq1\) for every \(m\geq2\).
If
\[
  \mathscr{H}
  \coloneqq
  \{m\geq2:q(\mathcal{V}_m)>1\}
\]
then
\[
  \mathscr{H}
  =
  \left(
    \bigcup_{\substack{p\ {\rm prime}\\p\ {\rm odd}}}
    p\,d_p\,\mathbb{N}
  \right)
  \cup
  \left(
    \bigcup_{p\in\mathcal{W}}
    d_p\,\mathbb{N}
  \right).
\]
\end{theorem}

\begin{proof}
The number \(L_m\) is a positive odd integer, since it divides \(2^m-1\).
Proposition~\ref{prop:mersenne-exact} at quality one gives
\[
  q(\mathcal{V}_m)>1
  \quad\Longleftrightarrow\quad
  L_m>2(1-2^{-m}).
\]
For \(m\geq2\) one has
\[
  1<2(1-2^{-m})<2.
\]
The last inequality therefore holds precisely when \(L_m\geq3\), which is equivalent to \(L_m>1\) and to the failure of squarefreeness of \(2^m-1\).

Theorem~\ref{thm:exact-defect-law} gives
\[
  L_m
  =
  B(\mathcal{W}(m))G_m.
\]
Both factors are positive odd integers, so that \(L_m>1\) precisely when \(\mathcal{W}(m)\) is non-empty or \(G_m>1\).

Finally, if \(p\mid2^m-1\), then \(d_p\mid m\) and
\[
  v_p(2^m-1)
  =
  A_p+v_p(m).
\]
This valuation is at least two precisely when \(A_p\geq2\) or \(p\mid m\).
By Lemma~\ref{lem:wieferich-lifting}, the first alternative reads \(p\in\mathcal{W}\).
Since \(p\nmid d_p\), the simultaneous conditions \(p\mid m\) and \(d_p\mid m\) amount to \(pd_p\mid m\).
This proves the final description.
\end{proof}

\begin{corollary}[A positive-density supply of crossings]\label{cor:mersenne-hit-density}
One has
\[
  \liminf_{X\to\infty}
  \frac{
    \#\{m\leq X:q(\mathcal{V}_m)>1\}
  }{X}
  \geq
  \frac{47}{210}.
\]
Moreover, every \(m\in\mathscr{H}\) satisfies
\[
  q(\mathcal{V}_m)-1
  >
  \frac{\log(3/2)}{(m+1)\log2}.
\]
\end{corollary}

\begin{proof}
The orders
\[
  d_3=2
  \qquad
  d_5=4
  \qquad
  d_7=3
\]
together with Theorem~\ref{thm:mersenne-squarefree-classification} show that \(\mathscr{H}\) contains every multiple of \(6\), of \(20\), and of \(21\).
Their union has density
\[
\begin{aligned}
  {}&
  \frac16+\frac1{20}+\frac1{21}
  -\frac1{60}-\frac1{42}-\frac1{420}
  +\frac1{420}
  \\
  &=
  \frac{47}{210}.
\end{aligned}
\]

If \(m\in\mathscr{H}\), then \(L_m\geq3\).
The exact quality identity is
\[
  q(\mathcal{V}_m)-1
  =
  \frac{
    \log L_m-\log\bigl(2(1-2^{-m})\bigr)
  }{
    \log\rad\bigl(2^m(2^m-1)\bigr)
  }.
\]
Its numerator exceeds \(\log(3/2)\), while its denominator is smaller than \((m+1)\log2\).
The claimed estimate follows.
\end{proof}

\begin{remark}[Margins available without Wieferich primes]\label{rem:repetition-sources}
Of the two mechanisms separated by Theorem~\ref{thm:mersenne-squarefree-classification}, the divisibility \(p\mid m\) is the one that operates in the absence of Wieferich primes, and every elementary instance of \(q(\mathcal{V}_m)>1\) belongs to it, \(3^2\mid2^6-1\) among them.
It yields margins of order \(1/m\) across a set of positive lower density by Corollary~\ref{cor:mersenne-hit-density}, and of order \(\log m/m\) upon the least common multiples by Theorem~\ref{thm:lcm-margin}.
A uniform margin \(\varepsilon>0\) holding for infinitely many \(m\) lies beyond it.
\end{remark}

\subsection{Prime-power exponents and the prime subline}\label{subsec:prime-power-exponents}

The exact law of Theorem~\ref{thm:exact-defect-law} splits the Mersenne line according to the arithmetic of the exponent.
On the prime powers the divisor \(G_m\) is trivial, and the defect is purely Wieferich there.

\begin{lemma}[Orders over a prime-power exponent]\label{lem:prime-power-orders}
Let \(\ell\) be an odd prime and let \(f\geq1\).
Then \(\ell\) does not divide \(2^{\ell^f}-1\).
Moreover, every prime \(p\) dividing \(2^{\ell^f}-1\) satisfies
\[
  d_p=\ell^i
  \quad\text{for some}\ i\in\{1,\dots,f\}
  \qquad\text{and}\qquad
  p\equiv1\pmod{2\ell^i}
\]
and in particular \(p\geq2\ell+1\).
In particular \(G_{\ell^f}=1\).
\end{lemma}

\begin{proof}
Fermat's little theorem gives \(2^{\ell}\equiv2\pmod{\ell}\), and induction on \(f\) gives \(2^{\ell^f}\equiv2\pmod{\ell}\).
Since \(\ell\) is odd, \(2\not\equiv1\pmod{\ell}\), so \(\ell\) does not divide \(2^{\ell^f}-1\).

Let \(p\) divide \(2^{\ell^f}-1\).
Lemma~\ref{lem:mersenne-lifting} gives \(d_p\mid\ell^f\), and \(d_p>1\) since \(2^1-1=1\), whence \(d_p=\ell^i\) with \(1\leq i\leq f\).
The order \(d_p\) divides \(p-1\), and \(p-1\) is even, so that \(2\ell^i\) divides \(p-1\) and \(p\geq2\ell^i+1\geq2\ell+1\).

No prime factor of \(\ell^f\) divides \(2^{\ell^f}-1\), so the divisor \(G_{\ell^f}\) of Theorem~\ref{thm:exact-defect-law} equals \(1\).
\end{proof}

\begin{theorem}[Exact decomposition on prime-power exponents]\label{thm:prime-power-exact}
Let \(m=\ell^f\) with \(\ell\) an odd prime and \(f\geq1\).
Then
\[
  \Delta_m=\Omega_m
  \qquad\text{and}\qquad
  L_m
  =
  \prod_{p\in\mathcal{W}(m)}p^{A_p-1}
  =
  B(\mathcal{W}(m)).
\]
Thus
\[
  E_{\varepsilon}(\mathcal{V}_m)\geq0
  \quad\Longleftrightarrow\quad
  B(\mathcal{W}(m))
  \geq
  \mathcal{A}_{\varepsilon}(m).
\]
\end{theorem}

\begin{proof}
Theorem~\ref{thm:exact-defect-law} together with the identity \(G_{\ell^f}=1\) of Lemma~\ref{lem:prime-power-orders} gives the equalities, and Proposition~\ref{prop:mersenne-exact} converts the transgression condition into the stated comparison.
\end{proof}

\begin{corollary}[Quality on the prime-power subline]\label{cor:prime-power-quality}
Let \(m=\ell^f\geq3\) with \(\ell\) an odd prime.
Then
\[
  q(\mathcal{V}_m)>1
  \quad\Longleftrightarrow\quad
  \mathcal{W}(m)\neq\varnothing
\]
and \(q(\mathcal{V}_m)<1\) whenever \(\mathcal{W}(m)=\varnothing\).
\end{corollary}

\begin{proof}
Suppose that \(\mathcal{W}(m)\) contains a prime \(p\).
Lemma~\ref{lem:prime-power-orders} gives \(p\geq2\ell+1\geq7\), so that
\[
  L_m=B(\mathcal{W}(m))\geq p^{A_p-1}\geq p\geq7>2>2(1-2^{-m})
\]
and Proposition~\ref{prop:mersenne-exact} gives \(q(\mathcal{V}_m)>1\).
Suppose that \(\mathcal{W}(m)\) is empty.
Then \(L_m=1<2(1-2^{-m})\) for \(m\geq2\), and the same criterion gives \(q(\mathcal{V}_m)<1\).
Both implications are also the case \(G_m=1\) of Theorem~\ref{thm:mersenne-squarefree-classification}.
\end{proof}

\begin{remark}[Known Wieferich primes and prime-power exponents]\label{rem:prime-power-known}
The orders of the two known Wieferich primes are
\[
  d_{1093}=364=2^2\cdot7\cdot13
  \qquad\text{and}\qquad
  d_{3511}=1755=3^3\cdot5\cdot13.
\]
Neither of these is a power of an odd prime, so neither known Wieferich prime divides any number \(2^{\ell^f}-1\).
On presently known arithmetic, Corollary~\ref{cor:prime-power-quality} accordingly places every prime-power exponent below quality one.
Were a single Mersenne triple of prime-power exponent to have quality above one, it would exhibit a Wieferich prime, at present undiscovered, whose order is a power of an odd prime.
The composite exponents of Theorem~\ref{thm:mersenne-hits} behave differently, reaching quality above one through the lifting of the fixed prime \(3\) and without any Wieferich prime.
That route is unavailable on the prime powers, since \(G_m\) is trivial there and the second source of repetition fails with it.
\end{remark}

\begin{theorem}[The order datum on prime exponents]\label{thm:prime-subline}
Let \(\ell\) be an odd prime.
Then every prime \(p\) dividing \(2^{\ell}-1\) satisfies \(d_p=\ell\) and \(p\equiv1\pmod{2\ell}\), and
\[
  \Delta_{\ell}=\Omega_{\ell}
  \qquad\text{together with}\qquad
  D(\mathcal{W}(\ell))=\ell\ \text{whenever}\ \mathcal{W}(\ell)\neq\varnothing.
\]
Define
\[
  \rho_{\mathcal{W}}^{\mathrm{pr}}
  \coloneqq
  \limsup_{\ell\to\infty}
  \frac{\Omega_{\ell}}{\ell\log2}
\]
where \(\ell\) runs through the odd primes.
Then, for every fixed \(\varepsilon>0\), the prime-exponent family \((\mathcal{V}_{\ell})\) contains infinitely many triples transgressive at exponent \(\varepsilon\) precisely when infinitely many primes \(\ell\) satisfy
\[
  B(\mathcal{W}(\ell))
  \geq
  \mathcal{A}_{\varepsilon}(\ell)
\]
and such an \(\varepsilon>0\) exists precisely when \(\rho_{\mathcal{W}}^{\mathrm{pr}}>0\).
\end{theorem}

\begin{proof}
Let \(p\) divide \(2^{\ell}-1\).
Lemma~\ref{lem:prime-power-orders} with \(f=1\) gives \(d_p=\ell\) and \(p\equiv1\pmod{2\ell}\), and the identity \(\Delta_{\ell}=\Omega_{\ell}\) is the case \(f=1\) of Theorem~\ref{thm:prime-power-exact}.
Every order over the exponent \(\ell\) equals \(\ell\), so \(D(\mathcal{W}(\ell))=\ell\) whenever the set is non-empty.
The first equivalence is that of Theorem~\ref{thm:prime-power-exact} read along prime exponents.

Suppose that \(\rho_{\mathcal{W}}^{\mathrm{pr}}>0\).
Choose \(\theta>0\) and \(\eta>0\) with \(\theta+\eta<\rho_{\mathcal{W}}^{\mathrm{pr}}\), and put \(\varepsilon=\theta/(1-\theta)\).
Infinitely many primes \(\ell\) satisfy \(\Omega_{\ell}\geq(\theta+\eta)\ell\log2\), and for all large such \(\ell\) one has \(\eta\ell\log2>\log(2(1-2^{-\ell}))\), whence \(\Delta_{\ell}=\Omega_{\ell}\) attains the threshold of Proposition~\ref{prop:mersenne-exact} at exponent \(\varepsilon\).
Conversely, transgression at a fixed exponent \(\varepsilon\) along infinitely many primes \(\ell\) gives
\[
  \Omega_{\ell}
  =
  \Delta_{\ell}
  \geq
  \theta_{\varepsilon}\ell\log2+\log\bigl(2(1-2^{-\ell})\bigr)
\]
whence \(\rho_{\mathcal{W}}^{\mathrm{pr}}\geq\theta_{\varepsilon}>0\).
\end{proof}

The condition of Theorem~\ref{thm:order-defect-density} takes a sharp form on the prime subline.
The admissible sets are the full sets \(\mathcal{W}(\ell)\), the order datum \(D\) collapses to the exponent itself, and the defect identity carries no error term.
The requirement on the Wieferich primes then reads as follows.
For one fixed \(\kappa>0\), infinitely many primes \(\ell\) must satisfy \(B(\mathcal{W}(\ell))\geq2^{\kappa\ell}\), every contributing prime being congruent to \(1\) modulo \(2\ell\).

\begin{proposition}[A certificate from one Wieferich prime of prime order]\label{prop:prime-order-certificate}
Suppose that \(p\) is a Wieferich prime whose order \(d_p\) is an odd prime \(\ell\).
Then
\[
  E_{\varepsilon}(\mathcal{V}_{\ell})\geq0
\]
for every \(\varepsilon>0\) satisfying
\[
  2(1-2^{-\ell})\,2^{\theta_{\varepsilon}\ell}
  \leq
  p^{A_p-1}
\]
and in particular for every \(\varepsilon\) with
\[
  \theta_{\varepsilon}
  \leq
  \frac{(A_p-1)\log_2p-1}{\ell}.
\]
Here \(p\equiv1\pmod{2\ell}\) and \(p^{A_p}\leq2^{\ell}-1\).
\end{proposition}

\begin{proof}
The hypothesis gives \(p^{A_p}\mid2^{d_p}-1=2^{\ell}-1\) with \(A_p\geq2\), whence \(L_{\ell}\geq p^{A_p-1}\) and \(p^{A_p}\leq2^{\ell}-1\), while Theorem~\ref{thm:prime-subline} gives \(p\equiv1\pmod{2\ell}\).
The first sufficiency is Proposition~\ref{prop:mersenne-exact}.
For the second, suppose that \(\theta_{\varepsilon}\ell\leq(A_p-1)\log_2p-1\).
Then
\[
  2(1-2^{-\ell})\,2^{\theta_{\varepsilon}\ell}
  <
  2^{1+\theta_{\varepsilon}\ell}
  \leq
  p^{A_p-1}
\]
and the first sufficiency applies.
\end{proof}

The orders \(364\) and \(1755\) of the two known Wieferich primes are composite, so that the proposition is for the present without an unconditional instance.
Its unconditional counterpart at the composite order \(364\) is Proposition~\ref{prop:wieferich-explicit-hit} below.

\subsection{Order-closed exponents and an extremal family}\label{subsec:order-closed}

Consider now the exponents whose odd prime divisors contain their own multiplicative orders modulo \(2\).

\begin{definition}[Order-closed exponent]\label{def:order-closed}
An integer \(m\geq2\) is called \emph{order-closed} if \(d_p\mid m\) for every odd prime \(p\mid m\).
\end{definition}

\begin{lemma}[Order closure and the divisor \(G_m\)]\label{lem:order-closed-G}
For every \(m\geq2\),
\[
  G_m
  \leq
  \frac{m}{2^{v_2(m)}}
\]
with equality precisely when \(m\) is order-closed.
In that case
\[
  \Delta_m
  =
  \Omega_m
  +
  \log\left(\frac{m}{2^{v_2(m)}}\right).
\]
\end{lemma}

\begin{proof}
The divisor \(G_m\) of Theorem~\ref{thm:exact-defect-law} divides the odd part of \(m\), which gives the inequality.
Equality holds precisely when \(v_p(G_m)=v_p(m)\) for every odd prime \(p\) dividing \(m\), that is, precisely when every such prime divides \(2^m-1\).
By Lemma~\ref{lem:mersenne-lifting} this last condition holds precisely when \(d_p\mid m\) for every odd prime \(p\) dividing \(m\).
The displayed identity is then Theorem~\ref{thm:exact-defect-law}.
\end{proof}

Lemma~\ref{lem:order-closed-G} gives \(G_m\) equal to the entire odd part of \(m\) on an order-closed exponent, against the trivial value it takes on the prime powers \(\ell^f\) of Lemma~\ref{lem:prime-power-orders}.
The least common multiples \(m_k\) are order-closed, and the divisor attains there its greatest size.

\begin{proposition}[Order saturation of the least common multiples]\label{prop:lcm-saturated}
For \(k\geq1\) put
\[
  m_k
  \coloneqq
  \operatorname{lcm}(1,2,\dots,k).
\]
Then every odd prime \(p\leq k\) divides \(2^{m_k}-1\), the exponent \(m_k\) is order-closed, and
\[
  \Delta_{m_k}
  =
  \Omega_{m_k}
  +
  \log\left(\frac{m_k}{2^{v_2(m_k)}}\right)
  \qquad\text{with}\qquad
  2^{v_2(m_k)}\leq k.
\]
\end{proposition}

\begin{proof}
Let \(p\leq k\) be an odd prime.
The order \(d_p\) divides \(p-1\) and is therefore at most \(k-1\).
Every positive integer not exceeding \(k\) divides \(m_k\), whence \(d_p\mid m_k\), and Lemma~\ref{lem:mersenne-lifting} gives \(p\mid2^{m_k}-1\).

Every odd prime dividing \(m_k\) is at most \(k\), so the first paragraph shows that \(m_k\) is order-closed, and Lemma~\ref{lem:order-closed-G} gives the identity.
Finally, \(2^{v_2(m_k)}\) is the largest power of \(2\) not exceeding \(k\), whence \(2^{v_2(m_k)}\leq k\).
\end{proof}

\begin{lemma}[An elementary lower bound for \(m_k\)]\label{lem:nair}
For every \(k\geq1\),
\[
  m_k\geq2^{k-1}.
\]
\end{lemma}

\begin{proof}
For \(k\geq7\) this is the inequality of Nair, whose proof proceeds through the integrals \(\int_0^1x^{m-1}(1-x)^{n-m}\,dx\) and elementary divisibility alone~\cite{Nair1982}.
The cases \(1\leq k\leq6\) are verified directly from the values \(1\), \(2\), \(6\), \(12\), \(60\), \(60\).
\end{proof}

\begin{theorem}[A quantitative margin on the least common multiples]\label{thm:lcm-margin}
For every \(k\geq3\),
\[
  q(\mathcal{V}_{m_k})>1.
\]
Moreover,
\[
  \Delta_{m_k}
  \geq
  \log m_k-\log\log_2(2m_k)
\]
and, for every \(k\geq4\),
\[
  q(\mathcal{V}_{m_k})
  \geq
  1
  +
  \frac{
    \log m_k-\log\log_2(2m_k)-\log2
  }{
    (m_k+1)\log2
  }.
\]
In particular,
\[
  q(\mathcal{V}_{m_k})-1
  \geq
  \bigl(1-o(1)\bigr)
  \frac{\log m_k}{m_k\log2}
\]
as \(k\to\infty\), and therefore as \(m_k\to\infty\), the exponents \(m_k=\operatorname{lcm}(1,\dots,k)\) being strictly increasing in \(k\).
\end{theorem}

\begin{proof}
For \(k\geq3\) the odd part of \(m_k\) is divisible by \(3\), so Proposition~\ref{prop:lcm-saturated} gives
\[
  \Delta_{m_k}
  \geq
  \log\left(\frac{m_k}{2^{v_2(m_k)}}\right)
  \geq
  \log3
  >
  \log\bigl(2(1-2^{-m_k})\bigr)
\]
and Proposition~\ref{prop:mersenne-exact} gives \(q(\mathcal{V}_{m_k})>1\).

Lemma~\ref{lem:nair} gives \(k\leq\log_2m_k+1=\log_2(2m_k)\), and \(2^{v_2(m_k)}\leq k\) then gives
\[
  \Delta_{m_k}
  \geq
  \log m_k-\log k
  \geq
  \log m_k-\log\log_2(2m_k).
\]

For every \(m\geq2\) one has the identity
\[
  q(\mathcal{V}_m)-1
  =
  \frac{
    \Delta_m-\log\bigl(2(1-2^{-m})\bigr)
  }{
    \log\rad\bigl(2^m(2^m-1)\bigr)
  }
\]
whose denominator equals \((m+1)\log2+\log(1-2^{-m})-\Delta_m\) and is therefore at most \((m+1)\log2\).
The map \(x\mapsto\log x-\log\log_2(2x)\) increases for \(x\geq1\), so for \(k\geq4\) one has \(m_k\geq12\) and
\[
  \Delta_{m_k}-\log2
  \geq
  \log m_k-\log\log_2(2m_k)-\log2
  \geq
  \log12-\log\log_2(24)-\log2
  >
  0.
\]
Since the numerator of the quality identity is at least \(\Delta_{m_k}-\log2\), it follows that
\[
  q(\mathcal{V}_{m_k})-1
  \geq
  \frac{\Delta_{m_k}-\log2}{(m_k+1)\log2}
  \geq
  \frac{\log m_k-\log\log_2(2m_k)-\log2}{(m_k+1)\log2}.
\]
Since \(\log\log_2(2m_k)=o(\log m_k)\) and \(m_k\to\infty\), the final estimate follows.
\end{proof}

One has \(\log m_k=\psi(k)\), the Chebyshev function, and by entirely elementary means the estimates of Chebyshev give \(\psi(k)\asymp k\)~\cite{Chebyshev1852}.
Along the family the margin of Theorem~\ref{thm:lcm-margin} therefore has the exact order \(\log m/m\).
Lemma~\ref{lem:nair} supplies the one inequality the argument needs, in place of the prime number theorem.
The exponents \(m_k\) are multiples of \(6\) for \(k\geq3\), so this family is a subfamily of the one in Theorem~\ref{thm:mersenne-hits}.
Measured against the height \(c=2^{m_k}\), the margin has order \(\log\log c/\log c\), well beneath the scale \((\log c\,\log\log c)^{-1/2}\) of Bright recalled in Remark~\ref{rem:lower-bound-examples}.

\begin{theorem}[Extremality among the mechanisms free of Wieferich primes]\label{thm:extremal-lcm}
For every \(m\geq2\),
\[
  \Delta_m-\Omega_m
  =
  \log G_m
  \leq
  \log\left(\frac{m}{2^{v_2(m)}}\right)
  \leq
  \log m
\]
while
\[
  \Delta_{m_k}-\Omega_{m_k}
  \geq
  \log m_k-\log\log_2(2m_k)
\]
for every \(k\geq1\).
If a sequence \((m_j)\) satisfies
\[
  \Delta_{m_j}-\log m_j\longrightarrow\infty
\]
then \(\Omega_{m_j}\to\infty\).
The part of the Mersenne defect free of Wieferich primes is confined to the scale \(\log m\), attained by the exponents \(m_k\) within the correction \(\log\log_2(2m_k)\).
\end{theorem}

\begin{proof}
The equality and the first two inequalities are Theorem~\ref{thm:exact-defect-law} and Lemma~\ref{lem:order-closed-G}.
The lower bound along the exponents \(m_k\) is the second display of Theorem~\ref{thm:lcm-margin} together with \(\Delta_{m_k}-\Omega_{m_k}=\log(m_k2^{-v_2(m_k)})\geq\log(m_k/k)\).
For the final claim,
\[
  \Omega_{m_j}
  =
  \Delta_{m_j}-\log G_{m_j}
  \geq
  \Delta_{m_j}-\log m_j
  \longrightarrow
  \infty.
\]
\end{proof}

\begin{corollary}[Quality margins beyond the extremal scale]\label{cor:superlog-wieferich}
Let \((m_j)\) be an increasing sequence with
\[
  \liminf_{j\to\infty}
  \frac{m_j\bigl(q(\mathcal{V}_{m_j})-1\bigr)\log2}{\log m_j}
  >
  1.
\]
Then \(\Omega_{m_j}\to\infty\).
A Mersenne quality margin beyond the extremal scale of Theorem~\ref{thm:lcm-margin} therefore requires an unbounded Wieferich contribution, while the margin of that theorem itself is produced with none.
\end{corollary}

\begin{proof}
Choose \(C\) with \(1<C<\liminf_jm_j(q(\mathcal{V}_{m_j})-1)\log2/\log m_j\) and \(\eta\in(0,1)\) with \(C(1-\eta)>1\).
As in the proof of Theorem~\ref{thm:lcm-margin}, write
\[
  q(\mathcal{V}_m)-1
  =
  \frac{\Delta_m-\log\bigl(2(1-2^{-m})\bigr)}{(m+1)\log2+\log(1-2^{-m})-\Delta_m}.
\]
Let \(j\) be large, and abbreviate \(m=m_j\), so that \((q(\mathcal{V}_m)-1)\,m\log2\geq C\log m\).
Suppose first that \(\Delta_m\geq\eta m\log2\).
Theorem~\ref{thm:extremal-lcm} then gives \(\Omega_m\geq\eta m\log2-\log m\).
Suppose instead that \(\Delta_m<\eta m\log2\).
For \(m\geq2\) one has \(\log(1-2^{-m})\geq\log(3/4)\), so the denominator above exceeds \((1-\eta)m\log2\), whence
\[
  \Delta_m
  \geq
  \bigl(q(\mathcal{V}_m)-1\bigr)(1-\eta)m\log2
  \geq
  C(1-\eta)\log m
\]
and Theorem~\ref{thm:extremal-lcm} gives \(\Omega_m\geq\bigl(C(1-\eta)-1\bigr)\log m\).
In either case \(\Omega_{m_j}\to\infty\).
\end{proof}

\begin{proposition}[A divisor ceiling from primitive prime divisors]\label{prop:primitive-ceiling}
For every \(m\geq2\),
\[
  \rad(2^m-1)
  \geq
  \prod_{\substack{d\mid m\\d\notin\{1,6\}}}
  (d+1)
\]
and hence
\[
  \Omega_m
  \leq
  \Delta_m
  \leq
  m\log2
  -
  \sum_{\substack{d\mid m\\d\notin\{1,6\}}}
  \log(d+1).
\]
The empty product is read as \(1\) and the empty sum as \(0\), which covers the divisor sets reduced by the two exclusions.
\end{proposition}

\begin{proof}
By the theorem of Bang, proved independently by Zsigmondy for \(a^n-b^n\) and recovered by Birkhoff and Vandiver, the number \(2^d-1\) possesses for every \(d\notin\{1,6\}\) a \emph{primitive} prime divisor \(p_d\), that is, a prime dividing \(2^d-1\) and no \(2^{d'}-1\) with \(1\leq d'<d\)~\cite{Bang1886,Zsigmondy1892,BirkhoffVandiver1904}. For the base \(2\) such primes are commonly called Zsigmondy primes.
The two excluded divisors are the exceptions of the theorem itself, for \(2^1-1=1\) has no prime divisor at all, while \(2^6-1=63=3^2\cdot7\) has none that is primitive, since \(3\) divides \(2^2-1\) and \(7\) divides \(2^3-1\).
For the base \(2\) these are the only exceptions, \(d=6\) being the sole nontrivial one.

The order of \(2\) modulo \(p_d\) equals \(d\), since a primitive divisor divides no smaller \(2^{d'}-1\).
If \(d\neq d'\) then \(\ord_{p_d}(2)=d\neq d'=\ord_{p_{d'}}(2)\), and since a prime has only one order modulo which \(2\) is considered, the primes \(p_d\) attached to distinct divisors are pairwise distinct.
Their product therefore divides the radical without repetition, and \(d\mid p_d-1\) gives \(p_d\geq d+1\).

For every \(d\mid m\) the prime \(p_d\) divides \(2^d-1\), which divides \(2^m-1\).
Multiplication over the divisors of \(m\) other than \(1\) and \(6\) gives the first inequality.
The second follows from
\[
  \Delta_m
  =
  \log(2^m-1)-\log\rad(2^m-1)
  \leq
  m\log2-\log\rad(2^m-1)
\]
together with Corollary~\ref{cor:defect-bounds}.
\end{proof}

The whole defect, Wieferich part included, is bounded in Proposition~\ref{prop:primitive-ceiling} by a quantity that diminishes as the divisor count \(\tau(m)\) grows.
On a prime exponent \(\ell\) this yields only the weak constraint \(\Delta_{\ell}\leq\ell\log2-\log(\ell+1)\), whereas on highly composite exponents the subtracted divisor sum grows with the number of divisors.
The estimate is unconditional and elementary, and it has no bearing upon the density criterion of Theorem~\ref{thm:wieferich-density-criterion}.
The elementary divisor bound \(\tau(m)=m^{o(1)}\) shows the subtracted sum to be \(m^{o(1)}\log m\), which is \(o(m)\).

\subsection{Order cores and the bounded order-cofactor}\label{subsec:order-cofactor-constraint}

As Remark~\ref{rem:order-divides-m} observes, every prime \(p\mid2^m-1\) has order \(d_p\mid m\), so that the least common multiple of the orders of the Wieferich primes occurring at \(m\) divides the exponent.
This divisibility singles out a canonical core within each transgressive exponent, bounds the quotient by that core by a finite range of integers, and lets every infinite transgressive family be taken down to the exponents that coincide with their own cores.

\begin{definition}[Order core and order-cofactor]\label{def:order-cofactor}
For \(m\geq2\), put
\[
  D_m
  \coloneqq
  D(\mathcal{W}(m))
  =
  \operatorname{lcm}_{p\in\mathcal{W}(m)}d_p
  \qquad\text{and}\qquad
  B_m
  \coloneqq
  B(\mathcal{W}(m))
  =
  \prod_{p\in\mathcal{W}(m)}p^{A_p-1}.
\]
The conventions of Definition~\ref{def:order-defect-data} give \(D_m=B_m=1\) when \(\mathcal{W}(m)=\varnothing\).
Since \(D_m\mid m\), define the \emph{order-cofactor}
\[
  f_m
  \coloneqq
  \frac{m}{D_m}
  \in\mathbb{Z}_{\geq1}.
\]
An exponent \(m\) is called \emph{Wieferich-supported} when \(\mathcal{W}(m)\neq\varnothing\), and \emph{order-saturated} when \(D_m=m\), equivalently when \(f_m=1\).
\end{definition}

If \(\mathcal{W}(m)\) is empty, the conventions give \(D_m=1\) and hence \(f_m=m\).
The value is a convention, adopted so that \(m=f_mD_m\) holds for every \(m\). It carries no arithmetic information.

\begin{lemma}[Idempotence of the order core]\label{lem:order-core-idempotence}
If \(m\) is Wieferich-supported, then
\[
  \mathcal{W}(D_m)=\mathcal{W}(m)
  \qquad
  D_{D_m}=D_m
  \qquad
  B_{D_m}=B_m.
\]
In particular, \(D_m\) is order-saturated.
\end{lemma}

\begin{proof}
Let \(p\in\mathcal{W}(m)\).
By the definition of \(D_m\), its order \(d_p\) divides \(D_m\), so that \(p\in\mathcal{W}(D_m)\).
Conversely, if \(p\in\mathcal{W}(D_m)\), then
\[
  d_p\mid D_m\mid m
\]
and hence \(p\in\mathcal{W}(m)\).
Thus the two Wieferich sets are equal.
Taking the least common multiple of their orders and the product of their intrinsic defect contributions gives the remaining identities.
\end{proof}

\begin{proposition}[Defect and radical bounds at the order core]\label{prop:order-core-bounds}
Let \(\varepsilon>0\), and suppose that \(m\) is Wieferich-supported and
\[
  E_{\varepsilon}(\mathcal{V}_m)\geq0.
\]
Write
\[
  D=D_m
  \qquad
  f=f_m
  \qquad
  m=fD.
\]
Then
\[
  B_m
  \geq
  \frac{2(1-2^{-fD})}{fD}\,
  2^{\theta_{\varepsilon}fD}
  \qquad\text{and}\qquad
  B_m\leq L_D.
\]
From these one obtains
\[
  \frac{\log B_m}{D\log2}
  \geq
  f\theta_{\varepsilon}
  -
  \frac{
    \log_2(fD)
    -
    \log_2\bigl(2(1-2^{-fD})\bigr)
  }{D}
\]
and
\[
  \rad(2^D-1)
  \leq
  \frac{fD}{2}\,
  \frac{1-2^{-D}}{1-2^{-fD}}\,
  2^{(1-f\theta_{\varepsilon})D}.
\]
If \(f\theta_{\varepsilon}=1\), then \(f>1\) and the last estimate becomes
\[
  \rad(2^D-1)<\frac{fD}{2}.
\]
\end{proposition}

\begin{proof}
The first inequality is Proposition~\ref{prop:order-obstruction}, rewritten with \(m=fD\).

For every \(p\in\mathcal{W}(m)\) one has \(d_p\mid D\), and Lemma~\ref{lem:mersenne-lifting} gives
\[
  v_p(2^D-1)
  =
  A_p+v_p\left(\frac{D}{d_p}\right)
  \geq
  A_p.
\]
It follows that
\[
  \prod_{p\in\mathcal{W}(m)}p^{A_p-1}
  \quad\text{divides}\quad
  \frac{2^D-1}{\rad(2^D-1)}
\]
and hence \(B_m\leq L_D\).
Taking logarithms in the first inequality gives the normalized defect estimate.

Combining the lower bound for \(B_m\) with
\[
  B_m
  \leq
  L_D
  =
  \frac{2^D-1}{\rad(2^D-1)}
\]
gives
\[
\begin{aligned}
  \rad(2^D-1)
  &\leq
  \frac{
    (2^D-1)fD
  }{
    2(1-2^{-fD})2^{\theta_{\varepsilon}fD}
  }
  \\
  &=
  \frac{fD}{2}\,
  \frac{1-2^{-D}}{1-2^{-fD}}\,
  2^{(1-f\theta_{\varepsilon})D}.
\end{aligned}
\]
If \(f\theta_{\varepsilon}=1\), then \(f>1\), since \(0<\theta_{\varepsilon}<1\).
In this case \(fD>D\), so that
\[
  1-2^{-D}<1-2^{-fD}
\]
which gives the strict final bound.
\end{proof}

\begin{theorem}[Reduction to the order core]\label{thm:order-core-descent}
Let \(\varepsilon>0\), and suppose that \(m\) is Wieferich-supported and
\[
  E_{\varepsilon}(\mathcal{V}_m)\geq0.
\]
Put \(D=D_m\).
If \(D=m\), then the exponent is order-saturated and no reduction is required.
If \(D<m\) and
\[
  2^{\theta_{\varepsilon}(m-D)}
  \geq
  m\,
  \frac{1-2^{-D}}{1-2^{-m}}
\]
then
\[
  E_{\varepsilon}(\mathcal{V}_D)\geq0.
\]
In particular, the simpler condition
\[
  2^{\theta_{\varepsilon}(m-D)}\geq m
\]
is sufficient.

The two statements below are then equivalent, for every fixed \(\varepsilon>0\).
\begin{enumerate}[label=\textup{(\arabic*)}]
\item There are infinitely many \(m\) such that
\[
  E_{\varepsilon}(\mathcal{V}_m)\geq0.
\]
\item There are infinitely many order-saturated \(n\) such that
\[
  E_{\varepsilon}(\mathcal{V}_n)\geq0.
\]
\end{enumerate}
Thus any infinite transgressive Mersenne family may be replaced, at the same exponent \(\varepsilon\), by an infinite family satisfying \(D_n=n\).
\end{theorem}

\begin{proof}
Proposition~\ref{prop:order-core-bounds} gives
\[
  L_D
  \geq
  B_m
  \geq
  \frac{2(1-2^{-m})}{m}\,
  2^{\theta_{\varepsilon}m}.
\]
Under the displayed hypothesis,
\[
\begin{aligned}
  L_D
  &\geq
  2(1-2^{-m})2^{\theta_{\varepsilon}D}
  \frac{2^{\theta_{\varepsilon}(m-D)}}{m}
  \\
  &\geq
  2(1-2^{-D})2^{\theta_{\varepsilon}D}.
\end{aligned}
\]
Proposition~\ref{prop:mersenne-exact} therefore gives \(E_{\varepsilon}(\mathcal{V}_D)\geq0\).
Since \(D<m\) gives
\[
  \frac{1-2^{-D}}{1-2^{-m}}<1
\]
the simpler condition is sufficient.

Suppose now that infinitely many indices \(m\) are transgressive at exponent \(\varepsilon\).
These indices are unbounded, and Proposition~\ref{prop:order-obstruction} gives
\[
  D_m
  >
  \theta_{\varepsilon}m
  -
  \log_2m
  +
  \log_2\bigl(2(1-2^{-m})\bigr)
\]
so that \(D_m\to\infty\) along the transgressive indices.
In particular, all sufficiently large such indices are Wieferich-supported.

If \(f_m=1\), then \(D_m=m\) and no reduction is needed.
If \(f_m\geq2\), then
\[
  D_m\leq\frac{m}{2}
  \qquad\text{and hence}\qquad
  m-D_m\geq\frac{m}{2}.
\]
For all sufficiently large \(m\) one has
\[
  2^{\theta_{\varepsilon}m/2}\geq m
\]
so that the sufficient condition applies and \(\mathcal{V}_{D_m}\) is transgressive at the same exponent.
The values \(D_m\) are unbounded, and Lemma~\ref{lem:order-core-idempotence} shows every resulting core to be order-saturated.
This proves \textup{(1)}$\Rightarrow$\textup{(2)}, while the reverse implication is immediate.
\end{proof}

From an infinite family the endpoint \(f\theta_{\varepsilon}=1\) may be excluded by an elementary radical estimate.

\begin{lemma}[Superlinear radical growth for Mersenne numbers]\label{lem:mersenne-radical-superlinear}
For every integer \(n>36\),
\[
  \rad(2^n-1)>n^{3/2}.
\]
\end{lemma}

\begin{proof}
Suppose first that \(n\) is composite.
Let \(q\) be its least prime divisor and put
\[
  d=\frac{n}{q}.
\]
Then \(q\leq\sqrt n\), and therefore
\[
  \sqrt n\leq d<n.
\]
Since \(n>36\), one has \(d>6\), so that both \(n\) and \(d\) lie outside the exceptional set \(\{1,6\}\) of the primitive-divisor theorem employed in Proposition~\ref{prop:primitive-ceiling}.
Hence \(2^n-1\) and \(2^d-1\) possess primitive prime divisors \(p_n\) and \(p_d\), whose orders are \(n\) and \(d\) respectively, so that the two primes are distinct and
\[
  p_n\geq n+1
  \qquad\text{and}\qquad
  p_d\geq d+1.
\]
Both divide \(2^n-1\), since \(2^d-1\) divides \(2^n-1\), and therefore
\[
  \rad(2^n-1)
  \geq
  (n+1)(d+1)
  >
  n^{3/2}.
\]

Suppose next that \(n\) is prime.
By Theorem~\ref{thm:prime-subline}, every prime \(p\mid2^n-1\) has order exactly \(n\), and therefore
\[
  p\equiv1\pmod{2n}.
\]
Observe first that \(2^n-1\) cannot be a proper prime power.
Suppose that
\[
  2^n-1=p^a
  \qquad(a\geq2).
\]
If \(a\) is even, then \(p^a\) is an odd square and hence congruent to \(1\) modulo \(8\), whereas \(2^n-1\equiv7\pmod8\).
Thus \(a\) is odd and at least three.
Then
\[
  2^n
  =
  p^a+1
  =
  (p+1)
  \bigl(
    p^{a-1}-p^{a-2}+\cdots-p+1
  \bigr).
\]
The second factor is a sum of an odd number of odd terms, hence odd, and it exceeds one.
This contradicts the fact that the product is a power of two.
The two congruence arguments above settle the case without appeal to Catalan's theorem.

If \(2^n-1\) is prime, then
\[
  \rad(2^n-1)=2^n-1>n^2.
\]
If it is composite, then it carries two distinct prime divisors \(p_1<p_2\).
Each is congruent to \(1\) modulo \(2n\), hence of the form \(1+2nk\) with \(k\geq1\), and distinctness forces distinct values of \(k\), so that \(p_1\geq2n+1\) and \(p_2\geq4n+1\).
It follows that
\[
  \rad(2^n-1)
  \geq
  p_1p_2
  \geq
  (2n+1)(4n+1)
  >
  8n^2.
\]
In either case the asserted estimate follows.
\end{proof}

\begin{corollary}[Exclusion of the critical cofactor]\label{cor:critical-cofactor-exclusion}
Fix \(\varepsilon>0\) and an integer \(f\geq2\) satisfying
\[
  f\theta_{\varepsilon}=1.
\]
If \(m=fD\) is transgressive at exponent \(\varepsilon\) and \(D=D_m\), then
\[
  D\leq36
  \qquad\text{or}\qquad
  D<\frac{f^2}{4}.
\]
In particular, no infinite transgressive sequence can carry a fixed cofactor \(f\) satisfying \(f\theta_{\varepsilon}=1\).
\end{corollary}

\begin{proof}
Proposition~\ref{prop:order-core-bounds} gives
\[
  \rad(2^D-1)<\frac{fD}{2}.
\]
If \(D>36\), then Lemma~\ref{lem:mersenne-radical-superlinear} gives
\[
  D^{3/2}
  <
  \frac{fD}{2}.
\]
Division by \(D>0\) leaves
\[
  \sqrt{D}<\frac{f}{2}
\]
and squaring gives \(D<f^2/4\).
The two alternatives confine \(D\), and therefore confine \(m=fD\), to a finite range depending upon \(f\) alone.
\end{proof}

\begin{theorem}[Rigid bound on the order-cofactor]\label{thm:bounded-order-cofactor}
Fix \(\varepsilon>0\), and recall from Proposition~\ref{prop:order-obstruction} the base-\(2\) logarithm \(\log_2\).
For a transgressive index \(m\), define
\[
  r_{\varepsilon}(m)
  \coloneqq
  \theta_{\varepsilon}m
  -
  \log_2m
  +
  \log_2\bigl(2(1-2^{-m})\bigr).
\]
Whenever \(r_{\varepsilon}(m)>0\), one has the exact bound
\[
  f_m
  <
  \frac{m}{r_{\varepsilon}(m)}
\]
and therefore
\[
  f_m
  \leq
  \left\lceil
    \frac{m}{r_{\varepsilon}(m)}
  \right\rceil-1.
\]

More explicitly, put
\[
  N_{\varepsilon}
  \coloneqq
  \left\lfloor\frac{1}{\theta_{\varepsilon}}\right\rfloor
  \qquad\text{and}\qquad
  \delta_{\varepsilon}
  \coloneqq
  \theta_{\varepsilon}-\frac{1}{N_{\varepsilon}+1}>0
\]
together with
\[
  M_{\varepsilon}^{(0)}
  \coloneqq
  \left\lceil
    \max\left\{
      4,\,
      N_{\varepsilon}+1,\,
      \frac{2}{\delta_{\varepsilon}}
      \log_2\left(\frac{2}{\delta_{\varepsilon}}\right)
    \right\}
  \right\rceil.
\]
Every transgressive index \(m\geq M_{\varepsilon}^{(0)}\) satisfies \(f_m\leq N_{\varepsilon}\).

The endpoint may be removed.
Define
\[
  F_{\varepsilon}
  \coloneqq
  \left\lceil\frac{1}{\theta_{\varepsilon}}\right\rceil-1
  =
  \left\lceil1+\frac{1}{\varepsilon}\right\rceil-1.
\]
If \(1/\theta_{\varepsilon}\notin\mathbb{Z}\), then put \(M_{\varepsilon}=M_{\varepsilon}^{(0)}\).
If \(1/\theta_{\varepsilon}=N_{\varepsilon}\in\mathbb{Z}\), then put
\[
  C_{\varepsilon}
  \coloneqq
  \max\left\{
    36,\,
    \frac{N_{\varepsilon}^2}{4}
  \right\}
  \qquad\text{and}\qquad
  M_{\varepsilon}
  \coloneqq
  \max\left\{
    M_{\varepsilon}^{(0)},\,
    \left\lfloor
      N_{\varepsilon}C_{\varepsilon}
    \right\rfloor+1
  \right\}.
\]
Then every \(m\geq M_{\varepsilon}\) satisfying
\[
  E_{\varepsilon}(\mathcal{V}_m)\geq0
\]
obeys the strict integral constraint
\[
  1\leq f_m\leq F_{\varepsilon}
  <
  \frac{1}{\theta_{\varepsilon}}.
\]
In particular, every such \(m\) is Wieferich-supported.
\end{theorem}

\begin{proof}
Throughout, \(m\) is transgressive at exponent \(\varepsilon\), and \(f_m=m/D_m\) is the order-cofactor of Definition~\ref{def:order-cofactor}, so that \(m=f_mD_m\) holds whether or not \(\mathcal{W}(m)\) is empty.

\medskip
\noindent\emph{The exact bound.}\quad
Proposition~\ref{prop:order-obstruction} gives \(D_m>r_{\varepsilon}(m)\).
If \(r_{\varepsilon}(m)>0\), then dividing by \(m\) yields
\[
  \frac{1}{f_m}
  =
  \frac{D_m}{m}
  >
  \frac{r_{\varepsilon}(m)}{m}
  >
  0.
\]
Both sides being positive, taking reciprocals reverses the inequality and gives \(f_m<m/r_{\varepsilon}(m)\).
Since \(f_m\) is an integer, it is at most the largest integer strictly below \(m/r_{\varepsilon}(m)\), which is \(\lceil m/r_{\varepsilon}(m)\rceil-1\).

\medskip
\noindent\emph{The bound by \(N_{\varepsilon}\).}\quad
Since \(\theta_{\varepsilon}=\varepsilon/(1+\varepsilon)<1\), one has \(1/\theta_{\varepsilon}>1\) and therefore \(N_{\varepsilon}\geq1\).
The floor in the definition of \(N_{\varepsilon}\) places \(\theta_{\varepsilon}\) strictly above \(1/(N_{\varepsilon}+1)\), so that \(\delta_{\varepsilon}>0\), and from \(\theta_{\varepsilon}\leq1/N_{\varepsilon}\) it follows that
\[
  0 < \delta_{\varepsilon} \leq \frac{1}{N_{\varepsilon}} - \frac{1}{N_{\varepsilon}+1} \leq \frac12.
\]

Put
\[
  y=\frac{2}{\delta_{\varepsilon}}
  \qquad\text{and}\qquad
  x_0=y\log_2y.
\]
Then \(\delta_{\varepsilon}\leq1/2\) gives \(y\geq4\), hence \(\log_2y\geq2\) and \(x_0\geq8\).
Since \(\log_2y\geq1\),
\[
\begin{aligned}
  \log_2x_0
  &=
  \log_2y+\log_2\log_2y
  \\
  &\leq
  2\log_2y
  =
  \delta_{\varepsilon}x_0
\end{aligned}
\]
the middle step using \(\log_2\log_2y\leq\log_2y\), valid since \(y\geq4\).

The map \(x\mapsto(\log_2x)/x\) is decreasing for \(x\geq e\), and the definition of \(M_{\varepsilon}^{(0)}\) gives \(M_{\varepsilon}^{(0)}\geq x_0\geq8\).
Hence, for every \(m\geq M_{\varepsilon}^{(0)}\),
\[
  \frac{\log_2m}{m}
  \leq
  \frac{\log_2x_0}{x_0}
  \leq
  \delta_{\varepsilon}
\]
the second inequality being the display above divided by \(x_0\).
Moreover
\[
  \log_2\bigl(2(1-2^{-m})\bigr)>0
\]
for every \(m\geq2\), since \(2(1-2^{-m})\geq3/2>1\) there.
Dividing the second assertion of Proposition~\ref{prop:order-obstruction} by \(m\) and discarding this positive term,
\[
\begin{aligned}
  \frac{1}{f_m}
  =
  \frac{D_m}{m}
  &>
  \theta_{\varepsilon}-\frac{\log_2m}{m}
  \\
  &\geq
  \theta_{\varepsilon}-\delta_{\varepsilon}
  =
  \frac{1}{N_{\varepsilon}+1}.
\end{aligned}
\]
It follows that \(f_m<N_{\varepsilon}+1\), and the integrality of \(f_m\) gives \(f_m\leq N_{\varepsilon}\).

\medskip
\noindent\emph{Removal of the endpoint.}\quad
If \(1/\theta_{\varepsilon}\) is not an integer, then \(F_{\varepsilon}=\lceil1/\theta_{\varepsilon}\rceil-1=N_{\varepsilon}\) and \(M_{\varepsilon}=M_{\varepsilon}^{(0)}\), so the preceding estimate is itself the assertion.
Suppose instead that
\[
  \frac{1}{\theta_{\varepsilon}}=N_{\varepsilon}\in\mathbb{Z}
\]
in which case \(F_{\varepsilon}=N_{\varepsilon}-1\) and the one value remaining to exclude is \(f_m=N_{\varepsilon}\).
Here \(\theta_{\varepsilon}<1\) forces \(N_{\varepsilon}\geq2\).

Suppose that \(f_m=N_{\varepsilon}\) for some \(m\geq M_{\varepsilon}\) satisfying \(E_{\varepsilon}(\mathcal{V}_m)\geq0\).
Then
\[
  f_m\theta_{\varepsilon}=1
  \qquad\text{and}\qquad
  D_m=\frac{m}{N_{\varepsilon}}.
\]
The definition of \(M_{\varepsilon}\) gives \(M_{\varepsilon}\geq\lfloor N_{\varepsilon}C_{\varepsilon}\rfloor+1>N_{\varepsilon}C_{\varepsilon}\), and therefore
\[
  D_m
  =
  \frac{m}{N_{\varepsilon}}
  \geq
  \frac{M_{\varepsilon}}{N_{\varepsilon}}
  >
  C_{\varepsilon}
  =
  \max\left\{
    36,\,
    \frac{N_{\varepsilon}^2}{4}
  \right\}.
\]
Corollary~\ref{cor:critical-cofactor-exclusion}, applied with \(f=N_{\varepsilon}\geq2\) and \(D=D_m\), gives \(D_m\leq36\) or \(D_m<N_{\varepsilon}^2/4\), and each alternative contradicts the display above.
Therefore \(f_m\neq N_{\varepsilon}\), and with the bound of the preceding part,
\[
  f_m\leq N_{\varepsilon}-1=F_{\varepsilon}.
\]

\medskip
\noindent\emph{Wieferich support and the strict inequality.}\quad
In both cases \(M_{\varepsilon}\geq M_{\varepsilon}^{(0)}\geq N_{\varepsilon}+1>F_{\varepsilon}\).
If \(\mathcal{W}(m)\) were empty for some transgressive \(m\geq M_{\varepsilon}\), then Definition~\ref{def:order-cofactor} would give \(D_m=1\) and \(f_m=m\geq M_{\varepsilon}>F_{\varepsilon}\), against the bound just proved.
Hence \(\mathcal{W}(m)\neq\varnothing\).
The lower bound \(f_m\geq1\) is part of Definition~\ref{def:order-cofactor}, since \(D_m\mid m\).
Finally \(F_{\varepsilon}=\lceil1/\theta_{\varepsilon}\rceil-1<1/\theta_{\varepsilon}\), since \(\lceil x\rceil-1<x\) for every \(x>0\), which is the strict inequality closing the displayed chain.
\end{proof}

\begin{corollary}[The order-cofactor and exponent amplification]\label{cor:lattice-confinement}
Fix \(\varepsilon>0\), and let \((m_j)\) be an increasing sequence satisfying
\[
  E_{\varepsilon}(\mathcal{V}_{m_j})\geq0.
\]
After finitely many terms have been discarded,
\[
  f_{m_j}
  \in
  \{1,2,\dots,F_{\varepsilon}\}.
\]
There are then an integer \(f\in\{1,2,\dots,F_{\varepsilon}\}\) and an infinite subsequence, again written \((m_j)\), such that
\[
  m_j=fD_j
  \qquad\text{and}\qquad
  D_j=D_{m_j}\longrightarrow\infty.
\]
Along this subsequence,
\[
  \mathcal{W}(D_j)=\mathcal{W}(m_j)
  \qquad\text{and}\qquad
  \liminf_{j\to\infty}
  \frac{\log B_{m_j}}{D_j\log2}
  \geq
  f\theta_{\varepsilon}
\]
with \(f\theta_{\varepsilon}<1\).
For every \(\varepsilon'>0\) satisfying
\[
  \theta_{\varepsilon'}
  <
  f\theta_{\varepsilon}
\]
the core triples \(\mathcal{V}_{D_j}\) are transgressive at exponent \(\varepsilon'\) for all sufficiently large \(j\).
In particular,
\[
  \liminf_{j\to\infty}q(\mathcal{V}_{D_j})
  \geq
  \frac{1}{1-f\theta_{\varepsilon}}.
\]
If \(f>1\), then the core reduction amplifies the original exponent, since one may choose \(\varepsilon'>\varepsilon\).
\end{corollary}

\begin{proof}
The finite range for \(f_{m_j}\) is Theorem~\ref{thm:bounded-order-cofactor}, and an infinite subsequence with constant cofactor follows from the pigeonhole principle.
Since \(m_j\to\infty\) and \(m_j=fD_j\), one has \(D_j\to\infty\), and Lemma~\ref{lem:order-core-idempotence} gives the equality of the Wieferich sets.

Proposition~\ref{prop:order-core-bounds} gives
\[
  \frac{\log B_{m_j}}{D_j\log2}
  \geq
  f\theta_{\varepsilon}
  -
  \frac{
    \log_2(fD_j)
    -
    \log_2\bigl(2(1-2^{-fD_j})\bigr)
  }{D_j}
\]
whose final quotient tends to zero.
Moreover,
\[
  f
  \leq
  F_{\varepsilon}
  <
  \frac{1}{\theta_{\varepsilon}}
\]
so that \(f\theta_{\varepsilon}<1\).

Now fix \(\varepsilon'\) with \(\theta_{\varepsilon'}<f\theta_{\varepsilon}\).
The positive linear gap between these two quantities dominates the logarithmic correction above, so that for all sufficiently large \(j\),
\[
  \log B_{m_j}
  \geq
  \theta_{\varepsilon'}D_j\log2
  +
  \log\bigl(2(1-2^{-D_j})\bigr).
\]
Since \(L_{D_j}\geq B_{m_j}\), Proposition~\ref{prop:mersenne-exact} gives \(E_{\varepsilon'}(\mathcal{V}_{D_j})\geq0\).
Letting \(\theta_{\varepsilon'}\) increase to \(f\theta_{\varepsilon}\) proves the quality estimate.
If \(f>1\), then
\[
  \theta_{\varepsilon}<f\theta_{\varepsilon}<1
\]
so that the monotonicity of \(\varepsilon\mapsto\theta_{\varepsilon}\) permits a choice \(\varepsilon'>\varepsilon\).
\end{proof}

The amplification of Corollary~\ref{cor:lattice-confinement} cannot be repeated indefinitely.
Its limits give the cofactor at once on the upper half of the transgressive range.

\begin{corollary}[Saturation near the critical exponent]\label{cor:saturation-near-critical}
Fix \(\varepsilon>0\) satisfying
\[
  \theta_{\varepsilon}>\frac{\rho_{\mathcal{W}}}{2}.
\]
Then all but finitely many indices \(m\) with \(E_{\varepsilon}(\mathcal{V}_m)\geq0\) are order-saturated.
\end{corollary}

\begin{proof}
Suppose that infinitely many transgressive indices satisfy \(f_m\geq2\).
Theorem~\ref{thm:bounded-order-cofactor} confines their cofactors to \(\{2,\dots,F_{\varepsilon}\}\) once finitely many indices have been discarded, so that one value \(f\geq2\) recurs infinitely often.
Corollary~\ref{cor:lattice-confinement} then furnishes, for every \(\varepsilon'\) with \(\theta_{\varepsilon'}<f\theta_{\varepsilon}\), infinitely many distinct indices transgressive at exponent \(\varepsilon'\), and Theorem~\ref{thm:wieferich-density-criterion} gives \(\theta_{\varepsilon'}\leq\rho_{\mathcal{W}}\).
Passage to the supremum over these exponents gives
\[
  f\theta_{\varepsilon}\leq\rho_{\mathcal{W}}
\]
whence \(\theta_{\varepsilon}\leq\rho_{\mathcal{W}}/2\), against the hypothesis.
\end{proof}

The range of the hypothesis is the window
\[
  \frac{\rho_{\mathcal{W}}}{2}
  <
  \theta_{\varepsilon}
  \leq
  \rho_{\mathcal{W}}
\]
together with everything above it.
The transgressive indices may be infinite in number on the window itself, and the corollary makes them order-saturated from some point onward.
Above the window, and in particular at every \(\varepsilon>0\) when \(\rho_{\mathcal{W}}=0\), Theorem~\ref{thm:wieferich-density-criterion} leaves only finitely many transgressive indices, so the saturation concerns a finite set.

\begin{remark}[Significance of the cofactor bound]\label{rem:cofactor-meaning}
Theorem~\ref{thm:bounded-order-cofactor} refines the asymptotic inequality \(D_m \geq (\theta_\varepsilon - o(1))m\) of Proposition~\ref{prop:order-obstruction} into the exact finite constraint \(f_m \in \{1, 2, \dots, F_\varepsilon\}\).
The critical value \(f_m = 1/\theta_\varepsilon\) is excluded because it would force \(\rad(2^D-1) < fD/2\), contradicting the superlinear radical growth of Lemma~\ref{lem:mersenne-radical-superlinear}.
For instance:
\begin{itemize}
\item At \(\varepsilon=1\), one has \(\theta_1 = 1/2\) and \(F_1 = \lceil 2\rceil - 1 = 1\), so every sufficiently large transgressive index is order-saturated (\(m = D_m\)).
\item At \(\varepsilon=1/40\), one has \(\theta_{1/40} = 1/41\) and \(F_{1/40} = 40\), giving \(f_m \le 40\). The transgressive index \(m=364\) in Proposition~\ref{prop:wieferich-explicit-hit} satisfies \(D_{364}=364\) and thus belongs to the saturated stratum \(f=1\).
\end{itemize}
On the prime subline (Theorem~\ref{thm:prime-subline}), every Wieferich-supported exponent \(\ell\) satisfies \(D_\ell = \ell\) and \(f_\ell = 1\).
\end{remark}

\subsection{The finite case}\label{subsec:wieferich-finite}

If \(\mathcal{W}\) is finite, the intrinsic defect \(\Omega_m\) is uniformly bounded, which precludes transgressive Mersenne triples for any fixed \(\varepsilon>0\).

\begin{theorem}[Finiteness of \(\mathcal{W}\) implies finiteness of transgressive triples]\label{thm:wieferich-finite}
Suppose that \(\mathcal{W}\) is finite, and put
\[
  C_{\mathcal{W}}
  \coloneqq
  \sum_{p\in\mathcal{W}}(A_p-1)\log p.
\]
Then \(\Delta_m\leq\log m+C_{\mathcal{W}}\) and \(L_m\leq e^{C_{\mathcal{W}}}m\) for every \(m\geq2\). Thus, for every \(\varepsilon>0\), the set of integers \(m\) satisfying \(E_{\varepsilon}(\mathcal{V}_m)\geq0\) is finite.
\end{theorem}

\begin{proof}
Since \(\mathcal{W}(m)\subseteq\mathcal{W}\), one has \(\Omega_m\leq C_{\mathcal{W}}\).
Corollary~\ref{cor:defect-bounds} gives the two displayed bounds.
Fix \(\varepsilon>0\).
Proposition~\ref{prop:mersenne-exact} shows that \(E_{\varepsilon}(\mathcal{V}_m)\geq0\) forces
\[
  e^{C_{\mathcal{W}}}m
  \geq
  L_m
  \geq
  2(1-2^{-m})\,2^{\theta_{\varepsilon}m}
  \geq
  2^{\theta_{\varepsilon}m}
\]
for \(m\geq2\).
Since \(\theta_{\varepsilon}>0\), the right-hand side grows exponentially while the left-hand side grows linearly.
The inequality can therefore hold for at most finitely many \(m\).
\end{proof}

A prime dividing \(2^m-1\) contributes to the defect quotient \(L_m\) in two ways.
There is the lifting constant \(A_p\), intrinsic to the prime and indifferent to \(m\), and there is \(v_p(m/d_p)\), which depends upon \(m\) and whose aggregate is bounded by \(m\).
A lifting constant greater than one belongs to the Wieferich primes alone.
Suppose these to be finitely many.
An absolute constant then bounds the intrinsic contribution, and the whole growth of \(L_m\) falls to the second factor, which can never overtake the exponential threshold \(\mathcal{A}_{\varepsilon}(m)\).

\begin{remark}[The two known Wieferich primes]\label{rem:known-wieferich}
For \(p=1093\) one has \(d_p=364\) and \(A_p=2\).
For \(p=3511\) one has \(d_p=1755\) and \(A_p=2\).
Suppose that these are the only Wieferich primes.
Then Theorem~\ref{thm:wieferich-finite} gives
\[
  L_m
  \leq
  1093\cdot3511\cdot m
  =
  3\,837\,523\,m
\]
for every \(m\geq2\).
Unconditionally nothing of the kind can be asserted, the finiteness of \(\mathcal{W}\) being unknown.
\end{remark}

The closure asserted in Theorem~\ref{thm:wieferich-finite} may be made effective in quantity.

\begin{theorem}[Effective bound from a bounded intrinsic defect]\label{thm:effective-wieferich-bound}
Fix \(\varepsilon>0\) and put
\[
  \lambda_{\varepsilon}
  \coloneqq
  \theta_{\varepsilon}\log2.
\]
Let \(C\geq0\).
If
\[
  E_{\varepsilon}(\mathcal{V}_m)\geq0
  \qquad\text{and}\qquad
  \Omega_m\leq C
\]
then
\[
  m
  \leq
  \frac{2}{\lambda_{\varepsilon}}
  \left(
    C+\log\frac{2}{\lambda_{\varepsilon}}
  \right).
\]
In particular, if \(\mathcal{W}\) is finite, then every transgressive Mersenne index satisfies
\[
  m
  \leq
  \frac{2}{\theta_{\varepsilon}\log2}
  \left(
    C_{\mathcal W}
    +
    \log\frac{2}{\theta_{\varepsilon}\log2}
  \right)
\]
with \(C_{\mathcal W}\) as in Theorem~\ref{thm:wieferich-finite}.
\end{theorem}

\begin{proof}
Corollary~\ref{cor:defect-bounds} and the hypothesis give
\[
  \Delta_m\leq C+\log m.
\]
On the other hand, Proposition~\ref{prop:mersenne-exact} gives
\[
  \Delta_m
  \geq
  \lambda_{\varepsilon}m
  +
  \log\bigl(2(1-2^{-m})\bigr).
\]
The final logarithm is positive for \(m\geq2\), and hence
\[
  \lambda_{\varepsilon}m
  <
  C+\log m.
\]

Put
\[
  H
  =
  C+\log\frac{2}{\lambda_{\varepsilon}}.
\]
Since \(\theta_{\varepsilon}<1\), one has \(\lambda_{\varepsilon}<\log2<2\), so that \(H>0\).
Suppose that \(m>2H/\lambda_{\varepsilon}\), and set
\[
  x=\frac{\lambda_{\varepsilon}m}{2}.
\]
Then \(x>H\), while
\[
\begin{aligned}
  C+\log m
  &=
  C+\log\frac{2x}{\lambda_{\varepsilon}}
  \\
  &=
  H+\log x
  \\
  &<
  x+x
  =
  \lambda_{\varepsilon}m
\end{aligned}
\]
where \(\log x<x\) for every \(x>0\).
This contradicts the necessary inequality established above.
\end{proof}

\begin{corollary}[Explicit form under the two-known-primes hypothesis]\label{cor:two-known-effective-bound}
Assume that
\[
  \mathcal{W}=\{1093,3511\}
\]
and that both primes have lifting constant two, as in Remark~\ref{rem:known-wieferich}.
Then every \(m\) satisfying \(E_{\varepsilon}(\mathcal{V}_m)\geq0\) obeys
\[
  m
  \leq
  \frac{2}{\theta_{\varepsilon}\log2}
  \log\left(
    \frac{7\,675\,046}
         {\theta_{\varepsilon}\log2}
  \right).
\]
\end{corollary}

\begin{proof}
Under the stated hypothesis,
\[
  C_{\mathcal W}
  =
  \log(1093\cdot3511)
  =
  \log(3\,837\,523).
\]
Substitution of this value into Theorem~\ref{thm:effective-wieferich-bound}, together with \(2\cdot3\,837\,523=7\,675\,046\), gives the bound.
\end{proof}

\begin{remark}[The probabilistic model]\label{rem:wieferich-heuristic}
Were \(\mathcal{W}\) finite, the standard probabilistic model of the Fermat quotient would be contradicted.
Under the heuristic that
\[
  \frac{2^{p-1}-1}{p}\bmod p
\]
is distributed as a random residue, the probability that a given prime \(p\) lies in \(\mathcal{W}\) is approximately \(1/p\).
Mertens' theorem then predicts
\[
  \#\bigl\{p\in\mathcal{W}:p\leq x\bigr\}
  \sim
  \log\log x.
\]
One expects \(\mathcal{W}\) to be infinite, then, and extremely sparse.
A proof of finiteness would show this probabilistic model to fail for Fermat quotients.
\end{remark}

\begin{remark}[The first case of Fermat's last theorem]\label{rem:wieferich-flt}
Wieferich's original theorem states that if the first case of Fermat's last theorem fails for a prime exponent \(p\), then \(p\) lies in \(\mathcal{W}\)~\cite{Wieferich1909}.
Were \(\mathcal{W}\) finite, the first case could have failed for finitely many exceptional primes only.
The consequence keeps a historical interest of its own, and it shows once more how the condition \(p^2\mid2^{p-1}-1\) returns across several Diophantine questions.
\end{remark}

\subsection{The infinite case}\label{subsec:wieferich-infinite}

If \(\mathcal{W}\) is infinite, the Mersenne criterion of Proposition~\ref{prop:mersenne-exact} remains applicable.
Theorem~\ref{thm:order-defect-density} requires positive exponential order--defect growth.

\begin{proposition}[The Wieferich obstruction along a transgressive sequence]\label{prop:wieferich-necessary}
Fix \(\varepsilon>0\) and let \(m\geq2\) satisfy
\[
  E_{\varepsilon}(\mathcal{V}_m)\geq0.
\]
Then
\[
  \prod_{p\in\mathcal{W}(m)}p^{A_p-1}
  \geq
  \frac{2(1-2^{-m})}{m}\,
  2^{\theta_{\varepsilon}m}.
\]
In particular, \(\mathcal{W}(m)\) is non-empty for every sufficiently large such \(m\).
If the set of such \(m\) is infinite, then \(\mathcal{W}\) is infinite.
\end{proposition}

\begin{proof}
Proposition~\ref{prop:mersenne-exact} gives
\[
  L_m
  \geq
  2(1-2^{-m})\,2^{\theta_{\varepsilon}m}.
\]
Corollary~\ref{cor:defect-bounds} gives
\[
  L_m
  \leq
  m\prod_{p\in\mathcal{W}(m)}p^{A_p-1}.
\]
Combining the two proves the inequality.

For the non-emptiness, note that \(\theta_{\varepsilon}>0\), so the displayed right-hand side grows exponentially and exceeds \(1\) for every \(m\) beyond some \(m_0(\varepsilon)\).
The product over \(\mathcal{W}(m)\) then exceeds \(1\), while an empty product would equal \(1\), so \(\mathcal{W}(m)\) is non-empty.

Suppose that \(\mathcal{W}\) is finite.
Then Theorem~\ref{thm:wieferich-finite} permits only finitely many such \(m\).
\end{proof}

Let \(p_1,\dots,p_k\) be Wieferich primes, each with lifting constant equal to \(2\), as happens for the two known examples.
Their combined intrinsic contribution to \(L_m\) is at most \(p_1\cdots p_k\), and it is present only when every \(p_i\) divides \(2^m-1\), which requires
\[
  \operatorname{lcm}(d_{p_1},\dots,d_{p_k})
  \mid m.
\]
Proposition~\ref{prop:wieferich-necessary} then asks the product \(p_1\cdots p_k\) to be exponentially large in \(m\), the factor \(m\) apart, while \(m\) must at the same time be a multiple of the least common multiple of the orders.
No mechanism is known that would keep this least common multiple small against the logarithm of the product of the primes.

\begin{proposition}[A transgressive triple from \(1093\)]\label{prop:wieferich-explicit-hit}
The Mersenne triple
\[
  \mathcal{V}_{364}
  =
  \bigl(1,2^{364}-1,2^{364}\bigr)
\]
satisfies
\[
  E_{1/40}(\mathcal{V}_{364})>0.
\]
Equivalently,
\[
  q(\mathcal{V}_{364})>\frac{41}{40}.
\]
\end{proposition}

\begin{proof}
As stated in Remark~\ref{rem:known-wieferich}, one has
\[
  d_{1093}=364
  \qquad\text{and}\qquad
  A_{1093}=2.
\]
Thus
\[
  1093^2\mid2^{364}-1
\]
and therefore
\[
  L_{364}\geq1093.
\]
For \(\varepsilon=1/40\), one has
\[
  \theta_{\varepsilon}=\frac{1}{41}.
\]
The Mersenne threshold satisfies
\[
\begin{aligned}
  \mathcal{A}_{1/40}(364)
  &=
  2(1-2^{-364})\,2^{364/41}
  \\
  &<
  2^{1+364/41}
  =
  2^{405/41}
  \\
  &<
  2^{10}
  =
  1024
  <
  1093.
\end{aligned}
\]
Proposition~\ref{prop:mersenne-exact} gives the result.
\end{proof}

Proposition~\ref{prop:wieferich-explicit-hit} gives a single transgressive triple at a fixed exponent.
An infinite family at one fixed exponent would require \(\rho_{\mathcal{W}}>0\), and whether this density vanishes is not known.

\begin{remark}[Relation to Silverman's theorem]\label{rem:silverman}
The infinitude of Wieferich primes is consistent with the \(abc\) conjecture. Silverman proved that the \(abc\) conjecture implies the existence of infinitely many non-Wieferich primes, at least \(c\log x\) of them up to \(x\)~\cite{Silverman1988}. Under the heuristic of Remark~\ref{rem:wieferich-heuristic}, non-Wieferich primes have natural density one, while \(\mathcal{W}\) has density zero despite being infinite.

Silverman's theorem bounds the non-Wieferich set from below, while an \(abc\) violation on the Mersenne line would bear on the multiplicative orders within \(\mathcal{W}\).
Individual Wieferich primes yield isolated low-radical triples, as in the construction of Granville and Tucker~\cite{GranvilleTucker2002}. Theorems~\ref{thm:wieferich-density-criterion} and~\ref{thm:order-defect-density} give a two-sided equivalence -- an infinite transgressive family at a fixed exponent \(\varepsilon>0\) exists on the Mersenne line if and only if the order--defect density \(\sigma_{\mathcal{W}}\) is strictly positive.
\end{remark}

\begin{remark}[Heuristic expectations versus positive density]\label{rem:heuristic-gap}
The standard probabilistic model of Remark~\ref{rem:wieferich-heuristic} predicts that the number of Wieferich primes up to \(x\) is of order \(\log\log x\).
For the Mersenne number \(2^m-1\), taking \(x = 2^m\) indicates that
\[
  \#\mathcal{W}(m) \le \#\{p\in\mathcal{W} : p \le 2^m\} \ll \log m
\]
meaning that the intrinsic defect \(\Omega_m\) is distributed over at most \(O(\log m)\) primes.

The condition in Theorem~\ref{thm:order-defect-density}, however, requires the stronger bound \(B(S) \geq 2^{\kappa D(S)}\) for some fixed \(\kappa>0\).
If every Wieferich prime satisfies \(A_p = 2\) (as is the case for \(1093\) and \(3511\)), and if the set \(S\) contains at most \(O(\log D)\) elements, then some prime \(p \in S\) must satisfy
\[
  \log p \geq \frac{\kappa D\log 2}{\#S} \gg \frac{D}{\log D}.
\]
Because \(d_p \mid D\) and \(p \mid 2^{d_p}-1\), this implies \(d_p \gg D / \log D\).
Under the additional assumption that the multiplicative order \(d_p\) typically has size \(p^{1-o(1)}\), one would obtain \(D(S) \ge \max_{p\in S} d_p \ge P^{1-o(1)}\) (where \(P = \max S\)), whereas \(\log B(S) \le (\log\log P)(\log P)\).
The ratio \(\log B(S) / D(S)\) would then tend to zero as \(P \to \infty\), yielding \(\sigma_{\mathcal{W}} = 0\) and \(\rho_{\mathcal{W}} = 0\).

This heuristic deduction relies on the hypothesis that the lifting constants \(A_p\) remain bounded, or grow more slowly than \(\log p\). If there exist Wieferich primes with large valuations \(A_p\), or with anomalously small multiplicative orders \(d_p \ll \log p\), the density \(\sigma_{\mathcal{W}}\) could be strictly positive.
\end{remark}

\begin{remark}[The Mersenne reduction]\label{rem:wieferich-remains}
Non-Wieferich mechanisms account for quality margins up to \(\log m / (m\log 2)\) (Theorem~\ref{thm:extremal-lcm}). A uniform excess at a fixed exponent \(\varepsilon>0\) would require an infinite, exponentially dense family of Wieferich primes, and by Theorem~\ref{thm:order-core-descent} the exponents in any such family may be taken to be order-saturated.
\end{remark}

\section{Exceptional sets and current benchmarks}\label{sec:benchmarks}

\begin{definition}[Exceptional counting quantity]\label{def:exceptional-count}
For \(0<\lambda<1\) and \(X\geq1\), let \(N_{\lambda}(X)\) be the number of pairwise coprime positive triples \((a,b,c)\in[1,X]^3\) satisfying
\[
  a+b=c
\]
and
\[
  \rad(abc)<c^\lambda.
\]
Let \(N_{\lambda}^{\mathrm{par}}(X)\) count only the parity-class triples among them.
\end{definition}

\begin{theorem}[Classical exceptional-set bound]\label{thm:classical-exceptional}
For every fixed \(0<\lambda<1\) and every \(\eta>0\),
\[
  N_{\lambda}(X)
  \ll_{\lambda,\eta}
  X^{2\lambda/3+\eta}.
\]
\end{theorem}

\begin{proof}
This follows from the radical-counting estimate of de Bruijn and the observation that at least one of \(\rad(ab)\), \(\rad(ac)\), and \(\rad(bc)\) is at most \(\rad(abc)^{2/3}\).
See~\cite{deBruijn1962,Lichtman2025}.
\end{proof}

\begin{theorem}[Recent power saving]\label{thm:recent-exceptional}
Let \(0<\lambda<1\) be fixed.
For every \(\eta>0\),
\[
  N_{\lambda}(X)
  \ll_{\lambda,\eta}
  X^{3/5+\eta}.
\]
The same bound holds for
\[
  N_{\lambda}^{\mathrm{par}}(X)
  \ll_{\lambda,\eta}
  X^{3/5+\eta}.
\]
\end{theorem}

\begin{proof}
The first assertion is Theorem~1.3 of Bernert, Browning, Lichtman, and Ter\"av\"ainen~\cite{BernertEtAl2026}, which gives \(N_{\lambda}(X)\ll_{\lambda,\eta}X^{0.6+\eta}=X^{3/5+\eta}\) for every fixed \(0<\lambda<1\).
Their Theorem~1.2 gives \(N_{\lambda}(X)\ll_{\lambda,\eta}X^{(23\lambda+3)/40+\eta}\) on the wider range \(0<\lambda\leq2\).
The bound improves on Theorem~\ref{thm:classical-exceptional} precisely for \(\lambda>9/10\).
The second assertion follows because the parity-class triples form a subset of the triples counted by \(N_{\lambda}(X)\).
\end{proof}

\begin{remark}[Pointwise constraints and counting bounds]\label{rem:exceptional-relation}
Theorem~\ref{thm:recent-exceptional} bounds the spatial density of exceptional triples across large boxes \([1,X]^3\) by analytic and counting means, whereas the defect identities of Sections~\ref{sec:defects} and~\ref{sec:boundary-free} constrain each transgressive triple taken singly.
The two kinds of statement meet at the prime support.

Theorem~\ref{thm:total-concentration} produces a prime \(p\mid abc\) with
\[
  v_p(N_p)-1
  \geq
  \frac{
    \log(2K-1)+\theta_{\varepsilon}\log(2K)
  }{
    \omega_{\tot}(\mathcal{T})\log P_{\tot}(\mathcal{T})
  }
\]
so that a triple can keep every multiplicity small only by spreading its defect over a support with \(\omega(abc)\) large.
The counting estimates draw their strength from that same regime, a radical small against \(c\) being most easily arranged by numerous primes each entering to a modest power, which is the configuration the smooth-number estimates of de Bruijn behind Theorem~\ref{thm:classical-exceptional} are designed to count.
A triple whose support stays bounded falls outside that regime, and Corollary~\ref{cor:total-floor} then forces upon it a multiplicity of at least \((1+\theta_{\varepsilon})/(\delta_0k)\), independently of the height.

An infinite transgressive family could take either of two forms, and neither has been excluded.
Along such a family \(\omega(abc)\) may grow, in which case the family lies inside a set the spatial estimates show to be sparse.
Or the support may stay bounded, and then the individual multiplicities diverge at the rates of Theorem~\ref{thm:counterexample-portrait}.
The counting theorems bound the exceptions in aggregate, whereas the defect identities constrain each exception individually.
\end{remark}

\begin{remark}[Lower-bound benchmarks]\label{rem:lower-bound-examples}
The Mersenne family of Theorem~\ref{thm:mersenne-hits} is explicit and lies within the parity class, but its excess \(q-1\) decays to zero as \(m\to\infty\). Among unrestricted triples, Bright proved the existence of infinitely many coprime triples satisfying
\[
  \log\left(\frac{c}{\rad(abc)}\right)
  >
  6.563\sqrt{\frac{\log c}{\log\log c}}.
\]
Earlier subexponential bounds of the form \(c > \rad(abc)\exp(C\sqrt{\log c/\log\log c})\), with smaller constants \(C\), were established by Stewart and Tijdeman and by van Frankenhuysen~\cite{StewartTijdeman1986,vanFrankenhuysen2000}. While these lower bounds grow faster than any power of \(\log c\), they remain \(o(\log c)\) and do not produce a fixed positive margin in quality.
\end{remark}

\section{Explicit parity descent curves}\label{sec:descent}

The square extraction of Subsection~\ref{subsec:square-extraction} yields two elliptic constructions, and they retain different parts of the arithmetic.
The first is an oriented Frey curve, whose minimal discriminant retains the whole multiplicity discarded by the radical.
The second is the intersection of two quadrics formed from the squarefree kernels.
Its Jacobian depends only upon the square class of \(\Gamma_{\mathcal{T}}\), so that every square divisor of that parameter is lost, while the point on the intersection retains the large square divisors forced by transgression.
Remark~\ref{rem:descent-jacobian-information} returns to this division.

\subsection{The oriented parity Frey curve}
\label{subsec:descent-frey}

\begin{definition}[Oriented parity Frey curve]
\label{def:descent-frey}
Let \(\mathcal{T}=(a,b,c)\) be a parity-class triple.
When \(4\mid c\), let \(\alpha\) be the unique member of \(\{a,b\}\) satisfying
\[
  \alpha\equiv1\pmod4
\]
and let \(\beta\) be the other member.
When \(4\nmid c\), retain the given ordering and put
\[
  (\alpha,\beta)=(a,b).
\]
Thus
\[
  \alpha+\beta=c
  \qquad\text{and}\qquad
  \alpha\beta=ab.
\]
Define the \emph{oriented parity Frey curve} by
\[
  E_{\mathcal{T}}:
  \quad
  y^2=x(x-\alpha)(x+\beta).
\]
\end{definition}

The orientation has no effect at odd primes.
Its role is confined to the integral model at \(2\).

\begin{proposition}[Invariants of the oriented Frey curve]
\label{prop:descent-frey-invariants}
The model of Definition~\ref{def:descent-frey} has invariants
\[
  c_4
  =
  16\bigl(\alpha^2+\alpha\beta+\beta^2\bigr)
\]
and
\[
  c_6
  =
  -32(\beta-\alpha)\bigl(2(\beta-\alpha)^2+9\alpha\beta\bigr).
\]
Its discriminant is
\[
  \Delta_0(E_{\mathcal{T}})
  =
  16(\alpha\beta c)^2
  =
  16(abc)^2.
\]
The curve has full rational \(2\)-torsion, with nonzero points
\[
  (0,0)
  \qquad
  (\alpha,0)
  \qquad
  (-\beta,0).
\]
\end{proposition}

\begin{proof}
Expansion gives
\[
  y^2=x^3+(\beta-\alpha)x^2-\alpha\beta x
\]
so that
\[
  a_1=a_3=a_6=0
  \qquad
  a_2=\beta-\alpha
  \qquad
  a_4=-\alpha\beta.
\]
The standard polynomial expressions for \(c_4\) and \(c_6\) give the first two identities.
The roots of the cubic are \(0\), \(\alpha\), and \(-\beta\), with pairwise differences \(\alpha\), \(\beta\), and \(c=\alpha+\beta\).
The discriminant identity follows, and the three roots give the displayed points of order two.
\end{proof}

\begin{proposition}[Minimal discriminant and conductor]
\label{prop:descent-minimal}
Let
\[
  n=v_2(c)
\]
and write
\[
  f_2(\mathcal{T})
  \coloneqq
  v_2(N_{E_{\mathcal{T}}}).
\]
Then the following assertions hold:
\begin{enumerate}[label=\textup{(\arabic*)}]
\item At every odd prime \(p\mid abc\), the model of Definition~\ref{def:descent-frey} is minimal and has multiplicative reduction of type \(I_{2v_p(abc)}\), so that
\[
  v_p(N_{E_{\mathcal{T}}})=1.
\]
\item If \(1\leq n\leq3\), the displayed model is minimal at \(2\) and has additive reduction there, with
\[
  \Delta_{\min}(E_{\mathcal{T}})
  =
  16(abc)^2
  \qquad\text{and}\qquad
  2\leq f_2(\mathcal{T})\leq8.
\]
\item If \(n\geq4\), the change of variables
\[
  x=4X+\alpha
  \qquad
  y=8Y+4X
\]
gives the integral model
\[
  Y^2+XY
  =
  X^3
  +\frac{\alpha+c-1}{4}X^2
  +\frac{\alpha c}{16}X
\]
which is minimal and has discriminant
\[
  \Delta_{\min}(E_{\mathcal{T}})
  =
  \frac{(abc)^2}{2^8}.
\]
\item If \(n=4\), the reduction at \(2\) is good and \(f_2(\mathcal{T})=0\).
If \(n\geq5\), it is multiplicative and \(f_2(\mathcal{T})=1\).
\end{enumerate}
Consequently,
\[
  N_{E_{\mathcal{T}}}
  =
  2^{f_2(\mathcal{T})-1}\rad(abc)
\]
and
\[
  \tfrac12\rad(abc)
  \leq
  N_{E_{\mathcal{T}}}
  \leq
  2^7\rad(abc).
\]
\end{proposition}

\begin{proof}
Let \(p\) be an odd prime dividing \(abc\).
Pairwise coprimality shows that \(p\) divides exactly one of \(\alpha\), \(\beta\), and \(c\).
If \(p\mid\alpha\), then
\[
  \alpha^2+\alpha\beta+\beta^2
  \equiv
  \beta^2
  \not\equiv0
  \pmod p
\]
and the cases \(p\mid\beta\) and \(p\mid c\) are the same.
Hence \(v_p(c_4)=0\), the model is minimal, its reduction is multiplicative, and
\[
  v_p(\Delta_{\min})=2v_p(abc).
\]
This proves the first assertion.

Since \(\alpha\) and \(\beta\) are odd, the quantity \(\alpha^2+\alpha\beta+\beta^2\) is odd, whence
\[
  v_2(c_4)=4
  \qquad\text{and}\qquad
  v_2(\Delta_0)=4+2n.
\]
For \(1\leq n\leq3\) this last valuation is smaller than \(12\).
An integral change lowering the discriminant would lower its valuation by a positive multiple of \(12\), which is impossible, so the model is minimal.
Neither its discriminant nor \(c_4\) is a unit, so the reduction is additive, and the finite branches of Tate's algorithm give
\[
  2\leq f_2(\mathcal{T})\leq8
\]
as tabulated, for instance, in~\cite{Cohen1993}.

Suppose now that \(n\geq4\).
Then \(16\mid c\), and the orientation of Definition~\ref{def:descent-frey} gives \(\alpha\equiv1\pmod4\).
Substitution of
\[
  x=4X+\alpha
  \qquad
  y=8Y+4X
\]
into the defining equation gives
\[
  (8Y+4X)^2
  =
  (4X+\alpha)(4X)(4X+c)
\]
and division by \(64\) gives the displayed integral model.
Its discriminant is the original discriminant divided by \(2^{12}\), that is,
\[
  \Delta
  =
  \frac{16(abc)^2}{2^{12}}
  =
  \frac{(abc)^2}{2^8}
\]
and its \(c_4\)-invariant is the odd quantity \(\alpha^2+\alpha\beta+\beta^2\).
The model is therefore minimal at \(2\).
When \(n=4\) its discriminant is a \(2\)-adic unit, giving good reduction.
When \(n\geq5\) the discriminant has positive valuation while \(c_4\) remains a unit, giving multiplicative reduction.
The conductor formula follows from the odd-prime calculation together with the single occurrence of \(2\) in \(\rad(abc)\).
\end{proof}

\begin{definition}[Szpiro quotient]
\label{def:descent-szpiro}
For an elliptic curve \(E/\Q\) with minimal discriminant \(\Delta_{\min}(E)\) and conductor \(N_E>1\), define
\[
  \operatorname{Szp}(E)
  \coloneqq
  \frac{\log\lvert\Delta_{\min}(E)\rvert}{\log N_E}.
\]
\end{definition}

\begin{corollary}[Total defect in the Szpiro quotient]
\label{cor:descent-szpiro}
Put
\[
  \mathcal{R}(\mathcal{T})
  \coloneqq
  \rad(abc)
\]
and define
\[
  \kappa_2(\mathcal{T})
  \coloneqq
  \begin{cases}
    4 & 1\leq v_2(c)\leq3 \\[0.3em]
    -8 & v_2(c)\geq4.
  \end{cases}
\]
Then
\[
  \log\lvert\Delta_{\min}(E_{\mathcal{T}})\rvert
  =
  2\delta_{\tot}(\mathcal{T})
  +2\log\mathcal{R}(\mathcal{T})
  +\kappa_2(\mathcal{T})\log2
\]
and
\[
  \log N_{E_{\mathcal{T}}}
  =
  \log\mathcal{R}(\mathcal{T})
  +\bigl(f_2(\mathcal{T})-1\bigr)\log2
\]
so that
\[
  \operatorname{Szp}(E_{\mathcal{T}})
  =
  \frac{
    2\delta_{\tot}(\mathcal{T})
    +2\log\mathcal{R}(\mathcal{T})
    +\kappa_2(\mathcal{T})\log2
  }{
    \log\mathcal{R}(\mathcal{T})
    +\bigl(f_2(\mathcal{T})-1\bigr)\log2
  }.
\]
If \(E_\varepsilon(\mathcal{T})\geq0\), then
\[
  \delta_{\tot}(\mathcal{T})
  \geq
  \log(ab)+\varepsilon\log\mathcal{R}(\mathcal{T})
\]
and therefore
\[
  \operatorname{Szp}(E_{\mathcal{T}})
  \geq
  \frac{
    2(1+\varepsilon)\log\mathcal{R}(\mathcal{T})
    +2\log(ab)
    +\kappa_2(\mathcal{T})\log2
  }{
    \log\mathcal{R}(\mathcal{T})
    +\bigl(f_2(\mathcal{T})-1\bigr)\log2
  }.
\]
Along a transgressive sequence with \(\mathcal{R}(\mathcal{T})\to\infty\), this gives
\[
  \operatorname{Szp}(E_{\mathcal{T}})
  \geq
  2+2\varepsilon
  +O\left(\frac{1}{\log\mathcal{R}(\mathcal{T})}\right)
\]
with an absolute implied constant.
\end{corollary}

\begin{proof}
Lemma~\ref{lem:total-defect-identity} gives
\[
  \log(abc)
  =
  \delta_{\tot}(\mathcal{T})+\log\mathcal{R}(\mathcal{T}).
\]
Substitution into the two cases of Proposition~\ref{prop:descent-minimal} gives the discriminant identity, and the conductor identity is the final assertion of that proposition.
If the triple is transgressive, then
\[
  \log c\geq(1+\varepsilon)\log\mathcal{R}(\mathcal{T})
\]
whence
\[
\begin{aligned}
  \delta_{\tot}(\mathcal{T})
  &=
  \log(ab)+\log c-\log\mathcal{R}(\mathcal{T})
  \\
  &\geq
  \log(ab)+\varepsilon\log\mathcal{R}(\mathcal{T}).
\end{aligned}
\]
The remaining assertions follow by substitution, the \(2\)-adic terms being bounded independently of the triple.
\end{proof}

\begin{corollary}[The Mersenne Frey curve]
\label{cor:descent-mersenne}
Let \(m\geq5\).
The oriented Frey curve attached to \(\mathcal{V}_m\) has conductor
\[
  N_{E_{\mathcal{V}_m}}
  =
  2\rad(2^m-1)
\]
and minimal discriminant
\[
  \Delta_{\min}(E_{\mathcal{V}_m})
  =
  2^{2m-8}(2^m-1)^2.
\]
Its Szpiro quotient satisfies the exact identity
\[
  \operatorname{Szp}(E_{\mathcal{V}_m})
  =
  2
  +\frac{
    2(m-1)\log2+2\Omega_m+2\log G_m-8\log2
  }{
    \log\bigl(2\rad(2^m-1)\bigr)
  }.
\]
\end{corollary}

\begin{proof}
For \(\mathcal{V}_m\) the oriented summands are
\[
  \alpha=1
  \qquad
  \beta=2^m-1.
\]
Since \(m\geq5\), Proposition~\ref{prop:descent-minimal} gives the conductor and the minimal discriminant.
Moreover
\[
  \delta_{\tot}(\mathcal{V}_m)
  =
  (m-1)\log2+\Delta_m
\]
and Theorem~\ref{thm:exact-defect-law} gives \(\Delta_m=\Omega_m+\log G_m\).
Substitution into Corollary~\ref{cor:descent-szpiro} proves the identity.
\end{proof}

\subsection{The genus-one intersection}
\label{subsec:descent-curve}

Retain the square extraction
\[
  a=d_au_a^2
  \qquad
  b=d_bu_b^2
\]
where \(d_a\) and \(d_b\) are positive odd squarefree integers.
The triple \((1,1,2)\) is the sole parity-class triple with \(M=0\), and it is excluded from the construction below.

\begin{definition}[Parity descent curve]
\label{def:descent-curve}
Let \(\mathcal{T}\neq(1,1,2)\) be a parity-class triple.
Define \(\mathscr{D}_{\mathcal{T}}\subset\mathbb{P}^3_{(U:V:W:T)}\) by
\[
  \mathscr{D}_{\mathcal{T}}:
  \quad
  \begin{cases}
    d_aU^2+d_bV^2=2KW^2 \\[0.3em]
    d_bV^2-d_aU^2=2MT^2.
  \end{cases}
\]
Put
\[
  \Gamma_{\mathcal{T}}
  \coloneqq
  d_ad_bKM.
\]
\end{definition}

\begin{lemma}[The pencil of quadrics]
\label{lem:descent-pencil}
Let \(Q_0\) and \(Q_1\) be the two diagonal quadrics defining \(\mathscr{D}_{\mathcal{T}}\).
Then
\[
  \det(Q_0+\lambda Q_1)
  =
  4\Gamma_{\mathcal{T}}\lambda(1-\lambda^2).
\]
The roots of the determinant polynomial are
\[
  0
  \qquad
  1
  \qquad
  -1
  \qquad
  \infty
\]
and they are distinct.
Hence \(\mathscr{D}_{\mathcal{T}}\) is a smooth complete intersection of two quadrics and has genus one.
\end{lemma}

\begin{proof}
The diagonal entries of \(Q_0+\lambda Q_1\) are
\[
  d_a(1-\lambda)
  \qquad
  d_b(1+\lambda)
  \qquad
  -2K
  \qquad
  -2\lambda M
\]
and their product gives the determinant.
Since \(\mathcal{T}\neq(1,1,2)\), one has \(M\neq0\) and therefore \(\Gamma_{\mathcal{T}}\neq0\), so the determinant has four distinct roots.
For a pencil of two quadrics in \(\mathbb{P}^3\), this separability is equivalent to smoothness of the intersection.
A smooth complete intersection of two quadrics in \(\mathbb{P}^3\) has degree four and trivial canonical sheaf, hence genus one.
\end{proof}

\begin{proposition}[The distinguished rational point]
\label{prop:descent-basic}
The point
\[
  P_{\mathcal{T}}
  \coloneqq
  (u_a:u_b:1:1)
\]
lies on \(\mathscr{D}_{\mathcal{T}}(\Q)\).
The defining equations may equivalently be written
\[
  d_aU^2=KW^2-MT^2
  \qquad\text{and}\qquad
  d_bV^2=KW^2+MT^2
\]
and at \(P_{\mathcal{T}}\) these reduce to
\[
  d_au_a^2=K-M=a
  \qquad\text{and}\qquad
  d_bu_b^2=K+M=b.
\]
\end{proposition}

\begin{proof}
The identities \(a=K-M\) and \(b=K+M\) give
\[
  d_au_a^2+d_bu_b^2=a+b=2K
  \qquad\text{and}\qquad
  d_bu_b^2-d_au_a^2=b-a=2M.
\]
Thus \(P_{\mathcal{T}}\) lies on both quadrics, and adding and subtracting the two defining equations gives the equivalent diagonal form.
\end{proof}

\begin{corollary}[Witnessed local solubility]
\label{cor:descent-witnessed-local}
For every prime \(p\), the point \(P_{\mathcal{T}}\) gives an explicitly displayed point in \(\mathscr{D}_{\mathcal{T}}(\Q_p)\), and it gives a real point.
The local solubility of the descent curve attached to an actual parity-class triple therefore requires no local--global argument.
\end{corollary}

\subsection{The Jacobian and its square-class reduction}
\label{subsec:descent-jacobian}

\begin{theorem}[Jacobian of the parity descent curve]
\label{thm:descent-jacobian}
The Jacobian of \(\mathscr{D}_{\mathcal{T}}\) has the model
\[
  J_{\mathcal{T}}:
  \quad
  y^2=x^3-16\Gamma_{\mathcal{T}}^2x.
\]
Let
\[
  h_{\mathcal{T}}
  \coloneqq
  \sq\bigl(\lvert\Gamma_{\mathcal{T}}\rvert\bigr)
\]
and define the positive squarefree integer
\[
  \gamma_{\mathcal{T}}
  \coloneqq
  \frac{\lvert\Gamma_{\mathcal{T}}\rvert}{h_{\mathcal{T}}^2}.
\]
Then \(J_{\mathcal{T}}\) is isomorphic over \(\Q\) to
\[
  C_{\gamma_{\mathcal{T}}}:
  \quad
  Y^2=X^3-\gamma_{\mathcal{T}}^2X
\]
so that \(j(J_{\mathcal{T}})=1728\).
The curve has full rational \(2\)-torsion, and over \(\overline{\Q}\) it has complex multiplication by \(\Z[i]\).
\end{theorem}

\begin{proof}
The determinant construction for a smooth intersection of two quadrics gives the Jacobian as the smooth projective model of
\[
  Z^2=4\Gamma_{\mathcal{T}}\lambda(1-\lambda^2).
\]
Writing \(c_{\mathcal{T}}=-4\Gamma_{\mathcal{T}}\) this becomes
\[
  Z^2=c_{\mathcal{T}}(\lambda^3-\lambda)
\]
and the change
\[
  x=c_{\mathcal{T}}\lambda
  \qquad
  y=c_{\mathcal{T}}Z
\]
gives
\[
  y^2=x^3-c_{\mathcal{T}}^2x=x^3-16\Gamma_{\mathcal{T}}^2x.
\]
Now write \(\lvert\Gamma_{\mathcal{T}}\rvert=h_{\mathcal{T}}^2\gamma_{\mathcal{T}}\).
The change
\[
  x=4h_{\mathcal{T}}^2X
  \qquad
  y=8h_{\mathcal{T}}^3Y
\]
gives
\[
  Y^2=X^3-\gamma_{\mathcal{T}}^2X
\]
whose nonzero rational points of order two are
\[
  (0,0)
  \qquad
  (\gamma_{\mathcal{T}},0)
  \qquad
  (-\gamma_{\mathcal{T}},0).
\]
The displayed equation has \(j\)-invariant \(1728\), and the map
\[
  (X,Y)\longmapsto(-X,iY)
\]
over \(\Q(i)\) gives the stated complex multiplication.
\end{proof}

\begin{proposition}[Conductor of the descent Jacobian]
\label{prop:descent-jacobian-conductor}
Let \(\gamma\) be a positive squarefree integer and put
\[
  C_\gamma:
  \quad
  Y^2=X^3-\gamma^2X.
\]
Then the displayed equation is globally minimal, with
\[
  \Delta_{\min}(C_\gamma)=2^6\gamma^6
\]
and conductor
\[
  N_{C_\gamma}
  =
  \begin{cases}
    2^5\gamma^2 & \gamma\ \text{odd} \\[0.4em]
    2^4\gamma^2 & \gamma\ \text{even}.
  \end{cases}
\]
At every odd prime \(p\mid\gamma\) one has \(v_p(N_{C_\gamma})=2\), while at \(2\),
\[
  v_2(N_{C_\gamma})
  =
  \begin{cases}
    5 & \gamma\ \text{odd} \\[0.3em]
    6 & \gamma\ \text{even}.
  \end{cases}
\]
\end{proposition}

\begin{proof}
For \(Y^2=X^3-\gamma^2X\) the invariants are
\[
  c_4=48\gamma^2
  \qquad
  c_6=0
  \qquad
  \Delta=2^6\gamma^6.
\]
If an odd prime \(p\) does not divide \(\gamma\), the discriminant is a \(p\)-adic unit and the reduction is good.
Let \(p\) be an odd prime dividing \(\gamma\).
Since \(\gamma\) is squarefree, \(v_p(\gamma^2)=2\), the model is minimal at \(p\), and the finite Tate algorithm gives conductor exponent two.
At \(2\) the same finite algorithm has two branches.
When \(\gamma\) is odd, the minimal discriminant has valuation six and the conductor exponent is five.
When \(\gamma\) is even and squarefree, the minimal discriminant has valuation twelve and the conductor exponent is six.
These calculations are instances of Tate's algorithm as presented in~\cite{Cohen1993}.
Multiplication of the local conductor factors gives the displayed formula.
\end{proof}

\begin{corollary}[Uniform Szpiro ceiling for the descent Jacobian]
\label{cor:descent-jacobian-szpiro}
For every parity-class triple \(\mathcal{T}\neq(1,1,2)\) one has
\[
  \operatorname{Szp}(J_{\mathcal{T}})<3.
\]
More precisely, writing \(\gamma=\gamma_{\mathcal{T}}\),
\[
  \operatorname{Szp}(J_{\mathcal{T}})
  =
  \begin{cases}
    \displaystyle
    3-\frac{9\log2}{2\log\gamma+5\log2}
    & \gamma\ \text{odd} \\[1.4em]
    \displaystyle
    3-\frac{6\log2}{2\log\gamma+4\log2}
    & \gamma\ \text{even}.
  \end{cases}
\]
\end{corollary}

\begin{proof}
Theorem~\ref{thm:descent-jacobian} and Proposition~\ref{prop:descent-jacobian-conductor} give
\[
  \log\lvert\Delta_{\min}(J_{\mathcal{T}})\rvert
  =
  6\log\gamma+6\log2.
\]
If \(\gamma\) is odd, then
\[
  \log N_{J_{\mathcal{T}}}=2\log\gamma+5\log2
\]
and if \(\gamma\) is even, then
\[
  \log N_{J_{\mathcal{T}}}=2\log\gamma+4\log2.
\]
Division gives the two formulas.
\end{proof}

\begin{remark}[What the Jacobian retains]
\label{rem:descent-jacobian-information}
The Frey curve and the descent Jacobian retain different parts of the defect.
Corollary~\ref{cor:descent-szpiro} places the total defect in the minimal discriminant of the Frey curve, whereas Theorem~\ref{thm:descent-jacobian} replaces
\[
  \Gamma_{\mathcal{T}}=d_ad_bKM
\]
by its squarefree part \(\gamma_{\mathcal{T}}\), so that every square divisor of \(\Gamma_{\mathcal{T}}\) disappears.
The uniform ceiling of Corollary~\ref{cor:descent-jacobian-szpiro} measures that loss.
The defect forced by transgression is therefore better studied through the point \(P_{\mathcal{T}}\) and the kernels \(d_a,d_b\) than through the Szpiro quotient of \(J_{\mathcal{T}}\).
Against this the Jacobian is rigid, the whole family lying within the single square class of curves with \(j=1728\) and complex multiplication by \(\Z[i]\).
\end{remark}

\subsection{Height and kernel compression}
\label{subsec:descent-compression}

\begin{definition}[Descent heights]
\label{def:descent-heights}
For the distinguished point \(P_{\mathcal{T}}=(u_a:u_b:1:1)\), define
\[
  H_{\times}(P_{\mathcal{T}})
  \coloneqq
  u_au_b
  \qquad\text{and}\qquad
  H_{\infty}(P_{\mathcal{T}})
  \coloneqq
  \max\{u_a,u_b\}.
\]
\end{definition}

\begin{proposition}[Exact height and kernel identity]
\label{prop:descent-height}
For every parity-class triple,
\[
  d_ad_bH_{\times}(P_{\mathcal{T}})^2
  =
  ab
  =
  s(\mathcal{T})\bigl(2K-s(\mathcal{T})\bigr)
\]
and
\[
  H_{\infty}(P_{\mathcal{T}})^2
  \geq
  H_{\times}(P_{\mathcal{T}}).
\]
If \(K>1\) and \(E_\varepsilon(\mathcal{T})\geq0\), then
\[
  H_{\times}(P_{\mathcal{T}})
  \geq
  \bigl(s(\mathcal{T})R_{\mathcal K}(\mathcal{T})\bigr)^{1/2}
  (2K)^{\theta_\varepsilon/2}
\]
and
\[
  H_{\infty}(P_{\mathcal{T}})
  \geq
  \bigl(s(\mathcal{T})R_{\mathcal K}(\mathcal{T})\bigr)^{1/4}
  (2K)^{\theta_\varepsilon/4}.
\]
\end{proposition}

\begin{proof}
The square extraction gives \(ab=d_ad_b(u_au_b)^2\), and the identity
\[
  ab=s(\mathcal{T})\bigl(2K-s(\mathcal{T})\bigr)
\]
is Proposition~\ref{prop:boundary-identity}.
The elementary inequality \(\max\{u_a,u_b\}^2\geq u_au_b\) gives the second assertion.
The transgressive lower bound for \(H_{\times}\) is Theorem~\ref{thm:large-square}, and taking its square root gives the bound for \(H_{\infty}\).
\end{proof}

\begin{theorem}[Compression of the squarefree kernels]
\label{thm:descent-kernel-compression}
Suppose that \(K>1\) and \(E_\varepsilon(\mathcal{T})\geq0\), and put \(s=s(\mathcal{T})\).
Then
\[
  d_ad_b
  \leq
  \frac{2K-s}{R_{\mathcal K}(\mathcal{T})(2K)^{\theta_\varepsilon}}.
\]
Furthermore
\[
  \lvert\Gamma_{\mathcal{T}}\rvert
  =
  d_ad_bK(K-s)
\]
and therefore both \(\lvert\Gamma_{\mathcal{T}}\rvert\) and the squarefree twist parameter \(\gamma_{\mathcal{T}}\) are at most
\[
  \frac{K(K-s)(2K-s)}{R_{\mathcal K}(\mathcal{T})(2K)^{\theta_\varepsilon}}.
\]
\end{theorem}

\begin{proof}
Proposition~\ref{prop:descent-height} gives
\[
  d_ad_b
  =
  \frac{s(2K-s)}{H_{\times}(P_{\mathcal{T}})^2}
\]
and the transgressive height floor gives
\[
  H_{\times}(P_{\mathcal{T}})^2
  \geq
  sR_{\mathcal K}(\mathcal{T})(2K)^{\theta_\varepsilon}.
\]
Division proves the kernel bound.
Since \(s=K-\lvert M\rvert\), one has \(\lvert M\rvert=K-s\) and hence
\[
  \lvert\Gamma_{\mathcal{T}}\rvert
  =
  d_ad_bK(K-s).
\]
The remaining bounds follow by substitution together with \(\gamma_{\mathcal{T}}\leq\lvert\Gamma_{\mathcal{T}}\rvert\).
\end{proof}

\begin{corollary}[A kernel certificate]
\label{cor:descent-kernel-certificate}
Suppose that \(K>1\).
If
\[
  R_{\mathcal K}(\mathcal{T})(2K)^{\theta_\varepsilon}
  >
  2K-s(\mathcal{T})
\]
then \(E_\varepsilon(\mathcal{T})<0\).
\end{corollary}

\begin{proof}
Were the triple transgressive, Theorem~\ref{thm:descent-kernel-compression} would give \(d_ad_b<1\), which is impossible for positive integers.
\end{proof}

\begin{corollary}[Conductor bound for the descent Jacobian]
\label{cor:descent-jacobian-conductor-bound}
Suppose that \(K>1\) and \(E_\varepsilon(\mathcal{T})\geq0\), and put \(s=s(\mathcal{T})\).
Then
\[
  N_{J_{\mathcal{T}}}
  \leq
  32
  \left(
    \frac{K(K-s)(2K-s)}{R_{\mathcal K}(\mathcal{T})(2K)^{\theta_\varepsilon}}
  \right)^2.
\]
\end{corollary}

\begin{proof}
Proposition~\ref{prop:descent-jacobian-conductor} gives \(N_{J_{\mathcal{T}}}\leq32\gamma_{\mathcal{T}}^2\).
Apply Theorem~\ref{thm:descent-kernel-compression}.
\end{proof}

Transgression forces the point upward and the squarefree kernels downward, and this is the part of the descent that remains sensitive to it.
The Jacobian, having passed to the square class, no longer reflects that multiplicity.

\subsection{Separate local conditions at \(2\)}
\label{subsec:descent-local}

\begin{definition}[Residual conics]
\label{def:descent-residual-conics}
Define the two residual conics
\[
  \mathscr{C}_{K}:
  \quad
  d_aU^2+d_bV^2=2KW^2
\]
and
\[
  \mathscr{C}_{M}:
  \quad
  d_bV^2-d_aU^2=2MT^2.
\]
Their intersection in \(\mathbb{P}^3\) is \(\mathscr{D}_{\mathcal{T}}\).
\end{definition}

\begin{proposition}[Unit criteria for the two residual conics]
\label{prop:descent-conic-local}
Let \(d_a\) and \(d_b\) be odd.
The conic \(\mathscr{C}_{K}\) has a point with \(U,V,W\in\Z_2^\times\) precisely when
\[
  d_a+d_b\equiv2K\pmod8
\]
and the conic \(\mathscr{C}_{M}\) has a point with \(U,V,T\in\Z_2^\times\) precisely when
\[
  d_b-d_a\equiv2M\pmod8.
\]
\end{proposition}

\begin{proof}
Every square in \(\Z_2^\times\) is congruent to \(1\) modulo \(8\), so the two congruences are necessary.
Suppose that \(d_a+d_b\equiv2K\pmod8\).
Take \(V=W=1\) and put
\[
  \eta=\frac{2K-d_b}{d_a}.
\]
The numerator and denominator are odd, and the congruence gives \(\eta\equiv1\pmod8\).
Every element of \(1+8\Z_2\) is the square of a \(2\)-adic unit, as in the proof of Proposition~\ref{prop:two-adic-complete}, so \(U^2=\eta\) has a solution in \(\Z_2^\times\).
Suppose now that \(d_b-d_a\equiv2M\pmod8\).
Take \(U=T=1\) and put
\[
  \xi=\frac{2M+d_a}{d_b}.
\]
Again \(\xi\equiv1\pmod8\), so there is a unit \(V\in\Z_2^\times\) satisfying \(V^2=\xi\).
\end{proof}

\begin{corollary}[Congruences supplied by square extraction]
\label{cor:descent-conic-congruences}
For an actual parity-class triple,
\[
  d_a\equiv a\pmod8
  \qquad\text{and}\qquad
  d_b\equiv b\pmod8
\]
whence
\[
  d_a+d_b\equiv2K\pmod8
  \qquad\text{and}\qquad
  d_b-d_a\equiv2M\pmod8.
\]
Each residual conic therefore has a \(2\)-adic unit point.
\end{corollary}

\begin{proof}
The integers \(u_a\) and \(u_b\) are odd, so \(u_a^2\equiv u_b^2\equiv1\pmod8\).
The identities \(a=d_au_a^2\) and \(b=d_bu_b^2\) give the first two congruences, and addition and subtraction give the remaining ones.
\end{proof}

\begin{proposition}[The exact condition on the distinguished affine fiber]
\label{prop:descent-affine-two-adic}
Let
\[
  \mathscr{D}_{\mathcal{T}}^\circ
  \coloneqq
  \mathscr{D}_{\mathcal{T}}\cap\{W=T=1\}.
\]
Then \(\mathscr{D}_{\mathcal{T}}^\circ(\Z_2^\times\times\Z_2^\times)\neq\varnothing\) precisely when
\[
  \frac{a}{d_a}\in1+8\Z_2
  \qquad\text{and}\qquad
  \frac{b}{d_b}\in1+8\Z_2
\]
that is, precisely when
\[
  a\equiv d_a\pmod8
  \qquad\text{and}\qquad
  b\equiv d_b\pmod8.
\]
For the squarefree kernels extracted from a parity-class triple, the pair \((u_a,u_b)\) is an integral solution on this affine fiber.
\end{proposition}

\begin{proof}
On the affine fiber \(W=T=1\) the equations are
\[
  d_aU^2+d_bV^2=2K
  \qquad\text{and}\qquad
  d_bV^2-d_aU^2=2M
\]
and adding and subtracting them gives
\[
  d_aU^2=K-M=a
  \qquad\text{and}\qquad
  d_bV^2=K+M=b.
\]
Hence a unit solution exists precisely when \(a/d_a\) and \(b/d_b\) are squares in \(\Z_2^\times\), and a \(2\)-adic unit is a square precisely when it is congruent to \(1\) modulo \(8\).
The extracted integers \(u_a,u_b\) give the final assertion.
\end{proof}

\begin{remark}[Compatibility of the two conics]
\label{rem:descent-local-compatibility}
Proposition~\ref{prop:descent-conic-local} gives the complete unit condition for each residual conic separately, while their simultaneous compatibility requires the same values of \(U\) and \(V\) in both equations.
On the distinguished affine fiber that compatibility is expressed by the stronger pair
\[
  a\equiv d_a\pmod8
  \qquad\text{and}\qquad
  b\equiv d_b\pmod8
\]
of Proposition~\ref{prop:descent-affine-two-adic}, of which the two summed congruences are the projections onto the two conics.
The summed congruences do not by themselves describe every \(2\)-adic point of the full projective intersection.
\end{remark}

\subsection{Arithmetic content of the descent}
\label{subsec:descent-content}

\begin{proposition}[Explicit descent packet]
\label{prop:descent-explicit-packet}
For every parity-class triple \(\mathcal{T}\neq(1,1,2)\), the data
\[
  d_a
  \qquad
  d_b
  \qquad
  u_a
  \qquad
  u_b
  \qquad
  \Gamma_{\mathcal{T}}
  \qquad
  \gamma_{\mathcal{T}}
\]
are obtained by finite integer operations, and they determine the following objects and bounds:
\begin{enumerate}[label=\textup{(\arabic*)}]
\item The point \(P_{\mathcal{T}}=(u_a:u_b:1:1)\) on the genus-one curve \(\mathscr{D}_{\mathcal{T}}\).
\item The Jacobian \(J_{\mathcal{T}}\simeq C_{\gamma_{\mathcal{T}}}:Y^2=X^3-\gamma_{\mathcal{T}}^2X\).
\item The exact minimal discriminant and conductor of \(C_{\gamma_{\mathcal{T}}}\).
\item The two unit congruences for the residual conics.
\item Under transgression, the height floor of Proposition~\ref{prop:descent-height} and the kernel compression of Theorem~\ref{thm:descent-kernel-compression}.
\end{enumerate}
\end{proposition}

\begin{proof}
The squarefree kernels and square divisors are obtained from the prime factorizations of \(a\) and \(b\), and the parameters \(\Gamma_{\mathcal{T}}\) and \(\gamma_{\mathcal{T}}\) then follow by multiplication and square extraction.
Every remaining assertion is given by the displayed formulas of this section.
\end{proof}

\section{Interpretation of the boundary reduction}\label{sec:interpretation-boundary}

The equivalence between an infinite transgressive family on the Mersenne line and a positive order--defect density \(\sigma_{\mathcal{W}}\) is exact, so as a statement of difficulty it is a reformulation, though not an arbitrary one. The boundary lines are the configurations on which several independent estimates take their most permissive form, and the reduction identifies the arithmetic quantity that remains once those estimates are saturated.

On a general boundary line the linear threshold of Theorem~\ref{thm:exact-linear-threshold} contains the term
\[
  \log\Bigl(s\Bigl(2-\frac{s}{K}\Bigr)\Bigr)
\]
whose minimum is attained at \(s=1\). An interior sequence with \(s_j\to\infty\) must meet a larger threshold, the certificate acquiring a supplement of order \(\log s_j\), and for \(s_j\geq K_j^{\alpha}\) Proposition~\ref{prop:exponent-amplification} transfers the certificate to the larger parameter \(\theta_{\varepsilon}+\alpha\). The exponential boundary lines, among them \((1,2^m-1,2^m)\), are therefore the lines on which a fixed positive excess imposes the weakest arithmetic conditions.

That advantage does not produce counterexamples on its own. The defect law
\[
  \Delta_m=\Omega_m+\log G_m
\]
of Theorem~\ref{thm:exact-defect-law} separates the possible sources. The quantity \(\log G_m\) measures every non-Wieferich repetition and is at most \(\log m\). The highly composite exponents \(m_k=\operatorname{lcm}(1,\dots,k)\) come close to that bound, and the quality margin they yield decays like \(\log m/(m\log2)\). Corollary~\ref{cor:superlog-wieferich} is sharper, forcing \(\Omega_m\to\infty\) for any uniform margin \(\varepsilon>0\). Ordinary prime lifting, multiplicity, and exponent divisibility are thus absorbed by a logarithmic term, and the intrinsic growth of the Wieferich order--defect data is the one parameter that remains.

A proof of the full \(abc\) conjecture would imply the finiteness of the boundary families, hence by Theorems~\ref{thm:order-defect-density} and~\ref{thm:even-base-density-criterion}
\[
  \sigma_{\mathcal{W},g}=0
\]
for every even base \(g\geq2\). In finite form it would exclude all admissible sequences of Wieferich sets with
\[
  B_g(S_j)\geq g^{\kappa D_g(S_j)}.
\]
Theorem~\ref{thm:bounded-order-cofactor} gives the parent transgressive sequence a bounded order-cofactor, and Theorem~\ref{thm:order-core-descent} together with Lemma~\ref{lem:order-core-idempotence} allows a strictly order-saturated representative to be substituted when the condition is tested, one for which \(D_n=n\) and the order-cofactor is identically \(1\). This is the canonical form in which the lower bound should be evaluated, and it removes any ambiguity about cofactor growth in the boundary data. The consequences are necessary ones only, and they might also be reached without isolating the Wieferich primes, since a global bound on the radical could make the exponential threshold unreachable.

The counting methods of Section~\ref{sec:benchmarks} offer a contrast. They are strongest in the interior regime, where \(\omega_{\tot}(\mathcal{T})\) grows without bound and no single prime power carries a positive proportion of the defect. The boundary lines can reverse those priorities, the number of prime factors staying bounded while the defect concentrates in a few large prime powers, as in Theorem~\ref{thm:counterexample-portrait}. Analytic estimates are not designed to detect whether a congruence such as \(2^{p-1}\equiv1\pmod{p^2}\) occurs with an unusually small multiplicative order. The boundary reduction separates that algebraic question and states exactly what remains of the problem once the available logarithmic estimates have been applied.

\section{The parity class under the negation}\label{sec:negation}

\begin{definition}[Critical quality]\label{def:critical-quality}
For \(B\geq1\), define
\[
  \Sigma(B)
  \coloneqq
  \sup
  \bigl\{
    q(\mathcal{T})
    :
    \mathcal{T}\ \text{is a parity-class triple and}\ K(\mathcal{T})\geq B
  \bigr\}.
\]
The \emph{critical quality} of the parity class is
\[
  q^*
  \coloneqq
  \inf_{B\geq1}\Sigma(B).
\]
\end{definition}

Theorem~\ref{thm:mersenne-hits} shows that triples of quality greater than one occur beyond every height, and Theorem~\ref{thm:pell-infinite-quality} gives a second family of the same kind.
Hence
\[
  q^*\geq1.
\]
Neither family contradicts Conjecture~\ref{conj:parity-abc}, since along both the margin \(q-1\) tends to \(0\), and the conjecture forbids infinitude only at a fixed \(\varepsilon>0\).
A crossing of quality one is therefore not a counterexample, and Proposition~\ref{prop:critical-equivalence} places the whole question in the value of \(q^*\) rather than in the existence of such crossings.

\begin{proposition}[Critical-quality equivalence]\label{prop:critical-equivalence}
Conjecture~\ref{conj:parity-abc} holds precisely when
\[
  q^*=1.
\]
\end{proposition}

\begin{proof}
Suppose first that Conjecture~\ref{conj:parity-abc} holds.
For every \(\varepsilon>0\), the finitely many triples with \(q\geq1+\varepsilon\) lie below some height, so that \(q^*\leq1+\varepsilon\).
Since \(q^*\geq1\), passage to \(\varepsilon\to0^+\) gives \(q^*=1\).

Suppose now that \(q^*=1\).
For every \(\varepsilon>0\), there is a height \(B\) with \(\Sigma(B)<1+\varepsilon\).
Hence every triple with \(q\geq1+\varepsilon\) satisfies \(K<B\), and there are only finitely many such triples.
\end{proof}

\begin{definition}[Transgressive sequence]\label{def:transgressive-sequence}
Fix \(\varepsilon>0\).
A \emph{transgressive sequence} is a sequence of distinct parity-class triples \((\mathcal{T}_j)\) satisfying
\[
  E_{\varepsilon}(\mathcal{T}_j)\geq0
\]
for every \(j\).
\end{definition}

The existence of a transgressive sequence at one fixed exponent is precisely the failure of Conjecture~\ref{conj:parity-abc}.

\begin{theorem}[Growth along a transgressive sequence]\label{thm:counterexample-portrait}
Let \((\mathcal{T}_j)\) be a transgressive sequence at exponent \(\varepsilon\), and write \(K_j=K(\mathcal{T}_j)\).
Then \(K_j\) is unbounded, and so are the total defect \(\delta_{\tot}(\mathcal{T}_j)\), the relative quotient \(Q_{\mathrm{rel}}(\mathcal{T}_j)\), and the largest square divisor \(\max\{\sq(a_j)^2,\sq(b_j)^2\}\).
More precisely, one has
\[
  \delta_{\tot}(\mathcal{T}_j)
  \geq
  \log(2K_j-1)+\theta_{\varepsilon}\log(2K_j).
\]
One has
\[
  Q_{\mathrm{rel}}(\mathcal{T}_j)
  \geq
  s(\mathcal{T}_j)\,(2K_j)^{\theta_{\varepsilon}}.
\]
Finally, writing \(u_{a_j}=\sq(a_j)\) and \(u_{b_j}=\sq(b_j)\) for the square roots of the largest square divisors of the two summands, Theorem~\ref{thm:large-square} gives
\[
  \max\{
    \sq(a_j)^2,
    \sq(b_j)^2
  \}
  \geq
  u_{a_j}u_{b_j}
  \geq
  \bigl(s(\mathcal{T}_j)\,R_{\mathcal{K}}(\mathcal{T}_j)\bigr)^{1/2}
  (2K_j)^{\theta_{\varepsilon}/2}.
\]
Every member with \(K_j>1\) contains at least two repeated primes among \(a_j\), \(b_j\), and \(c_j\).
\end{theorem}

\begin{proof}
Each fixed height contains finitely many parity-class triples, so the heights in a sequence of distinct triples are unbounded.
The displayed bounds are Theorem~\ref{thm:total-criterion}, Corollary~\ref{cor:quotient-certificate}, and Theorem~\ref{thm:large-square}, the last in the chained form given there.
The final assertion is Theorem~\ref{thm:two-repeated-primes}.
\end{proof}

\begin{theorem}[Two regimes for transgressive sequences]\label{thm:boundary-interior-dichotomy}
Let \((\mathcal{T}_j)\) be a transgressive sequence at exponent \(\varepsilon\), and put \(s_j=s(\mathcal{T}_j)\) and \(K_j=K(\mathcal{T}_j)\).
Then \((\mathcal{T}_j)\) contains a subsequence falling into exactly one of the following two mutually exclusive regimes:
\begin{enumerate}[label=\textup{(\arabic*)}]
\item \textit{Bounded smaller summand.} There is a fixed odd integer \(s_0\) such that \(s_j=s_0\) along the subsequence.
The subsequence lies on the affine line
\[
  \bigl\{(s_0,2K-s_0,2K):K>s_0,\ \gcd(s_0,2K)=1\bigr\}
\]
with \(K_j\to\infty\), and Corollary~\ref{cor:line-criterion} gives
\[
  \log\left(\frac{2K_j-s_0}{\rad(2K_j-s_0)}\right)
  \geq
  t^{\mathcal{K}}(\mathcal{T}_j)
  +
  \theta_{\varepsilon}\log(2K_j)
  +
  \log\rad(s_0).
\]
\item \textit{Unbounded smaller summand.} One has \(s_j\to\infty\) along the subsequence, and
\[
  \delta_{\lin}(\mathcal{T}_j)
  -
  t^{\mathcal{K}}(\mathcal{T}_j)
  -
  \theta_{\varepsilon}\log(2K_j)
  \geq
  \log s_j
  \longrightarrow
  \infty.
\]
If \(s_j\geq K_j^{\alpha}\) for a fixed \(\alpha\in(0,1]\) with \(\theta_{\varepsilon}+\alpha<1\), then Proposition~\ref{prop:exponent-amplification} applies, and the subsequence satisfies the linear threshold at the amplified exponent \(\theta_{\varepsilon'}=\theta_{\varepsilon}+\alpha\) up to the additive constant \(\log2\).
\end{enumerate}
\end{theorem}

\begin{proof}
The smaller summands \(s_j\) are positive odd integers.
Suppose first that some subsequence of \((s_j)\) is bounded.
Along it the values lie in a finite set, so one value \(s_0\) recurs infinitely often, and the corresponding triples lie on the stated line.
The heights along a sequence of distinct triples are unbounded, so \(K_j\to\infty\) there, and the displayed bound is the second assertion of Corollary~\ref{cor:line-criterion}.
Suppose instead that no subsequence of \((s_j)\) is bounded.
Then \(s_j\to\infty\), and the displayed interior bound follows from Theorem~\ref{thm:exact-linear-threshold} together with \(2-s_j/K_j\geq1\).
A subsequence cannot fall under both headings, since \(s_j\) cannot be constant and tend to infinity.
The final claim is Proposition~\ref{prop:exponent-amplification}.
\end{proof}

\begin{remark}[Boundary lines and the Mersenne line]\label{rem:dichotomy-reading}
Along a subsequence where \(s_j = s_0\) is constant, the triples take the form \((s_0, 2K-s_0, 2K)\). In this case, the lower bound in Corollary~\ref{cor:line-criterion} shifts by the fixed additive constant \(\log\rad(s_0)\), and the required defect must come almost entirely from the larger summand \(2K-s_0\). The choice \(s_0=1\) with \(2K=2^m\) gives the Mersenne line, where the factorization of \(2^m-1\) can be examined explicitly. Along a subsequence where \(s_j \to \infty\), the threshold increases by \(\log s_j\), requiring a strictly larger defect than on the boundary.
\end{remark}

\begin{remark}[Structural constraints on putative counterexamples]\label{rem:negative-resolution}
Any infinite transgressive sequence must exhibit squarefull parts that grow as a positive power of the height \(c\). Such growth cannot be generated by squarefree summands, by bounded defects, or by lifting a fixed finite set of primes (as shown in Corollary~\ref{cor:finite-prime-escape}).
By Theorem~\ref{thm:boundary-interior-dichotomy}, any transgressive sequence must either remain on a fixed line \((s_0, 2K-s_0, 2K)\) or satisfy a linear defect that grows faster than the boundary-free threshold by an additive term \(\log s_j \to \infty\).

Along the Mersenne line, Theorem~\ref{thm:order-defect-density} and Theorem~\ref{thm:order-core-descent} further require that the underlying Wieferich primes satisfy \(B(S_j) \ge 2^{\kappa D(S_j)}\) along an order-saturated sequence with \(D(S_j) \to \infty\).
No known arithmetic mechanism produces repetition on that scale.
\end{remark}

\bibliographystyle{smfalpha}
\bibliography{abc_conjecture}

\end{document}